\documentclass[fleqn,11pt]{amsart}
\usepackage[utf8]{inputenc}

\usepackage{graphicx}
\usepackage{float} 
\usepackage{amsmath,amssymb}
\usepackage{appendix}
\usepackage[hidelinks]{hyperref}
\usepackage{url}
\usepackage{comment,xcolor}
\usepackage[lined,boxed,linesnumbered,ruled]{algorithm2e}
\usepackage{bm}
\usepackage{stmaryrd}
\usepackage{graphicx}
\usepackage{subcaption}

\usepackage{setspace}
\usepackage{enumitem}
\setlist{nolistsep}

\usepackage{natbib}
\usepackage{tikz}

\newcommand{\avg}[1]{\{\!\{#1\}\!\}}
\newcommand{\omgr}{\Omega_{\mathrm{ref}}}
\newcommand{\domgr}{\partial\Omega_{\mathrm{ref}}}
\newcommand{\QN}{\mathbb Q_N(\omgr)}

\newtheorem{theorem}{Theorem}[section]

\newtheorem{lemma}[theorem]{Lemma}
\newtheorem{remark}{Remark}[section]
\newtheorem{definition}{Definition}[section]
\usepackage[a4paper, margin=1in]{geometry}

\title[ES-OEDG]{
An Entropy-Stable Oscillation-Eliminating DGSEM for the Euler Equations on Curvilinear Meshes}

 \author{Jieling Yang$^1$}
 \author{Guosheng Fu$^1$}
  \address{$^1$ Department of Applied and Computational Mathematics and Statistics (ACMS), University of Notre Dame, Notre Dame, IN 46556}
 \email{jyang38@nd.edu, gfu@nd.edu}
 \thanks{This work is supported in part by NSF DMS-2410740.}

\begin{document}

\maketitle

%%%%%%%%%%%%%%%%%%%%%%%%%%%%%%%%%%%%%%%%%%%%%%%%%%%%
%%%%%%%%%%%%%%%%%%%%%%%%%%%%%%%%%%%%%%%%%%%%%%%%%%%%
%%%%%%%%%%%%%%%%%%%%%%%%%%%%%%%%%%%%%%%%%%%%%%%%%%%%
%%%%%%%%%%%%%%%%%%%%%%%%%%%%%%%%%%%%%%%%%%%%%%%%%%%%
\begin{abstract}
%This paper develop a entropy stable high-order approximation of compressible Euler equations with oscillation-eliminating technical for general curvilinear meshes in two dimensions. The methodology is based on nodal discontinuous Galerkin method with summation-by-parts (SBP) property. The semi-discrete entropy conservation property is proved based on the SBP property and discrete metric identities for curvilinear meshes, and using the entropy stable numerical surface flux, the global entropy stable can be proved. Besides the entropy stable, we use and modify the very recent oscillation-eliminating discontinuous Galerkin(OEDG) method to control the nonphysical oscillations near strong discontinuities. We note that the zero order damping coefficient from the original OEDG method can play a role as shock indicator, and by which we can reduce the computation cost of OEDG method. The original OEDG method only applied to triangular mesh, and is hard to direct extend to general curvilinear meshes, because of original OEDG method is based on the local orthogonal modal basis which is hard to define on general curvilinear meshes. We use the projection operator to get over such point, and extend the OEDG to curvilinear meshes. Extensive numerical results are presented and shows the effectiveness and advantages of the entropy stable OEDG methods in general curvilinear meshes.
We develop an entropy-stable high-order numerical method for the
two-dimensional compressible Euler equations on general curvilinear meshes.
The proposed approach is based on a nodal discontinuous Galerkin spectral
element method (DGSEM) that satisfies the summation-by-parts (SBP) property.
At the semidiscrete level, entropy stability is established through the
SBP structure and the discrete metric identities associated with curvilinear
coordinate mappings.
By incorporating entropy-stable numerical fluxes at element interfaces, a
global discrete entropy inequality is obtained.
To further control nonphysical oscillations near strong discontinuities, the
entropy-stable DG formulation is combined with a modified
oscillation-eliminating discontinuous Galerkin (OEDG) method, which was originally proposed in \cite{peng2025oedg}.
We observe that the zero-order damping coefficient in the original OEDG method
naturally serves as an effective shock indicator, which enables localization
of the oscillation control mechanism and significantly reduces computational
cost.
Moreover, while the original OEDG formulation relies on local orthogonal modal
bases and is primarily restricted to simplicial meshes, we reformulate the OE
procedure using projection operators, allowing for a systematic extension to
general curvilinear meshes.
The resulting method preserves conservation and entropy stability while
effectively suppressing spurious oscillations.
A series of challenging numerical experiments is presented to demonstrate the
accuracy, robustness, and effectiveness of the proposed entropy-stable OEDG
method on both Cartesian and curvilinear meshes.

\end{abstract}

\vspace{0.05cm}

%%%%%%%%%%%%%%%%%%%%%%%%%%%%%%%%%%%%%%%%%%%%%%%%%%%%
%%%%%%%%%%%%%%%%%%%%%%%%%%%%%%%%%%%%%%%%%%%%%%%%%%%%
%%%%%%%%%%%%%%%%%%%%%%%%%%%%%%%%%%%%%%%%%%%%%%%%%%%%
%%%%%%%%%%%%%%%%%%%%%%%%%%%%%%%%%%%%%%%%%%%%%%%%%%%%
\section{Introduction}

%The compressible Euler equations play a fundamental role in fluid dynamics and serve as a core mathematical model for a wide range of applications, including aerodynamics, astrophysical flows, gas dynamics, and high-speed compressible flows. As a nonlinear hyperbolic system of conservation laws, the Euler equations admit complex solution structures such as shock waves, contact discontinuities, and rarefaction waves, even when starting from smooth initial data. Accurately and robustly resolving these multiscale and strongly nonlinear phenomena remains a central challenge in computational fluid dynamics.

%The design of numerical methods to approximate Euler equation is driven by the need for stable, accurate and robust scheme. In terms of accuracy, among numerous numerical methods, discontinuous Galerkin (DG) methods are naturally well-suited to the high order scheme and complex geometries. As a result, DG methods have become increasingly popular for the numerical approximation of hyperbolic conservation laws, including the compressible Euler equations\cite{cockburn1989tvb}. Besides accuracy, stability and robustness are also crucial for a scheme. For example preserving the entropy and positive properties of Euler equation are essentially important for a successful scheme of Euler equation. Controlling unphysical numerical oscillation is also important but challenging.

The compressible Euler equations play a fundamental role in fluid dynamics and serve as a core mathematical model for a wide range of applications, including aerodynamics, astrophysical flows, gas dynamics, and high-speed compressible flows. As a nonlinear hyperbolic system of conservation laws, the Euler equations admit complex solution structures such as shock waves, contact
discontinuities, and rarefaction waves, even when evolving from smooth initial data. Accurately and robustly resolving these multiscale and strongly nonlinear phenomena remains a central challenge in computational fluid dynamics \cite{abgrall_shu_2016_hyperbolic, abgrall_shu_2017_hyperbolic_applied}.

The development of numerical methods for the Euler equations is driven by the simultaneous requirements of accuracy, stability, and robustness. Among the many numerical approaches that have been proposed, discontinuous
Galerkin (DG) methods are particularly attractive due to their ability to achieve high-order accuracy on complex geometries while maintaining local conservation.
Consequently, DG methods have become an increasingly popular choice for the approximation of hyperbolic conservation laws, including the compressible
Euler equations \cite{cockburn1989tvb}. Beyond accuracy, stability and robustness are essential for reliable
simulations. In particular, preserving fundamental physical properties of the Euler equations, such as entropy stability and positivity of density and pressure,
is crucial for the success of a numerical scheme.
In addition, the control of nonphysical numerical oscillations near discontinuities remains an important and challenging issue for high-order methods.

The entropy inequality is a fundamental physical property of the Euler equations and serves as an admissibility criterion for physically relevant solutions. It is well known that shock waves and contact discontinuities may develop during the evolution of the Euler equations, even when the initial and boundary data are smooth. As a consequence, solutions are generally understood in the weak
(distributional) sense. However, weak solutions are not unique, and additional admissibility conditions are required to single out the physically meaningful solution.
This role is played by entropy conditions, which are consistent with the second law of thermodynamics.
From a numerical perspective, it is therefore essential to design schemes that preserve entropy properties at the discrete level. The pioneering work of Tadmor \cite{tadmor1987numerical} introduced the concept of entropy-conservative and entropy-stable schemes for systems of conservation laws, leading to the construction of first-order entropy-stable finite volume methods based on specially designed entropy-conservative numerical fluxes. Subsequent extensions include second-order schemes \cite{tadmor2003entropy} and arbitrarily high-order methods \cite{fjordholm2012arbitrarily}, relying on
high-order entropy-conservative flux constructions
\cite{lefloch2002fully}. These developments laid the theoretical foundation for entropy-stable
high-order discretizations.

In the context of discontinuous Galerkin (DG) methods, early progress on entropy stability was made by Jiang and Shu \cite{jiang1994cell}, who established a discrete entropy inequality for scalar conservation laws using the square entropy. This analysis was later extended to linear symmetric hyperbolic systems by Hou and LeFloch \cite{hou2007solutions}. Despite these foundational contributions, entropy-stable DG formulations remained limited in both scope and applicability for many years.
Substantial advances have been achieved more recently through the work of Carpenter et al. \cite{Carpenter14,fisher2013high}, who exploited the summation-by-parts (SBP) property of nodal DG discretizations based on Gauss–Legendre–Lobatto (GLL) points. These methods are commonly referred to as the DG spectral element method with Gauss–Legendre–Lobatto points (DGSEM–GLL) \cite{del2018simultaneous,zhang2023discontinuous}. Building on this framework, Chen and Shu \cite{chen2017entropy,chen2020review} developed a unified approach for constructing high-order entropy-stable DG methods for general systems on unstructured triangular meshes using SBP operators.
Further developments and refinements of entropy-stable DGSEM formulations have been reported in \cite{crean2018entropy,chan2018discretely,chan2019efficient}. Entropy-stable DGSEM–GLL methods have since been successfully applied to a broad class of systems, including the Euler and Navier–Stokes equations \cite{crean2018entropy}, shallow water equations \cite{gassner2016well,wintermeyer2017entropy}, magnetohydrodynamics (MHD) \cite{bohm2020entropy,liu2025entropy}, and multiphase or multicomponent flows \cite{renac2019entropy}. Finally, it is worth noting that DGSEM–GLL methods are closely related to summation-by-parts finite difference operators coupled with simultaneous approximation terms (SAT) \cite{gassner2016split,gassner2013skew}.

The construction of entropy-stable schemes on curvilinear meshes poses additional challenges. In this setting, entropy stability must be achieved simultaneously with the
preservation of discrete geometric conservation laws (GCLs), which is essential for maintaining free-stream preservation and nonlinear stability. Important progress in this direction has been reported in
\cite{fisher2012high,carpenter2016entropy,crean2018entropy,chan2019discretely,bohm2020entropy}, where careful discretization of geometric terms is required to ensure
compatibility between entropy stability and the discrete metric identities \cite{Kopriva2006}.

Despite these advances, entropy stability alone is generally insufficient to eliminate nonphysical oscillations near strong discontinuities. Such oscillations are closely related to the classical Gibbs phenomenon \cite{gottlieb1997gibbs} and are particularly pronounced in high-order numerical schemes, including discontinuous Galerkin methods. If left uncontrolled, Gibbs-type oscillations may severely degrade solution
quality, compromise robustness, or even lead to numerical breakdown. Consequently, additional nonlinear stabilization mechanisms are indispensable
in practical computations for hyperbolic conservation laws.
A wide range of shock-capturing techniques has been developed to suppress nonphysical oscillations.
Limiter-based approaches, such as total variation diminishing (TVD) and total variation bounded (TVB) limiters \cite{cockburn1989tvb,burbeau2001problem}, as
well as ENO and weighted ENO (WENO) limiters
\cite{qiu2005runge,zhong2013simple,zhu2008runge}, modify local solution reconstructions in troubled cells to enforce monotonicity or boundedness. Another widely used strategy is the introduction of artificial viscosity(AV),
which adds localized dissipation in regions with strong gradients \cite{jameson1981numerical,jameson1993artificial,persson2006sub,barter2010shock}. 
%While these methods are effective in many applications, they often introduce excessive numerical diffusion, reduce accuracy in smooth regions, or require problem-dependent tuning.

More recently, the oscillation-free DG (OFDG) framework has been systematically developed by Lu, Liu, and Shu
\cite{lu2021oscillation,liu2022essentially}. In OFDG, damping terms are incorporated directly into the DG formulation to suppress spurious oscillations.
This approach has demonstrated strong robustness and has been extended to a variety of systems, including chemically reacting flows\cite{du2023oscillation,zuo2024positivity}, shallow water equations \cite{liu2022oscillation}, and magnetohydrodynamics \cite{liu2024entropy}.
However, the additional damping terms render the resulting semidiscrete system highly stiff, leading to severe time-step restrictions in explicit time integration.
To alleviate this limitation, Peng et al.\ proposed the
oscillation-eliminating DG (OEDG) method in \cite{peng2025oedg}. The OEDG method follows a similar damping-based philosophy as OFDG, but formulates the oscillation control mechanism as a local pseudo-time evolution problem that can be solved exactly.
As a result, the stiffness associated with the damping terms is removed, and the stringent time-step constraints of OFDG are largely mitigated. Since its introduction, OEDG has been successfully applied to a variety of
problems, including systems with complex wave interactions
\cite{liu2025structure,ding2025robust,yan2024uniformly}.

Despite its advantages, the original OEDG formulation still faces several practical challenges. In particular, the computation of damping coefficients involves the evaluation of high-order spatial derivatives in multiple directions, which can be computationally expensive.
Moreover, the direct extension of OEDG to curvilinear meshes is not straightforward due to the lack of tensor-product structure and the presence of geometric complexities.
In this work, we address these limitations and further develop the OEDG framework. First, we observe that the damping coefficients appearing in the OE procedure naturally serve as effective shock indicators, closely related to the classical KXRCF indicator \cite{krivodonova2004shock}. By exploiting this connection, the OE procedure can be selectively applied
only in troubled cells, significantly reducing computational cost. Second, we extend the OEDG methodology to curvilinear meshes in a systematic and efficient manner. 
%Finally, the proposed OEDG procedure is combined with an entropy-stable DG formulation, yielding a robust high-order method that simultaneously enforces entropy stability and effectively suppresses nonphysical oscillations.

%The main contribution of this paper is the development and modification of OEDG method: generalizing it to the case of curvilinear meshes, reducing the computation cost by a new shock indicator; combining the modified OEDG with entropy stable DG methods in the framework of curvilinear meshes. Finally an entropy-stable nodal DG with oscillation-eliminating procedure is proposed for general curvilinear meshes, and some challenging numerical simulations are shown.

%The remainder of the paper is organized as follows.In Section~\ref{sec:preliminary}, we review the entropy properties of the Euler equations, their formulation in curvilinear coordinates, and the nodal DG polynomial framework. Section~\ref{subsec:dgsem_plain} presents the entropy-stable DGSEM on curvilinear meshes. The oscillation-eliminating procedure is introduced and analyzed in Section~\ref{sec:oedg}. Numerical examples demonstrating accuracy, robustness, and effectiveness of the proposed method are provided in Section~\ref{sec:numerical}. Finally, conclusions are drawn in Section~\ref{sec:conclude}.

The primary contribution of this work is the further development of the entropy stable oscillation-eliminating discontinuous Galerkin (OEDG) framework for the compressible Euler equations.
Specifically, the OEDG methodology is generalized to curvilinear meshes, where geometric complexity poses additional challenges.
A new shock indicator, derived directly from the OE damping coefficients, is introduced to localize the oscillation control mechanism and significantly
reduce computational cost.
The modified OEDG procedure is then combined with an entropy-stable DG formulation, resulting in a robust high-order nodal DG method that simultaneously enforces entropy stability and effectively suppresses nonphysical oscillations on general curvilinear meshes.
The performance of the proposed method is demonstrated through a series of challenging numerical experiments.

The remainder of the paper is organized as follows.
In Section~\ref{sec:preliminary}, we review the entropy properties of the Euler equations, their formulation in curvilinear coordinates, and the nodal DG polynomial framework.
Section~\ref{sec:dgsem_curve} presents the entropy-stable DGSEM formulation on curvilinear meshes.
The oscillation-eliminating procedure is introduced and analyzed in Section~\ref{sec:oedg}.
Numerical examples demonstrating the accuracy, robustness, and effectiveness of the proposed method are provided in Section~\ref{sec:numerical}.
Finally, concluding remarks are given in Section~\ref{sec:conclude}.

% ============================================================
% Preliminary section (polished narrative; manuscript-adapted)
% ============================================================

\section{Preliminary}
\label{sec:preliminary}
This section summarizes the theoretical background required for the
construction and analysis of entropy-stable nodal discontinuous
Galerkin methods for the two-dimensional Euler equations on
curvilinear meshes. In particular, we review the entropy properties of
the Euler equations, introduce their formulation in curvilinear
coordinates, and describe the nodal discontinuous Galerkin polynomial
spaces used throughout this work.
\subsection{Entropy properties of the Euler equations}

We consider the compressible Euler equations written in conservative
form
\begin{equation}
\partial_t \bm U + \nabla\cdot\bm F(\bm U)=0,
\end{equation}
where the vector of conservative variables is
\begin{equation}
\bm U =
\begin{pmatrix}
\rho\\
\rho\bm u\\
E
\end{pmatrix},
\end{equation}
with density $\rho$, velocity $\bm u=(u,v)^T$, and total energy
\[
E=\rho e+\tfrac12\rho|\bm u|^2.
\]
The flux tensor is given by
\begin{equation}
\bm F(\bm U)=
\begin{pmatrix}
\rho\bm u\\
\rho\bm u\otimes\bm u+p\bm I\\
(E+p)\bm u
\end{pmatrix},
\end{equation}
where the pressure $p$ is related to $(\rho,e)$ through an equation of
state. Throughout this work, we assume an ideal gas equation of state,
\[
p=(\gamma-1)\rho e,\qquad \gamma=1.4.
\]

In two spatial dimensions, the flux tensor can be written as
\begin{equation}
\label{flux-c}
\begin{aligned}
&\bm F(\bm U)=(\bm f(\bm U),\bm g(\bm U))^T,\qquad
\bm U=(U_1,U_2,U_3,U_4)^T,\\
&\bm f(\bm U)=(\rho u,\ \rho u^2+p,\ \rho uv,\ (E+p)u)^T,\\
&\bm g(\bm U)=(\rho v,\ \rho uv,\ \rho v^2+p,\ (E+p)v)^T,
\end{aligned}
\end{equation}
so that the Euler equations take the componentwise form
\begin{equation}
\label{euler-c}
\partial_t U_k+\partial_x f_k(\bm U)+\partial_y g_k(\bm U)=0,
\qquad k=1,\dots,4.
\end{equation}

The Euler system is nonlinear and hyperbolic, and solutions may develop
discontinuities such as shocks and contact discontinuities in finite
time, even for smooth initial data. As a consequence, weak solutions
are generally not unique, and additional admissibility criteria are
required to select physically relevant solutions. These criteria are
provided by entropy conditions.

An \emph{entropy function--entropy flux pair}
$(\eta(\bm U),\bm q(\bm U))$ consists of a strictly convex scalar entropy
$\eta$ and an associated entropy flux $\bm q$ satisfying
\begin{equation}
\eta'(\bm U)\bm F'(\bm U)=\bm q'(\bm U).
\end{equation}
For smooth solutions, multiplication of the Euler equations by
$\eta'(\bm U)$ yields the entropy conservation law
\begin{equation}
\partial_t\eta(\bm U)+\nabla\cdot\bm q(\bm U)=0.
\end{equation}
Across discontinuities, entropy is no longer conserved. A weak solution
$\bm U$ is called an \emph{entropy solution} if it satisfies the entropy
inequality
\begin{equation}
\partial_t\eta(\bm U)+\nabla\cdot\bm q(\bm U)\le0
\quad\text{in the sense of distributions},
\end{equation}
which is consistent with the second law of thermodynamics.

Associated with a convex entropy $\eta(\bm U)$ are the
\emph{entropy variables}
\begin{equation}
\bm V:=\frac{\partial\eta}{\partial\bm U}.
\end{equation}
Strict convexity of $\eta$ implies that the mapping
$\bm U\leftrightarrow\bm V$ is one-to-one on the admissible state space.
Introducing the flux functions expressed in entropy variables,
$\bm G(\bm V)=\bm F\Big(\bm U(\bm V)\Big)$, the compatibility condition above is
equivalent to the symmetry of the Jacobians $\bm G_j'(\bm V)$ for each
spatial direction $j$. This symmetrization property is fundamental for
entropy analysis and motivates the introduction of the potential
function and potential fluxes,
\begin{equation}
\begin{aligned}
&\phi(\bm V)=\bm U(\bm V)^T\bm V-\eta\Big(\bm U(\bm V)\Big),\\
&\bm\psi_j(\bm V)=\bm G_j(\bm V)^T\bm V-\bm q_j\Big(\bm U(\bm V)\Big),
\qquad j=1,2,
\end{aligned}
\end{equation}
which play a central role in the design of entropy-conserving and
entropy-stable numerical schemes.

For the ideal-gas Euler equations, the physical specific entropy is
$s=\log(p\rho^{-\gamma})$. While the Euler system admits a family of
entropy pairs \cite{Harten1983}, the physically relevant entropy for the
Navier--Stokes equations with viscosity and heat conduction is uniquely
given by \cite{Hughes1986}
\begin{equation}
\label{entropy}
\eta=-\frac{\rho s}{\gamma-1},
\qquad
\bm q=-\frac{\rho s\,\bm u}{\gamma-1}.
\end{equation}
The corresponding entropy variables and potential fluxes are
\begin{equation}
\label{entrx}
\bm V=
\begin{pmatrix}
\dfrac{\gamma-s}{\gamma-1}-\dfrac{|\bm u|^2}{2c^2}\\[0.8ex]
u/c^2\\[0.4ex]
v/c^2\\[0.4ex]
-1/c^2
\end{pmatrix},
\qquad
\bm\psi=(\psi^f,\psi^g)=\rho\bm u,
\end{equation}
where $c=\sqrt{\gamma p/\rho}$ denotes the speed of sound. Throughout
this paper, we adopt the entropy function \eqref{entropy}.

The entropy variables provide a symmetrizing change of variables under
which the Euler equations can be written in symmetric form. This
structure underpins the continuous entropy analysis and serves as the
foundation for the construction of entropy-conserving numerical fluxes
and entropy-stable DG discretizations developed in the subsequent
sections.

%%%%%%%%%%%%%%%%%%%%%%%%%%%%%%%%%%%%%%%%%%%%%%%%%%%%%%%%%%%%%%%%
%%%%%%%%%%%%%%%%%%%%%%%%%%%%%%%%%%%%%%%%%%%%%%%%%%%%%%%%%%%%%%%%
\subsection{Formulation in curvilinear coordinates}
\label{sec_form_in_curve}

To construct high-order discretizations on curvilinear meshes, we transform the Euler equations
\eqref{euler-c} to a reference coordinate system.
The spatial domain $\Omega\subset\mathbb R^2$ is decomposed into non-overlapping curved quadrilateral
elements $\{\Omega_e\}$, each of which is mapped to the reference element
$\Omega_{\mathrm{ref}}=[-1,1]^2$.

Let $(\xi,\eta)\in\Omega_{\mathrm{ref}}$ denote the reference coordinates and
$(x^e,y^e)\in\Omega_e$ the corresponding physical coordinates.
The element mapping is given by
\[
\vec{x}^{\,e}(\xi,\eta)
=
\big(x^{e}(\xi,\eta),\,y^{e}(\xi,\eta)\big),
\]
where the superscript $(\cdot)^e$ indicates quantities associated with the element $\Omega_e$.
Under this mapping, the conservative variables are pulled back to the reference element via
\[
\bm U^e(\xi,\eta) := \bm U\big(x^e(\xi,\eta),y^e(\xi,\eta)\big).
\]

Applying the chain rule to \eqref{euler-c}, the Euler equations on the reference element take the form
\begin{equation}
\mathcal J^{e}\,\partial_t U_k^e
+
\partial_{\xi}\Tilde{f}_k^e
+
\partial_{\eta}\Tilde{g}_k^e
=0,
\qquad k=1,2,3,4,
\label{eq_curve_coord}
\end{equation}
where $\mathcal J^{e}$ denotes the Jacobian of the coordinate transformation and
$\Tilde{f}_k^e$, $\Tilde{g}_k^e$ are the contravariant flux components.

The Jacobian and contravariant fluxes are defined by
\begin{equation}
\begin{aligned}
\mathcal J^{e}
&=
x^{e}_{\xi}y^{e}_{\eta}-x^{e}_{\eta}y^{e}_{\xi},\\
\Tilde{f}_k^e(\bm U^e)
&=
y^{e}_{\eta}f_k(\bm U^e)-x^{e}_{\eta}g_k(\bm U^e),
\qquad k=1,2,3,4,\\
\Tilde{g}_k^e(\bm U^e)
&=
-\,y^{e}_{\xi}f_k(\bm U^e)+x^{e}_{\xi}g_k(\bm U^e),
\qquad k=1,2,3,4,
\end{aligned}
\label{eq_curve_term}
\end{equation}
where $f_k$ and $g_k$ denote the Cartesian flux components defined in
\eqref{flux-c}.

Since the Jacobian $\mathcal J^{e}$ is generally non-constant, the transformed system
\eqref{eq_curve_coord} constitutes a variable-coefficient hyperbolic system on the reference element.

A fundamental property of smooth coordinate mappings is the satisfaction of the
\emph{metric identities}
\begin{equation}
\partial_{\xi}x^{e}_{\eta}=\partial_{\eta}x^{e}_{\xi},
\qquad
\partial_{\xi}y^{e}_{\eta}=\partial_{\eta}y^{e}_{\xi},
\label{eq_metric_id_cont}
\end{equation}
which follow from the commutation of mixed partial derivatives.
At the discrete level, however, these identities are not automatically satisfied and must be preserved
through a compatible discretization of the geometric terms.

%%%%%%%%%%%%%%%%%%%%%%%%%%%%%%%%%%%%%%%%%%%%%%%%%%%%%%%%%%%%%%%%
%%%%%%%%%%%%%%%%%%%%%%%%%%%%%%%%%%%%%%%%%%%%%%%%%%%%%%%%%%%%%%%%
\subsection{DG polynomials on the reference element}
\label{subsec:dg_poly_ref}

On the reference element $\Omega_{\mathrm{ref}}=[-1,1]^2$, we approximate all scalar- and vector-valued quantities
by polynomials of degree at most $N$ in each coordinate direction.
To this end, we introduce $N+1$ Legendre--Gauss--Lobatto (LGL) points
$\{\xi_i\}_{i=0}^N$ and $\{\eta_j\}_{j=0}^N$ on the interval $[-1,1]$ in each coordinate direction.
Associated with these nodes are the one-dimensional Lagrange interpolation polynomials
\begin{equation}
\label{lgl}
\phi_j(\xi)
=
\prod_{\substack{i=0\\ i\neq j}}^N
\frac{\xi-\xi_i}{\xi_j-\xi_i},
\qquad j=0,\ldots,N,
\end{equation}
which satisfy the nodal interpolation property $\phi_j(\xi_i)=\delta_{ij}$.

Using tensor products of the one-dimensional basis functions, we define the nodal polynomial space
\begin{equation}
\label{qN}
\mathbb Q_N(\Omega_{\mathrm{ref}})
:=
\mathrm{span}\Big\{
\phi_i(\xi)\phi_j(\eta)
\;:\;
0\le i,j\le N
\Big\},
\end{equation}
that is, the space of polynomials of degree at most $N$ in each coordinate direction.

Any scalar function $U$ on $\Omega_{\mathrm{ref}}$ can be approximated 
by $U_N\in \QN$ as
\begin{equation}
{U}(\xi,\eta)\approx U_N(\xi,\eta)=
\sum_{i=0}^N\sum_{j=0}^N U_{i,j}\phi_i(\xi)\phi_j(\eta),
\end{equation}
where $\{U_{i,j}\}$ are the nodal degrees of freedom (DOFs). 
In particular, for a component of the conservative variable $U_k$ on physical element $\Omega_e$, we approximate it via the pull back:
\begin{align}
\label{pull_back}
U_k(x^e(\xi, \eta), y^e(\xi, \eta))= {U}_k^e(\xi,\eta)
\approx {U}_{k,N}^e(\xi,\eta)=\sum_{i=0}^N\sum_{j=0}^N (U^e_k)_{i,j}\phi_i(\xi)\phi_j(\eta), \quad 
k=1,2,3,4.    
\end{align}
Moreover, we also use $\QN$ to approximate flux functions $\bm F(\bm U)$  via interpolation:
\begin{align}
    \label{flux}
\bm F(\bm U(x^e(\xi, \eta), y^e(\xi, \eta)))
\approx 
\bm F_N(\bm {U}_{N}^e)(\xi,\eta)=\sum_{i=0}^N\sum_{j=0}^N 
\bm F(\bm U_{i,j}^e)\phi_i(\xi)\phi_j(\eta).
\end{align}
This collocation-based formulation allows all unknowns and fluxes to be represented solely by their values
at the LGL nodes.
% We denote the vector  (value of $\bm U$ at collocation point $(\xi_i, \eta_j)$)
% \[
% \bm U_{i,j}:= [(U^e_1)_{i,j}, (U^e_2)_{i,j}, (U^e_3)_{i,j}, (U^e_4)_{i,j}]^T.
% \]
% Flux variables (functions of $\bm U$) can then be evaluated on each collocation point $(\xi_i, \eta_j)$ via
% \[
% \bm F_{i,j}=\bm F(\bm U_{i,j}), \qquad \forall\, i,j.
% \]

To approximate derivatives, we introduce the differentiation matrix
\begin{equation}
D_{i,j} = \phi_j'(\xi_i), \qquad i,j=0,\ldots,N,
\label{eq_D}
\end{equation}
so that nodal derivatives in reference space are given by
\begin{equation}
\begin{aligned}
\dfrac{\partial U}{\partial\xi}(\xi_i,\eta_j) &= \sum_{p=0}^N D_{i,p}\,U_{p,j},\\
\dfrac{\partial U}{\partial\eta}(\xi_i,\eta_j) &= \sum_{p=0}^N D_{j,p}\,U_{i,p}.
\end{aligned}
\label{der}
\end{equation}

The differentiation matrix satisfies the following standard properties.

\begin{lemma}[Properties of $D_{i,j}$]
\label{lemma_D}
\begin{itemize}
\item {Summation-by-parts (SBP) property:}
\begin{equation}
w_iD_{i,j}+w_jD_{j,i}=\delta_{N,i}\delta_{N,j}-\delta_{0,i}\delta_{0,j}.
\label{eq_sbp_1}
\end{equation}
\item {Summation identity:}
\begin{equation}
\label{sum-1}
\sum_{j=0}^N D_{i,j}=0.
% ,\qquad
% \sum_{i=0}^N w_iD_{i,j}=\delta_{N,j}-\delta_{0,j}.
\end{equation}
\end{itemize}
\end{lemma}

The SBP property is the discrete analogue of integration by parts and is central in proving entropy stability.

\begin{remark}[Discrete metric identities]
\label{rem:discrete_metric}
At the continuous level, metric identities \eqref{eq_metric_id_cont} hold whenever the mapping is sufficiently smooth.
At the discrete level, metric identities are not automatically satisfied and must be enforced through a compatible
discretization of the geometric terms. One convenient form is
\begin{equation}
\begin{aligned}
\sum_{p_2=0}^N D_{q_2,p_2}(y_{\xi})_{q_1,p_2}
&=\sum_{p_1=0}^N D_{q_1,p_1}(y_{\eta})_{p_1,q_2}, \qquad \forall\, q_1,q_2,\\
\sum_{p_2=0}^N D_{q_2,p_2}(x_{\xi})_{q_1,p_2}
&=\sum_{p_1=0}^N D_{q_1,p_1}(x_{\eta})_{p_1,q_2}, \qquad \forall\, q_1,q_2,
\end{aligned}
\label{eq_metric_id_dis}
\end{equation}
where $(\cdot)_{i,j}$ denotes nodal DOFs of the corresponding geometric derivatives. If the coordinate mapping $(x(\xi,\eta),y(\xi,\eta))$ belongs to
$\mathbb Q_N(\Omega_{\mathrm{ref}})$, then the discrete metric identities \eqref{eq_metric_id_dis} are satisfied
automatically when the geometric terms are evaluated using the same LGL interpolation and differentiation
operators; see Kopriva~\cite{Kopriva2006} for details.
\end{remark}

%%%%%%%%%%%%%%%%%%%%%%%%%%%%%%%%%%%%%%%%%%%%%%%%%%5%%
\subsection{LGL collocation-based quadrature on the reference element}
\label{sec:collocation}
Let  $\{w_i\}_{i=0}^N$ denotes the LGL quadrature weights on the interval $[-1,1]$. On the reference element $\Omega_{\mathrm{ref}} = [-1,1]^2$, tensor-product LGL quadrature yields
the following approximation of volume integrals:
\begin{equation}
\int_{\Omega_{\mathrm{ref}}} u\, v \, d\xi \, d\eta
\;\approx\;
(u,v)_{\Omega_{\mathrm{ref}},w}
:=
\sum_{q_1=0}^{N}\sum_{q_2=0}^{N}
w_{q_1} w_{q_2}\,
u_{q_1,q_2} \, v_{q_1,q_2}.
\label{eq_quad_rule}
\end{equation}

Boundary integrals on $\partial\Omega_{\mathrm{ref}}$ are approximated using
one-dimensional LGL quadrature applied independently on each face.
We define the discrete boundary inner product
\begin{equation}
\langle u, v \rangle_{\partial\Omega_{\mathrm{ref}},w}
:=
\sum_{\gamma \subset \partial\Omega_{\mathrm{ref}}}
\sum_{q=0}^{N}
w_q \,
u|_{\gamma,q} \, v|_{\gamma,q},
\label{eq:boundary_inner_product}
\end{equation}
where $\gamma$ denotes one of the four faces of $\Omega_{\mathrm{ref}}$, and
$u|_{\gamma,q}$, $v|_{\gamma,q}$ denote the evaluation of $u$ and $v$ at the
$q$th LGL node on face $\gamma$.

Equivalently, writing the contributions from the four faces explicitly,
\begin{align}
\langle u, v \rangle_{\partial\Omega_{\mathrm{ref}},w}
&=
\sum_{q=0}^{N} w_q \Big(
u(-1,\eta_q)v(-1,\eta_q)
+u(1,\eta_q)v(1,\eta_q)
\nonumber\\
&\qquad\qquad\quad
+u(\xi_q,-1)v(\xi_q,-1)
+u(\xi_q,1)v(\xi_q,1)
\Big).
\label{eq:boundary_inner_product_explicit}
\end{align}

The discrete inner products
$(\cdot,\cdot)_{\Omega_{\mathrm{ref}},w}$ and
$\langle\cdot,\cdot\rangle_{\partial\Omega_{\mathrm{ref}},w}$
are used throughout to approximate volume and surface integrals on the
reference element.

% For a numerical flux vector
% $\widehat{\boldsymbol F} = (\tilde f_k^{e,*}, \tilde g_k^{e,*})$,
% we define the discrete numerical flux pairing by
% \begin{equation}
% \left\langle
% \widehat{\boldsymbol F}\cdot \hat{\boldsymbol n},\, v
% \right\rangle_{\partial\Omega_{\mathrm{ref}},w}
% :=
% \sum_{\gamma \subset \partial\Omega_{\mathrm{ref}}}
% \sum_{q=0}^N
% w_q\,
% \Big(\widehat{\boldsymbol F}\cdot \hat{\boldsymbol n}\Big)\big|_{\gamma,q}
% \, v\big|_{\gamma,q},
% \label{eq:boundary_flux_pairing}
% \end{equation}
% where $\hat{\boldsymbol n}$ denotes the outward unit normal vector on
% $\partial\Omega_{\mathrm{ref}}$.

% Similarly
% \begin{equation}
% \int_{\Omega_{\mathrm{ref}}} u_{\xi} v\, d\xi d\eta
% \approx (u_{\xi},v)_w :=
% \sum_{q_1=0}^{N}\sum_{q_2=0}^N w_{q_1}w_{q_2}\, v_{q_1,q_2}
% \sum_{p_1=0}^N D_{q_1,p_1}u_{p_1,q_2}.
% \end{equation}

\section{Entropy stable DGSEM on Curvilinear Meshes}
\label{sec:dgsem_curve}
In this section we describe the discontinuous Galerkin spectral element method (DGSEM) on curvilinear quadrilateral meshes. We first present the standard strong-form DGSEM discretization with LGL quadrature rules, which will serve as the baseline scheme. Next, this formulation will be modified using entropy conserving volume flux
to obtain an entropy-stable DG method.

Throughout this section, the spatial domain $\Omega$ is decomposed into conforming non-overlapping curved quadrilateral elements $\Omega_h= \{\Omega_e\}_{e=1}^{N_e}$, each mapped to the reference element $\Omega_{\mathrm{ref}}=[-1,1]^2$.
All discretizations are defined elementwise; coupling between elements is achieved through numerical fluxes
at shared interfaces.

\subsection{Strong-form DGSEM on curvilinear meshes}
\label{subsec:dgsem_plain}
We discretize the mapped Euler equations \eqref{eq_curve_coord} on each reference
element $\Omega_{\mathrm{ref}}=[-1,1]^2$, using the contravariant flux formulation
introduced in \eqref{eq_curve_term}. Throughout this section, quantities
superscripted by $(\cdot)^e$ are understood to be local to the physical element
$\Omega_e$.

The standard nodal DG weak formulation on each element reads as follows:
find $\bm U_N^e= (U_1^e,U_2^e, U_3^e, U_4^e)^T\in[\QN]^4$ such that, for $k=1,2,3,4$,
\begin{equation}
\int_{\Omega_{\mathrm{ref}}} \mathcal{J}^e\,\partial_tU_k^e\,\hat\Phi \,d\xi d\eta
-\int_{\Omega_{\mathrm{ref}}}\Tilde{f}_k^e\,\partial_{\xi}\hat\Phi \,d\xi d\eta
-\int_{\Omega_{\mathrm{ref}}}\Tilde{g}_k^e\,\partial_{\eta}\hat\Phi \,d\xi d\eta
+\int_{\partial\Omega_{\mathrm{ref}}}(\Tilde{f}_k^{e,*},\Tilde{g}_k^{e,*})\cdot \hat{\bm n}\,\hat\Phi \,d\hat s=0,\label{eq_weak_form}
\end{equation}
for all test functions $\hat \Phi(\xi,\eta)\in \mathbb{Q}_N(\Omega_{\mathrm{ref}})$.
Here $\hat{\bm n}$ denotes the outward unit normal vector on the reference boundary
$\partial\Omega_{\mathrm{ref}}$, and $\Tilde{f}_k^{e,*}$, $\Tilde{g}_k^{e,*}$ are the 
contravariant numerical fluxes of the form
\begin{equation}
\label{eq_contravariant_star}
\tilde{f}_k^{e,*} = y_{\eta}^e f_k^* - x_{\eta}^e g_k^*,
\qquad
\tilde{g}_k^{e,*} = -y_{\xi}^e f_k^* + x_{\xi}^e g_k^*,
\end{equation} 
with the numerical flux $\bm F_k^* = (f_k^*, g_k^*)^T$ on the physical element $\Omega_e$, to be specified below.
The physical conservative variables $U_k(x^e,y^e)$ are recovered from
$U_k^e(\xi,\eta)$ via the pull-back mapping \eqref{pull_back}.

Applying integration by parts to \eqref{eq_weak_form} yields the equivalent
strong form
\begin{equation}
\int_{\Omega_{\mathrm{ref}}} \mathcal{J}^e\,\partial_tU_k^e\,\hat\Phi \,d\xi d\eta
+\int_{\Omega_{\mathrm{ref}}}\big(\partial_{\xi}\Tilde{f}_k^e+\partial_{\eta}\Tilde{g}_k^e\big)\,\hat\Phi \,d\xi d\eta
-\int_{\partial\Omega_{\mathrm{ref}}}(\Tilde{f}_k^e-\Tilde{f}_k^{e,*},\Tilde{g}_k^e-\Tilde{g}_k^{e,*})\cdot \hat{\bm n}\,\hat\Phi \,d\hat s=0.
\label{eq_strong_form}
\end{equation}

Next, we approximate the flux functions by polynomial interpolants
\begin{align}
 \Tilde f_{k}^e \approx  \Tilde f_{k,N}^e=\sum_{i=0}^N\sum_{j=0}^N
 (\Tilde f_{k}^e)_{i,j}\phi_i(\xi)\phi_j(\eta),\quad 
 \Tilde g_{k}^e \approx \Tilde g_{k,N}^e=\sum_{i=0}^N\sum_{j=0}^N
 (\Tilde g_{k}^e)_{i,j}\phi_i(\xi)\phi_j(\eta),
\end{align}
where the nodal contravariant flux values are computed via \eqref{eq_curve_term}. 
For example,
\[
(\Tilde f_{k}^e)_{i,j} = 
(y^{e}_{\eta})_{i,j}f_k\big((\bm U_N^e)_{i,j}\big)-
(x^{e}_{\eta})_{i,j}g_k\big((\bm U_N^e)_{i,j}\big).
\]
Next, evaluating the integrals in \eqref{eq_strong_form} using the tensor-product LGL quadrature rules \eqref{eq_quad_rule} and \eqref{eq:boundary_inner_product} yields the semidiscrete DGSEM scheme
\begin{equation}
\label{eq_dgsem}
\left(\mathcal{J}^e\,\partial_tU_k^e, \hat\Phi\right)_{\omgr,w}
+
\left(\big(\partial_{\xi}\Tilde{f}_{k,N}^e+\partial_{\eta}\Tilde{g}_{k,N}^e\big), \hat\Phi\right)_{\omgr,w}
-\left\langle(\Tilde{f}_{k,N}^e-\Tilde{f}_k^{e,*},\Tilde{g}_{k,N}^e-\Tilde{g}_k^{e,*})\cdot \hat{\bm n},\hat\Phi \right\rangle_{\domgr,w}=0.
\end{equation}

To motivate the definition of the numerical flux $\boldsymbol F_k^{*}$, it is
convenient to push forward the reference boundary integral
$\left\langle (\Tilde{f}_k^{e,*},\Tilde{g}_k^{e,*})\cdot \hat{\bm n}, \hat\Phi \right\rangle_{\partial\Omega_{\mathrm{ref}},w}$
to the physical element boundary $\partial\Omega_e$. By definition \eqref{eq_contravariant_star}, we obtain
\begin{align}
\label{eq_map}
\left\langle (\Tilde{f}_k^{e,*},\Tilde{g}_k^{e,*})\cdot \hat{\bm n}, \hat\Phi \right\rangle_{\partial\Omega_{\mathrm{ref}}}
&=
\int_{\partial\Omega_{\mathrm{ref}}}
\boldsymbol F_k^{*} \cdot
\bigl(\mathcal{J}^e(\mathcal{G}^e)^{-T} \hat{\boldsymbol n}\bigr)\,
\hat\Phi\, d\hat s =
\int_{\partial\Omega_e}
\boldsymbol F_k^{*} \cdot \boldsymbol n\, \Phi\, ds,
\end{align}
which is the corresponding physical-edge integral.
Here $\hat\Phi(\xi,\eta) = \Phi(x(\xi,\eta),y(\xi,\eta))$ denotes the standard
pull-back of the physical test function to the reference element, and
\[
\mathcal{G}^e
:=
\begin{pmatrix}
\partial_\xi x & \partial_\xi y \\
\partial_\eta x & \partial_\eta y
\end{pmatrix}
\]
is the Jacobian matrix of the mapping from reference element $\omgr$ to 
physical element $\Omega_e$.
The vector $\mathcal{J}^e(\mathcal{G}^e)^{-T} \hat{\boldsymbol n}$ represents the
scaled outward normal on the physical edge and satisfies the normal-measure
identity
\[
\bigl(\mathcal{J}^e(\mathcal{G}^e)^{-T} \hat{\boldsymbol n}\bigr)\, d\hat s
=
\boldsymbol n\, ds.
\]

Using the discrete boundary inner product defined in
\eqref{eq:boundary_inner_product}, the numerical flux term is approximated by
\begin{align}
\left\langle (\Tilde{f}_k^{e,*},\Tilde{g}_k^{e,*})\cdot \hat{\bm n}, \hat\Phi \right\rangle_{\partial\Omega_{\mathrm{ref}}}
=
\sum_{\gamma \subset \partial\Omega_{\mathrm{ref}}}
\sum_{q=0}^{N}
w_q \,
\left.
\Bigl(
\boldsymbol F_k^{*} \cdot
\bigl(\mathcal{J}^e(\mathcal{G}^e)^{-T} \hat{\boldsymbol n}\bigr)\,
\hat\Phi
\Bigr)
\right|_{\gamma,q},
\label{phy}
\end{align}
which approximates the physical boundary integral
$\int_{\partial\Omega_e} \bm F_k^* \cdot \bm n \, \Phi \, ds$.

Hence, to evaluate the integral \eqref{phy}, it remains to define the numerical
flux in the normal direction on each physical edge.
Let $E = \partial\Omega_L \cap \partial\Omega_R$ be an interior edge shared by
two elements $\Omega_L$ and $\Omega_R$. Denote by $\boldsymbol U_L$ and
$\boldsymbol U_R$ the traces of the conservative state on the two sides of $E$,
and let $\boldsymbol n = (n_x,n_y)$ be the unit normal vector on $E$, oriented
from $\Omega_L$ to $\Omega_R$.
The normal numerical flux is chosen as the local Lax--Friedrichs (Rusanov) flux,
defined componentwise by
\begin{equation}
\label{eq:llf_normal_flux}
\boldsymbol F_k^{*} \cdot \boldsymbol n
=
\frac12\Big(
\boldsymbol n \cdot \boldsymbol F_k(\boldsymbol U_L)
+
\boldsymbol n \cdot \boldsymbol F_k(\boldsymbol U_R)
\Big)
-
\frac{\alpha}{2}\Big(
(U_k)_R - (U_k)_L
\Big),
\end{equation}
where the dissipation parameter $\alpha$ is taken as
\begin{align}
\label{alpha-x}
\alpha
=
\max\bigl\{
|\boldsymbol u_L \cdot \boldsymbol n| + c_L,\;
|\boldsymbol u_R \cdot \boldsymbol n| + c_R
\bigr\},    
\end{align}
where $\boldsymbol u_{L/R}$ and $c_{L/R}$ denote the velocity vector and sound
speed evaluated at the left and right states, respectively.

Choosing the test function in \eqref{eq_dgsem} as
$\hat\Phi=\phi_{p_1}(\xi)\phi_{p_2}(\eta)$ and evaluating the resulting terms yields the nodal (collocation) form of the scheme:
\begin{equation}
\label{eq_dgsem1}
\begin{aligned}
&w_{p_1}w_{p_2}\mathcal{J}_{p_1,p_2}^e\,(\partial_t U_k^e)_{p_1,p_2}
+
w_{p_1}w_{p_2}\sum_{i=0}^N D_{p_1,i}\,(\tilde f_k^e)_{i,p_2}
+
w_{p_1}w_{p_2}\sum_{j=0}^N D_{p_2,j}\,(\tilde g_k^e)_{p_1,j} \\
&\quad
+
w_{p_1}\Bigl(
(\Delta\tilde g_k^e)_{p_1,0}\,\delta_{p_2,0}
-
(\Delta\tilde g_k^e)_{p_1,N}\,\delta_{p_2,N}
\Bigr)
+
w_{p_2}\Bigl(
(\Delta\tilde f_k^e)_{0,p_2}\,\delta_{p_1,0}
-
(\Delta\tilde f_k^e)_{N,p_2}\,\delta_{p_1,N}
\Bigr)
=0,
\end{aligned}
\end{equation}
for all $0\le p_1,p_2\le N$ and $k=1,2,3,4$, where
$\Delta\tilde f_k^e=\tilde f_k^e-\tilde f_k^{e,*}$ and
$\Delta\tilde g_k^e=\tilde g_k^e-\tilde g_k^{e,*}$.

The strong-form DGSEM on curvilinear meshes is defined by the semidiscrete scheme
\eqref{eq_dgsem} together with the local Lax--Friedrichs numerical flux
\eqref{eq:llf_normal_flux}. This formulation serves as the baseline DGSEM and
provides the foundation for the entropy-stable modifications introduced in the
following subsection.

%%%%%%%%%%%%%%%%%%%%%%%%%%%%%%%%%%%%%%%%%%%%%%%%%%%%%%%%%%%%%%%%%%55
\subsection{Entropy stable DGSEM}
The strong-form DGSEM discretization introduced in the previous subsection is conservative and high-order
accurate, but it does not satisfy a discrete entropy conservation or entropy stability property.
In particular, the pointwise evaluation of nonlinear fluxes in the volume terms of \eqref{eq_dgsem} prevents
the scheme from mimicking the continuous entropy balance, even when entropy-stable numerical fluxes are used
at element interfaces.

To obtain an entropy-stable formulation, we follow the entropy framework of Tadmor
\cite{tadmor1987numerical,tadmor2003entropy}
and exploit the flexibility of nodal DG methods to modify the
\emph{volume discretization}. The key idea is to replace the pointwise fluxes in the volume integrals of
\eqref{eq_dgsem} by symmetric two-point numerical fluxes that are entropy conservative at the discrete level; see, e.g. \cite{Carpenter14,carpenter2016entropy}.

We first recall the definition of an entropy-conservative two-point flux.
\begin{definition}
\label{def_ec_vol_flux}
A two-point numerical volume flux
\[
\bm F^\#(\bm U_L,\bm U_R)
=
\big(\bm f^\#(\bm U_L,\bm U_R),\bm g^\#(\bm U_L,\bm U_R)\big)
\]
is said to be \emph{entropy conservative} for the physical flux
$\bm F(\bm U)=(\bm f(\bm U),\bm g(\bm U))$ if it satisfies:
\begin{itemize}
\item \textbf{Consistency:}
$\bm F^\#(\bm U,\bm U)=\bm F(\bm U)$;
\item \textbf{Symmetry:}
$\bm F^\#(\bm U_L,\bm U_R)=\bm F^\#(\bm U_R,\bm U_L)$;
\item \textbf{Entropy conservation:}
\begin{equation}
\label{ecx}
(\bm V_R-\bm V_L)\cdot \bm F^\#(\bm U_L,\bm U_R)
=
\bm\psi_R-\bm\psi_L,
\end{equation}
where $\bm V$ and $\bm\psi=(\psi^f,\psi^g)$ denote the entropy variables and entropy potential fluxes \eqref{entrx} associated with the entropy function \eqref{entropy}.
\end{itemize}
\end{definition}

For the Euler equations, several entropy-conservative two-point fluxes have been proposed in the literature,
including the fluxes of Ismail and Roe \cite{ismail2009affordable} and Chandrashekar
\cite{chandrashekar2013kinetic}. In this work, we adopt the entropy-conservative flux introduced by
Chandrashekar \cite{chandrashekar2013kinetic}, which is also  kinetic-energy preserving. Its components are given by
\begin{equation}
\label{eq_chandrasheka}
\begin{alignedat}{2}
&f_1^\# = \rho^{\log}\avg{u},
&&\quad
g_1^\# = \rho^{\log}\avg{v},\\
&f_2^\# = \rho^{\log}\avg{u}^2+\hat{p},
&&\quad
g_2^\# = \rho^{\log}\avg{u}\avg{v},\\
&f_3^\# = \rho^{\log}\avg{u}\avg{v},
&&\quad
g_3^\# = \rho^{\log}\avg{v}^2+\hat{p},\\
&f_4^\# = \rho^{\log}\avg{u}\hat{h},
&&\quad
g_4^\# = \rho^{\log}\avg{v}\hat{h},
\end{alignedat}
\end{equation}
where $f_k^\#$ and $g_k^\#$ denote the components of the numerical fluxes
$\bm f^\#$ and $\bm g^\#$, respectively.

The arithmetic and logarithmic averages appearing above are defined by
\begin{equation}
\avg{u}=\frac{u_L+u_R}{2},
\qquad
u^{\log}=\frac{u_L-u_R}{\log(u_L)-\log(u_R)},
\end{equation}
and the averaged pressure and enthalpy are given by
\begin{equation}
\begin{aligned}
\hat{p}=\frac{\avg{\rho}}{2\avg{\beta}},\quad \hat{h}
&=
\frac{1}{2\beta^{\log}(\gamma-1)}
-\frac12\big(\avg{u}^2+\avg{v}^2\big)
+\frac{\hat{p}}{\rho^{\log}}
+\avg{u}^2+\avg{v}^2,
\end{aligned}
\end{equation}
where $\beta = \frac{\rho}{2p}.$

Since the entropy-conservative volume fluxes are defined in a two-point fashion, the geometric terms in the
contravariant fluxes must be discretized consistently. Following
\cite{wintermeyer2017entropy}, all metric coefficients are replaced by symmetric arithmetic averages, e.g.,
\begin{equation}
\{\!\{x_{\xi}\}\!\}_{L,R}
=
\frac12\big((x_{\xi})_L+(x_{\xi})_R\big).
\end{equation}

Using these averaged metric terms, the entropy-conservative contravariant volume fluxes are defined as
\begin{align}
\label{ec}
(\Tilde{f}_k^{e,\#})_{(i,p_1),p_2}
&=
\{\!\{y_{\eta}^e\}\!\}_{(i,p_1),p_2}(f_k^\#)_{(i,p_1),p_2}
-
\{\!\{x_{\eta}^e\}\!\}_{(i,p_1),p_2}(g_k^\#)_{(i,p_1),p_2},\\
(\Tilde{g}_k^{e,\#})_{p_1,(j,p_2)}
&=
-
\{\!\{y_{\xi}^e\}\!\}_{p_1,(j,p_2)}(f_k^\#)_{p_1,(j,p_2)}
+
\{\!\{x_{\xi}^e\}\!\}_{p_1,(j,p_2)}(g_k^\#)_{p_1,(j,p_2)}.
\end{align}

The entropy-stable DGSEM is obtained by replacing the volume fluxes in \eqref{eq_dgsem} with the
entropy-conservative fluxes \eqref{ec}. Specifically, for each element $\Omega_e$, we seek an approximatate solution
$$\bm U_N^e= (U_1^e,U_2^e, U_3^e, U_4^e)^T\in[\QN]^4$$ 
% nodal values
% $U^e=\{U_{p_1,p_2}^e\}$ 
such that for all $0\le p_1,p_2\le N$ and $k=1,\dots,4$, the following semi-discrete formulation holds:
\begin{equation}
\label{eq_esdgsem}
\begin{aligned}
&w_{p_1}w_{p_2}\left(\mathcal{J}_{p_1,p_2}^e\,(\partial_tU_k^e)_{p_1,p_2}
+
2\sum_{i=0}^N D_{p_1,i}\,(\tilde{f}_k^{e,\#})_{(i,p_1),p_2}
+
2
\sum_{j=0}^N D_{p_2,j}\,(\tilde{g}_k^{e,\#})_{p_1,(j,p_2)}\right)\\
&+
w_{p_1}\Bigl(
(\Delta\tilde g_k^e)_{p_1,0}\,\delta_{p_2,0}
-
(\Delta\tilde g_k^e)_{p_1,N}\,\delta_{p_2,N}
\Bigr)
+
w_{p_2}\Bigl(
(\Delta\tilde f_k^e)_{0,p_2}\,\delta_{p_1,0}
-
(\Delta\tilde f_k^e)_{N,p_2}\,\delta_{p_1,N}
\Bigr)
=0.
\end{aligned}
\end{equation}

The factor of $2$ arises from the symmetric split-form discretization and, together with the SBP property of
the differentiation matrices, ensures that the discrete entropy balance mirrors the continuous entropy
identity.

The fundamental properties of the resulting entropy-stable DGSEM are summarized below.
\begin{lemma}[Single-element analysis]
\label{lma:1}
Let $\{\bm U_N^e\}_e$ denote the numerical solution of the entropy-stable DGSEM~\eqref{eq_esdgsem}.  
On a single element $\Omega_e$, the discrete evolution of the conserved variables  satisfies the following relations:
for $k = 1,2,3,4$,
\begin{equation}
\label{eq_cc}
\begin{aligned}
\sum_{p_1,p_2=0}^N
w_{p_1} w_{p_2}\,
\mathcal{J}_{p_1,p_2}^e\,(\partial_t U_k^e)_{p_1,p_2}
= {} &
\sum_{p_1=0}^N w_{p_1}
\Bigl(
(\tilde g_k^{e,*})_{p_1,0}
-
(\tilde g_k^{e,*})_{p_1,N}
\Bigr) \\
&\quad +
\sum_{p_2=0}^N w_{p_2}
\Bigl(
(\tilde f_k^{e,*})_{0,p_2}
-
(\tilde f_k^{e,*})_{N,p_2}
\Bigr).
\end{aligned}
\end{equation}

Moreover, the discrete entropy evolves according to
\begin{equation}
\label{eq_ss}
\begin{aligned}
\sum_{p_1,p_2=0}^N
w_{p_1} w_{p_2}\,
\mathcal{J}_{p_1,p_2}^e\,
\bigl(\partial_t \eta(\bm U_N^e)\bigr)_{p_1,p_2}
= {} &
\sum_{p_1=0}^N \sum_{k=1}^4
w_{p_1}
\Bigl(
(\tilde g_k^{e,*})_{p_1,0} (V_k^e)_{p_1,0}
-
\tilde\psi^{e,g}_{p_1,0}
\Bigr) \\
&\; -
\sum_{p_1=0}^N \sum_{k=1}^4
w_{p_1}
\Bigl(
(\tilde g_k^{e,*})_{p_1,N} (V_k^e)_{p_1,N}
-
\tilde\psi^{e,g}_{p_1,N}
\Bigr) \\
&\; +
\sum_{p_2=0}^N \sum_{k=1}^4
w_{p_2}
\Bigl(
(\tilde f_k^{e,*})_{0,p_2} (V_k^e)_{0,p_2}
-
\tilde\psi^{e,f}_{0,p_2}
\Bigr) \\
&\; -
\sum_{p_2=0}^N \sum_{k=1}^4
w_{p_2}
\Bigl(
(\tilde f_k^{e,*})_{N,p_2} (V_k^e)_{N,p_2}
-
\tilde\psi^{e,f}_{N,p_2}
\Bigr).
\end{aligned}
\end{equation}

Here, the entropy variables are defined by
\[
(V_k^e)_{p_1,p_2}
=
\left(
\frac{\partial \eta}{\partial U_k}(\bm U_N^e)
\right)_{p_1,p_2},
\qquad k = 1,\dots,4,
\]
and the contravariant mapped entropy potentials are given by
\[
\tilde\psi^{e,f}_{p_1,p_2}
:=
(y_\eta^e)_{p_1,p_2}\,\psi^f_{p_1,p_2}
-
(x_\eta^e)_{p_1,p_2}\,\psi^g_{p_1,p_2},
\qquad
\tilde\psi^{e,g}_{p_1,p_2}
:=
-(y_\xi^e)_{p_1,p_2}\,\psi^f_{p_1,p_2}
+
(x_\xi^e)_{p_1,p_2}\,\psi^g_{p_1,p_2}.
\]
\end{lemma}

\begin{proof}
The proof closely follows the arguments presented in
\cite{fisher2012high,Fisher2013Discretely,wintermeyer2017entropy,Friedrich2018EntropyStable}.
For completeness, the proof adapted to the present notation is provided in the Appendix\eqref{sec:appendix 1}.
\end{proof}

\begin{theorem}[Multi-element entropy analysis]
Consider the Euler system \eqref{euler-c} equipped with periodic boundary conditions, and let
$\{\bm U_N^e\}_e$ denote the solution of the entropy-stable DGSEM \eqref{eq_esdgsem}.
Assume that the numerical interface flux \eqref{eq:llf_normal_flux} is entropy stable in the sense that
\begin{equation}
\label{eq:esf}    
(\bm V_R - \bm V_L)\cdot\bigl(\boldsymbol F^{*}(\bm U_L, \bm U_R)\cdot\bm n\bigr)
-
(\bm \psi_R - \bm \psi_L)\cdot\bm n \le 0,
\end{equation}
where the unit normal vector $\bm n$ points from the left state to the right state.
Then the DGSEM is entropy stable on the whole computational domain, i.e.,
\begin{equation}
\label{eq_ssg}
\sum_{e=1}^{N_e}\sum_{p_1,p_2=0}^N
w_{p_1} w_{p_2}\,
\mathcal{J}_{p_1,p_2}^e\,
\bigl(\partial_t \eta(\bm U_N^e)\bigr)_{p_1,p_2}
\le 0.
\end{equation}
\end{theorem}
\begin{proof}
Using the mapping relation \eqref{eq_map}, the single-element entropy balance \eqref{eq_ss} can be expressed in the physical domain as
\begin{equation*}
(\eta(\bm U^e), 1)_{\Omega_e,h}
=
-\left\langle
\sum_{k=1}^4 \bigl(\bm F_k^* \cdot \bm n\bigr) V_k
-
\bm \psi \cdot \bm n,
\, 1
\right\rangle_{\partial\Omega_e,h},
\end{equation*}
where $(\cdot,\cdot)_{\Omega_e,h}$ denotes the mapped LGL quadrature on the physical element $\Omega_e$, and
$\langle\cdot,\cdot\rangle_{\partial\Omega_e,h}$ denotes the corresponding mapped LGL quadrature on the element boundary $\partial\Omega_e$.

Summing over all elements $\{\Omega_e\}_{e=1}^{N_e}$ and invoking periodic boundary conditions yields
\begin{equation*}
\sum_{e=1}^{N_e} (\eta(\bm U^e), 1)_{\Omega_e,h}
=
\sum_{E\in\mathcal{E}_h^o}
\left\langle
(\bm V_R - \bm V_L)\cdot\bigl(\boldsymbol F^*(\bm U_L, \bm U_R)\cdot\bm n\bigr)
-
(\bm \psi_R - \bm \psi_L)\cdot\bm n,
\, 1
\right\rangle_{E,h},
\end{equation*}
where $\mathcal{E}_h^o$ denotes the set of interior interfaces of the mesh $\Omega_h$.
Here, $\bm U_{L/R}$, $\bm V_{L/R}$, and $\bm \psi_{L/R}$ denote the left and right traces of the conservative variables, entropy variables, and entropy potentials, respectively, on an interface $E$.

By the entropy stability assumption \eqref{eq:esf} on the numerical flux, each interface contribution on the right-hand side is non-positive. Consequently, the global discrete entropy inequality \eqref{eq_ssg} holds, which completes the proof.
\end{proof}
\begin{remark}[Entropy stability of Lax--Friedrichs flux]
The local Lax--Friedrichs flux \eqref{eq:llf_normal_flux} is entropy stable provided that the dissipation parameter $\alpha$ is no smaller than the maximum wave speed of the exact Riemann problem at the interface; see \cite{chen2017entropy} for a detailed discussion.
Since computing the exact wave speed is nontrivial \cite{GuermondPopov16}, we instead adopt the practical wave-speed estimate given in \eqref{alpha-x}.
While this choice cannot be theoretically proven to be entropy stable in the sense of \eqref{eq:esf}, it has been observed to behave in an entropy-stable manner in practice.
\end{remark}

We write the semidiscrete entropy-stable DGSEM formulation \eqref{eq_esdgsem} in the abstract form
\begin{equation}
\label{eq_semi_dis}
\frac{d\bm U_h}{dt}
=
\bm L_h(\bm U_h),
\end{equation}
where the solution \(\bm U_h \in [V_h]^4\), and \(V_h\) is the global discontinuous finite element space defined by
\[
V_h := \left\{
v \in L^2(\Omega) \;:\;
\left. v \right|_{\Omega^e} \circ \bm \Phi^e \in \mathbb{Q}_N(\Omega_{\mathrm{ref}}),
\quad \forall \Omega^e \in \Omega_h
\right\}.
\]
Here, \(\bm \Phi^e : \Omega_{\mathrm{ref}} \to \Omega^e\) denotes the element mapping. The operator \(\bm L_h(\cdot)\) represents the spatial discretization defined by the entropy-stable DGSEM formulation.

To advance the semidiscrete system \eqref{eq_semi_dis} in time, we employ the
third-order strong stability preserving Runge--Kutta (SSP-RK3) method
\cite{gottlieb2001strong}.
Given the numerical solution $\bm U_h^{n}$ at time level $t^n$, the solution
at the next time level $t^{n+1}=t^n+\Delta t$ is obtained through the following
stages:
\begin{equation}
\label{eq_RK3}
\begin{aligned}
& \bm U_h^{(1)}
=
\bm U_h^{n}
+
\Delta t\,\bm L_h(\bm U_h^{n}), \\
& \bm U_h^{(2)}
=
\dfrac{3}{4}\bm U_h^{n}
+
\dfrac{1}{4}\bm U_h^{(1)}
+
\dfrac{\Delta t}{4}\bm L_h(\bm U_h^{(1)}), \\
& \bm U_h^{(3)}
=
\dfrac{1}{3}\bm U_h^{n}
+
\dfrac{2}{3}\bm U_h^{(2)}
+
\dfrac{2\Delta t}{3}\bm L_h(\bm U_h^{(2)}), \\
& \bm U_h^{n+1}
=
\bm U_h^{(3)}.
\end{aligned}
\end{equation}

\begin{remark}[Positivity preservation]
Due to the use of the Lax--Friedrichs numerical flux \eqref{eq:llf_normal_flux} in the scheme \eqref{eq_esdgsem}, it can be shown that each stage of the above Runge--Kutta method preserves the positivity of the cell averages of density and pressure under a standard CFL-type time step restriction, following the work of Zhang and Shu \cite{zhang2010positivity}; see also \cite{chen2017entropy}. Consequently, a positivity-preserving scaling limiter \cite{zhang2010positivity} can be applied to enforce the positivity of density and pressure at all nodal degrees of freedom.
\end{remark}

\section{Oscillation-Eliminating DG (OEDG)}
\label{sec:oedg}

As discussed in the introduction, high-order discontinuous Galerkin
approximations may exhibit spurious oscillations in the vicinity of
discontinuities. To mitigate these nonphysical oscillations while
preserving high-order accuracy in smooth regions, we use the oscillation-eliminating (OE) procedure applied in a post-processing fashion \cite{peng2025oedg}.

Let 
$$\mathbb Q_N(\Omega_{e}) := \{v\in L^2(\Omega_e): v\circ \bm \Phi^e \in \mathbb Q_N(\Omega_{ref}) \} $$ denote the mapped tensor-product
polynomial space of degree $N$ on the physical element $\Omega_e$.
We define an element-wise linear operator
\[
\mathcal F_\tau^e : \mathbb Q_N(\Omega_e)
\longrightarrow \mathbb Q_N(\Omega_e),
\]
which acts on each component of the DG solution independently.
For a solution vector $\bm U_h^e$ restricted to an element
$\Omega_e$, we denote by
$U_{\sigma}^{e,k}(\hat t) = \mathcal F_{\hat t}^e(U_h^{e,k})$
the OE-processed solution associated with the $k$-th component,
$k = 1,\dots,4$, where $\hat t$ is a pseudo time.

The operator $\mathcal F_{\hat t}^e$ is defined as the solution of the
following initial value problem: for any element $\Omega_e$ and each
component $k=1,\dots,4$,
\begin{equation}
\label{eq:oe_ode}
\left\{
\begin{aligned}
&\frac{d}{d\hat t}
\int_{\Omega_e} U_{\sigma}^{e,k} \, v \, d\bm x
+ \sum_{m=0}^{N}
\delta_m^{(e)}(\bm U_h^e)
\int_{\Omega_e}
\bigl(
U_{\sigma}^{e,k} - P^{m-1} U_{\sigma}^{e,k}
\bigr) v \, d\bm x
= 0,
\quad \forall v \in \mathbb Q_N(\Omega_e), \\
& U_{\sigma}^{e,k}(\hat t = 0) = U_h^{e,k}.
\end{aligned}
\right.
\end{equation}

Here $\{\delta_m^{(e)}(\bm U_h^e)\}_{m=0}^N$ are nonnegative
strength factors that control the amount of modal damping at different
polynomial levels; their precise definition will be given later.
$P^m w\in \mathbb Q_m(\Omega_e)$ denotes the $L^2$ projection onto the polynomial
space $\mathbb Q_m(\Omega_e)$ for $0 \le m \le N-1$, defined by
\begin{equation}
\label{eq:oe_projection}
\begin{aligned}
% &\text{for any } w, \quad P^m w \in \mathbb Q_m(\Omega_e)
% \text{ satisfies} \\
&\int_{\Omega_e} (P^m w - w)\, v \, d\bm x = 0,
\quad \forall v \in \mathbb Q_m(\Omega_e),
\end{aligned}
\end{equation}
and we set $P^{-1} := P^0$ for notational convenience.

The evolution problem \eqref{eq:oe_ode} is linear and can be solved
exactly, making the OE procedure computationally efficient.
The specific implementation of the OE operator differs between
Cartesian meshes and general curvilinear meshes due to the presence
of metric terms. These two cases will be discussed separately in the
following subsections.
\subsection{OEDG on Cartesian Grids}

We first consider the implementation of the OE procedure on Cartesian
meshes, where orthogonal modal bases can be exploited to obtain a
closed-form expression of the operator $\mathcal F_{\hat t}$.
The presentation follows the general framework of
\cite{peng2025oedg}, with several modifications tailored to the
Cartesian setting.

Let $\{\psi_i(\xi)\}_{i=0}^N$ denote a set of normalized orthogonal
polynomials on $[-1,1]$, such as normalized Legendre polynomials.
A tensor-product orthogonal basis on the reference element
$\Omega_{\mathrm{ref}}=[-1,1]^2$ is then given by
\[
\Psi_{i,j}(\xi,\eta) := \psi_i(\xi)\psi_j(\eta),
\qquad 0 \le i,j \le N.
\]

Using this modal basis, the OE-processed solution
$U_{\sigma}^{e,k}$ defined by \eqref{eq:oe_ode} can be expanded as
\begin{equation}
\label{eq:oe_modal_expansion}
U_{\sigma}^{e,k}(\hat t,\xi,\eta)
=
\sum_{i,j=0}^N
(\hat U_{\sigma}^{e,k})_{i,j}(\hat t)\,
\Psi_{i,j}\bigl(\xi,\eta\bigr),
\end{equation}
where the expansion is performed on the reference element, and the
coordinate mapping $(x^e(\xi),y^e(\eta))$ is omitted for brevity.

Due to the hierarchical structure of orthogonal polynomial spaces,
the difference between $U_{\sigma}^{e,k}$ and its $L^2$ projection
onto $\mathbb Q_{m-1}$ admits the representation
\begin{equation}
\label{eq:oe_modal_truncation}
U_{\sigma}^{e,k} - P^{m-1} U_{\sigma}^{e,k}
=
\sum_{\max(i,j)\ge \max\{m,1\}}
(\hat U_{\sigma}^{e,k})_{i,j}(\hat t)\,
\Psi_{i,j},
\qquad m \ge 1.
\end{equation}

Substituting \eqref{eq:oe_modal_truncation} into
\eqref{eq:oe_ode} and choosing the test function
$v=\Psi_{i,j}$ yields a decoupled system of ordinary differential
equations for the modal coefficients:
\begin{equation}
\label{eq:oe_modal_ode}
\begin{aligned}
&\frac{d}{d\hat t}(\hat U_{\sigma}^{e,k})_{i,j}
+ \Bigl(\sum_{m=0}^{\max(i,j)} \delta_m^{(e)}\Bigr)
(\hat U_{\sigma}^{e,k})_{i,j}
= 0,
\qquad \forall\, \max(i,j)\ge 1, \\
&\frac{d}{d\hat t}(\hat U_{\sigma}^{e,k})_{0,0} = 0.
\end{aligned}
\end{equation}

Evaluating the solution at pseudo time $\hat t=\Delta t$, the above
system can be solved exactly as
\begin{equation}
\label{eq:oe_coeff_cart}
\begin{aligned}
(\hat U_{\sigma}^{e,k})_{i,j}(\Delta t)
&=
\exp\!\Bigl(
-\Delta t \sum_{m=0}^{\max(i,j)} \delta_m^{(e)}
\Bigr)
(\hat U_{h}^{e,k})_{i,j},
\qquad \max(i,j)\ge 1, \\
(\hat U_{\sigma}^{e,k})_{0,0}(\Delta t)
&=
(\hat U_{h}^{e,k})_{0,0}.
\end{aligned}
\end{equation}

Consequently, the OE operator on Cartesian grids admits the explicit
representation
\begin{equation}
\label{eq:oe_operator_cart}
\mathcal F_{\Delta t}(U_h^{e,k})
=
(\hat U_h^{e,k})_{0,0}\,\Psi_{0,0}
+
\sum_{\max(i,j)\ge 1}
\exp\!\Bigl(
-\Delta t \sum_{m=0}^{\max(i,j)} \delta_m^{(e)}
\Bigr)
(\hat U_h^{e,k})_{i,j}\,\Psi_{i,j}.
\end{equation}

We define the damping coefficients $\delta_m^{(e)}$ by
\begin{equation}
\label{eq:oe_delta_cart}
\delta_m^{(e)}
=
\frac{\beta_e}{h_e}
\sum_{f\in\partial\Omega_e}
\sigma_m^{(f)}(\bm U_h^e),
\end{equation}
where $\beta_e$ is an estimate of the local maximum wave speed on
$\Omega_e$, and $h_e$ denotes a characteristic element length scale,
chosen as the radius of the inscribed circle.

For the compressible Euler equations, we take
$\beta_e = \|\bar{\bm u}\| + \bar c$, where $\bar{\bm u}$ and $\bar c$
are the cell-averaged velocity and speed of sound, respectively.
The face-based indicator is defined as
\begin{equation}
\label{eq:oe_sigma_face}
\sigma_m^{(f)}(\bm U_h^e)
=
\max_{1\le k\le d}
\sigma_m^{(f)}(U_h^{e,k}),
\qquad f\in\partial\Omega_e,
\end{equation}
with
\begin{equation}
\label{eq:oe_damping_cart}
\sigma_m^{(f)}(U_h^{e,k})
=
\begin{cases}
0,
& U_h^{e,k} = \bar U_h^{e,k}, \\[1.2ex]
\dfrac{(2m+1)h_e^{m}}{2(2N-1)m!}
\Biggl[
\displaystyle
\sum_{n\le m}
\dfrac{
\frac{1}{|f|}
\int_f
\bigl|
\llbracket \partial_x^n U_h^{e,k} \rrbracket_f
\bigr|
\, dS
}{
\|U_h^{e,k}-\bar U_h^{e,k}\|_{L^\infty(\Omega_e)}
}
+
\sum_{n\le m}
\dfrac{
\frac{1}{|f|}
\int_f
\bigl|
\llbracket \partial_y^n U_h^{e,k} \rrbracket_f
\bigr|
\, dS
}{
\|U_h^{e,k}-\bar U_h^{e,k}\|_{L^\infty(\Omega_e)}
}
\Biggr],
& \text{otherwise}.
\end{cases}
\end{equation}

\begin{remark}[OE scale factor $s$] 
In practice, the damping induced by
\eqref{eq:oe_coeff_cart} may be overly strong for certain problems.
To improve accuracy, we introduce a scaling factor $s\in(0,1)$ and
replace the exponential factor in \eqref{eq:oe_coeff_cart} by
\begin{equation}
\label{eq:oe scale factor}
   \exp\!\Bigl(
- s\,\Delta t \sum_{m=0}^{\max(i,j)} \delta_m^{(e)}
\Bigr). 
\end{equation}
and we will call it OE scale factor $s$ in the following.
\end{remark}

\begin{remark}[Conservation of OE procedure]
Since the modal coefficient $(\hat U_h^{e,k})_{0,0}$ is left unchanged,
the OE procedure preserves the cell average of each conserved
variable, which is a crucial property for stability and conservation.
\end{remark}

\begin{remark}[Modifications tailored to the Cartesian setting]
Unlike \cite{peng2025oedg}, where directional derivatives in all
orientations are employed, the Cartesian grid structure allows us to
restrict \eqref{eq:oe_damping_cart} to derivatives aligned with the
coordinate directions.
\end{remark}

\subsection{OEDG on Curvilinear Grids}
\subsubsection{Exact solution of \eqref{eq:oe_ode}}

On general curvilinear meshes, constructing a local orthogonal modal
basis on each element $\Omega_e$ is impractical due to the metric terms. Nevertheless, the
projection operators $P^m$ defined in \eqref{eq:oe_projection} remain
well-defined on $\Omega_e$. We approximate the integrals in \eqref{eq:oe_projection}
using the LGL rule of degree $N$ \eqref{eq_quad_rule}.
We therefore derive an explicit expression
for the OE operator $\mathcal F_{\Delta t}^e$ using only the nested
structure of polynomial spaces.

Recall the hierarchy
\[
\mathbb Q_0(\Omega_e)\subset \mathbb Q_1(\Omega_e)\subset \cdots \subset \mathbb Q_N(\Omega_e).
\]
For any $u\in \mathbb Q_N(\Omega_e)$, define the incremental
(projection-difference) components
\begin{equation}
\label{eq:oe_curve_decomp}
u_0 := P^0 u, 
\qquad
\Delta_k := P^k u - P^{k-1}u,\quad k=1,\dots,N,
\end{equation}
so that the telescoping decomposition holds:
\begin{equation}
\label{eq:oe_curve_telescoping}
u = u_0 + \sum_{k=1}^N \Delta_k.
\end{equation}
Note that $\Delta_k \in \mathbb Q_k(\Omega_e)$ and
$P^{k-1}\Delta_k = 0$, hence $\Delta_k$ lies in the complement
$\mathbb W_k := \mathbb Q_k \ominus \mathbb Q_{k-1}$ in the $L^2$
sense. The following  lemmas summary some useful properties of $\Delta_k$.

\begin{lemma}
\label{lemma:oe_orthogonality}
For each $k=1,\dots,N$, the increment $\Delta_k$ is $L^2$-orthogonal
to $\mathbb Q_{k-1}(\Omega_e)$, i.e.,
\[
\int_{\Omega_e} \Delta_k\, v \, d\bm x = 0,
\qquad \forall v\in \mathbb Q_{k-1}(\Omega_e).
\]
\end{lemma}

\begin{lemma}
\label{lemma:oe_projection_action}
For $k=1,\dots,N$ and any $m\ge 0$,
\[
(I-P^{m-1})\Delta_k =
\begin{cases}
0, & m-1 \ge k,\\
\Delta_k, & m \le k,
\end{cases}
\qquad\text{and}\qquad
(I-P^{m-1})u_0 = 0.
\]
\end{lemma}

\begin{proof}
Since $\Delta_k = P^k u - P^{k-1}u$ and $P^{m-1}$ is a projection,
if $m-1\ge k$ then $P^{m-1}P^k u=P^k u$ and $P^{m-1}P^{k-1}u=P^{k-1}u$,
hence $P^{m-1}\Delta_k=\Delta_k$ and $(I-P^{m-1})\Delta_k=0$.

If $m\le k$, then $P^{m-1}P^k u=P^{m-1}u=P^{m-1}P^{k-1}u$, so
$P^{m-1}\Delta_k=0$ and $(I-P^{m-1})\Delta_k=\Delta_k$.

Finally, $u_0\in\mathbb Q_0$ implies $P^{m-1}u_0=u_0$ for all $m$,
hence $(I-P^{m-1})u_0=0$.
\end{proof}

\medskip
We now derive the explicit solution of the OE evolution
\eqref{eq:oe_ode}. Consider a scalar component on $\Omega_e$ and write
$u_\sigma(\hat t)=\mathcal F_{\hat t}(u)$.
Substituting \eqref{eq:oe_curve_telescoping} into \eqref{eq:oe_ode}
and using Lemma~\ref{lemma:oe_projection_action} gives, for all
$v\in\mathbb Q_N(\Omega_e)$,
\begin{equation}
\label{eq:oe_curve_weak_simplified}
\frac{d}{d\hat t}\int_{\Omega_e}
\Bigl(u_0+\sum_{k=1}^N \Delta_k(\hat t)\Bigr)v\,d\bm x
+
\sum_{k=1}^N
\Bigl(\sum_{m=0}^k \delta_m^{(e)}(u)\Bigr)
\int_{\Omega_e}\Delta_k(\hat t)\,v\,d\bm x
=0.
\end{equation}

The key observation is that \eqref{eq:oe_curve_weak_simplified}
decouples across the nested subspaces. Fix $\ell\in\{0,1,\dots,N\}$ and
choose $v\in\mathbb Q_\ell(\Omega_e)$. By
Lemma~\ref{lemma:oe_orthogonality}, we have
$\int_{\Omega_e}\Delta_k v\,d\bm x=0$ for all $k\ge \ell+1$, hence
\eqref{eq:oe_curve_weak_simplified} reduces to
\begin{equation}
\label{eq:oe_curve_restricted}
\frac{d}{d\hat t}\int_{\Omega_e}
\Bigl(u_0+\sum_{k=1}^{\ell} \Delta_k(\hat t)\Bigr)v\,d\bm x
+
\sum_{k=1}^{\ell}
\Bigl(\sum_{m=0}^k \delta_m^{(e)}(u)\Bigr)
\int_{\Omega_e}\Delta_k(\hat t)\,v\,d\bm x
=0,
\quad \forall v\in\mathbb Q_\ell(\Omega_e).
\end{equation}

Subtracting the same identity written for $\ell-1$ from that for $\ell$
and using the fact that $\mathbb Q_{\ell-1}\subset\mathbb Q_\ell$
yields an evolution equation on the increment subspace
$\mathbb W_\ell$. Equivalently, \eqref{eq:oe_ode} induces the
following \emph{strong} ODE for each increment $\Delta_\ell$:
\begin{equation}
\label{eq:oe_delta_ode}
\frac{d}{d\hat t}\Delta_\ell(\hat t)
+
\Bigl(\sum_{m=0}^{\ell}\delta_m^{(e)}(u)\Bigr)
\Delta_\ell(\hat t)=0,
\qquad \ell=1,\dots,N,
\end{equation}
together with $\frac{d}{d\hat t}u_0=0$.
Therefore,
\begin{equation}
\label{eq:oe_curve_solution}
u_\sigma(\Delta t)
=
u_0
+
\sum_{k=1}^N
\exp\!\Bigl(
-\Delta t\sum_{m=0}^{k}\delta_m^{(e)}(u)
\Bigr)\,
\Delta_k.
\end{equation}
In particular, the constant component $u_0=P^0u$ (hence the cell
average) is preserved exactly.

Finally, for the Euler system on $\Omega_e$, the operator is applied
component-wise. Denoting $\bm U_\sigma^e=\mathcal F_{\Delta t}(\bm U_h^e)$,
we have for each component $r=1,\dots,4$,
\begin{equation}
\label{eq:oe_curve_apply_damp}
U_{\sigma}^{e,r}
=
P^0(U_h^{e,r})
+
\sum_{k=1}^N
\exp\!\Bigl(
-\Delta t\sum_{m=0}^{k}\delta_m^{(e)}(\bm U_h^e)
\Bigr)\,
\bigl(P^k(U_h^{e,r})-P^{k-1}(U_h^{e,r})\bigr).
\end{equation}

\subsubsection{Damping coefficients}

The damping coefficients $\delta_m^{(e)}$ appearing in
\eqref{eq:oe_curve_apply_damp} are constructed from face-based jump
indicators. As in the Cartesian case, we first define for each face
$f\subset\partial\Omega_e$,
\begin{equation}
\label{eq:oe_sigma_face_curve}
\sigma_m^{(f)}(\bm U_h^e)
=
\max_{1\le r\le 4}\sigma_m^{(f)}(U_h^{e,r}),
\qquad f\in\partial\Omega_e,
\end{equation}
where $U_h^{e,r}$ denotes the $r$-th component of the Euler state.

For each scalar component $U_h^{e,r}$, we set
\begin{equation}
\label{eq:oe_damping_coeff_curve}
\sigma_m^{(f)}(U_h^{e,r})
=
\begin{cases}
0,
& U_h^{e,r} = \bar U_h^{e,r}, \\[1.2ex]
\dfrac{(2m+1)\,h_e^{m}}{2(2N-1)\,m!}
\displaystyle\sum_{|\bm\alpha|\le m}
\dfrac{
\frac{1}{|f|}
\int_{f}
\Bigl|\llbracket \partial^{\bm\alpha}U_h^{e,r}\rrbracket_f\Bigr|
\, dS
}{
\|U_h^{e,r}-\bar U_h^{e,r}\|_{L^{\infty}(\Omega_e)}
},
& \text{otherwise},
\end{cases}
\end{equation}
where $\bm\alpha=(\alpha_1,\alpha_2)$ is a multi-index with
$|\bm\alpha|=\alpha_1+\alpha_2$, and
\[
\partial^{\bm\alpha}U
:=
\frac{\partial^{|\bm\alpha|}U}{\partial x^{\alpha_1}\,\partial y^{\alpha_2}}.
\]
All remaining quantities (jump operator $\llbracket\cdot\rrbracket_f$,
face measure $|f|$, and element length scale $h_e$) follow the
definitions in \eqref{eq:oe_damping_cart}.

\subsubsection{Shock indicator and selective application of OE}

On curvilinear meshes, evaluating \eqref{eq:oe_damping_coeff_curve} can
be computationally expensive because it requires mixed derivatives up to
order $m$ on every face and for every state component. In our
experiments, the OE stage may account for a substantial fraction of the
overall runtime. To reduce the cost, we apply OE only on
\emph{troubled elements} detected by a shock indicator.

Among many available indicators
\cite{krivodonova2004shock,fu2017new,qiu2005comparison}, we adopt a
low-cost choice derived from the zero-order jump measure in
\eqref{eq:oe_damping_coeff_curve}. Specifically, for each element
$\Omega_e$ we define
\begin{equation}
\label{eq:oe_indicator}
\mathcal I^{(e)}(\bm U_h^e)
=
\sum_{f\in\partial\Omega_e}\sigma_0^{(f)}(\bm U_h^e),
\end{equation}
with $\sigma_0^{(f)}(\bm U_h^e)$ given by \eqref{eq:oe_sigma_face_curve}
at $m=0$, i.e.,
\begin{equation}
\label{eq:oe_indicator_explicit}
\sigma_0^{(f)}(U_h^{e,r})
=
\begin{cases}
0,
& U_h^{e,r} = \bar U_h^{e,r}, \\[1.2ex]
\dfrac{1}{2(2N-1)}
\dfrac{
\frac{1}{|f|}\int_{f}\bigl|\llbracket U_h^{e,r}\rrbracket_f\bigr|\,dS
}{
\|U_h^{e,r}-\bar U_h^{e,r}\|_{L^{\infty}(\Omega_e)}
},
& \text{otherwise}.
\end{cases}
\end{equation}
Elements satisfying the following conditions are marked as troubled and subjected to the full OE operator:
\begin{equation}
\label{eq:oe C}
   \mathcal I^{(e)}(\bm U_h^e)>C 
\end{equation}
where constant $C$ is named as shock indicator threshold in the following. Otherwise we set
$\bm U_\sigma^e=\bm U_h^e$.

\begin{algorithm}[H]\small
  \DontPrintSemicolon
  \SetKwInOut{Input}{Input}\SetKwInOut{Output}{Output}

  \Input{DG solution $\{\bm U_h^e\}_e$; threshold constant $C>0$.}
  \Output{OE-processed solution $\{\bm U_{\sigma}^e\}_e$.}
  \BlankLine

  \ForEach{element $\Omega_e$}{
    Compute $\mathcal I^{(e)}(\bm U_h^e)$ from \eqref{eq:oe_indicator}--\eqref{eq:oe_indicator_explicit}\;
    \eIf{$\mathcal I^{(e)}(\bm U_h^e) > C$}{
      Compute $\sigma_m^{(f)}(\bm U_h^e)$ for $m=1,\dots,N$ via \eqref{eq:oe_damping_coeff_curve}\;
      Form $\delta_m^{(e)}$ and apply \eqref{eq:oe_curve_apply_damp} to obtain $\bm U_\sigma^e$\;
    }{
      Set $\bm U_\sigma^e \leftarrow \bm U_h^e$\;
    }
  }
  \caption{Selective OE procedure on curvilinear meshes using a shock indicator.}
  \label{alg:oe_curve}
\end{algorithm}

\begin{remark}[Comparision with KXRCF indicator]
The indicator \eqref{eq:oe_indicator} is closely related in structure to
jump-based sensors such as the KXRCF indicator \cite{krivodonova2004shock}
and several recent variants \cite{wei2025jump}. This motivates the use
of $\sigma_0^{(f)}$ as a cost-effective troubled-cell detector.
\end{remark}
%%%%%%%%%%%%%%%%%%%%%%%%%%%%%%%%%%%%%%%%%%%%%%%%%%%%%
%%%%%%%%%%%%%%%%%%%%%%%%%%%%%%%%%%%%%%%%%%%%%%%%%%%%%

%%%%%%%%%%%%%%%%%%%%%%%%%%%%%%%%%%%%%%%%%%%%%%%%%%%%%
%%%%%%%%%%%%%%%%%%%%%%%%%%%%%%%%%%%%%%%%%%%%%%%%%%%%%
%%%%%%%%%%%%%%%%%%%%%%%%%%%%%%%%%%%%%%%%%%%%%%%%%%%%%
%%%%%%%%%%%%%%%%%%%%%%%%%%%%%%%%%%%%%%%%%%%%%%%%%%%%%

%%%%%%%%%%%%%%%%%%%%%%%%%%%%%%%%%%%%%%%%%%%%%%%%%%%%%
%%%%%%%%%%%%%%%%%%%%%%%%%%%%%%%%%%%%%%%%%%%%%%%%%%%%%
%%%%%%%%%%%%%%%%%%%%%%%%%%%%%%%%%%%%%%%%%%%%%%%%%%%%%
%%%%%%%%%%%%%%%%%%%%%%%%%%%%%%%%%%%%%%%%%%%%%%%%%%%%%
\section{Numerical Examples}
\label{sec:numerical}

In this section, we use multiple numerical examples to assess the performance of the proposed
entropy-stable DG scheme equipped with the oscillation-eliminating (OE)
procedure on both Cartesian and curvilinear meshes. Unless otherwise
specified, a positivity-preserving limiter
\cite{zhang2010positivity,zhang2017positivity} is applied in all
numerical experiments to guarantee non-negativity of the density and
pressure. The code implementation is based on \texttt{MFEM} library \cite{anderson2021mfem, andrej2024high} for all the simulations.

Time integration is performed using the third-order strong-stability-preserving Runge--Kutta method \eqref{eq_RK3}. The time step size
$\Delta t$ is determined according to the CFL condition
\begin{equation}
\label{eq:cfl}
\Delta t
=
K\,
\min_{\Omega_e} \frac{ h_e}{\bigl(\|\bm u_e\|+c_e\bigr)},
\end{equation}
where $K$ is the CFL constant, 
$h_e$ denotes a characteristic element length scale (taken as the
radius of the inscribed circle of $\Omega_e$), and $\bm u_e$ and $c_e$
are the local velocity vector and speed of sound on element $\Omega_e$,
respectively. In practice, $\bm u_e$ and $c_e$ are evaluated by taking
the maximum over the quadrature points within each element.
% {\color{red}
% For standard DG discretizations on Cartesian meshes, the CFL constant
% $K$ is typically chosen as
% $K=\alpha/(2N+1)$, where $N$ is the polynomial degree and $\alpha$ is a
% problem-dependent constant. In the present setting, however, the use of
% curvilinear elements, the OE procedure, and the positivity-preserving
% limiter precludes an explicit a priori formula for $K$. We therefore
% adopt an adaptive strategy similar to that proposed in
% \cite{zhang2017positivity}. Specifically, an initial CFL constant is
% chosen, and the time step \eqref{eq:cfl} is computed at each stage. If
% the positivity-preserving limiter fails due to the occurrence of
% negative cell averages of density or pressure at any Runge--Kutta
% substep, the current time step is rejected and recomputed using $\Delta t/2$.
% %The CFL constant $K$, the OE scaling factor $s$ in \eqref{eq:oe scale factor}, and the shock indicator threshold $C$ in \eqref{eq:oe_indicator} are problem dependent and will be specified for each test case.
% }
{%\color{blue}
In all the following simulations and tests, we use CFL constant $K=\frac{0.5}{2N+1}$. Unless otherwise specified, the OE scale factor $s$ defined in \eqref{eq:oe scale factor} and shock indicator threshold $C$ defined in \eqref{eq:oe C} are set as:
\begin{equation}
\label{eq:sC}
    s=0.2, \quad C=0.02.
\end{equation}
%[RK:::WE SHALL USE K=0.8 as a starting point, and reduces it if cell-average pp failed.  and s \& C  shall be given reference values. People usually do not like problem dependent constants...  CITE MFEM library]

\subsection{Isentropic Euler vortex problem}

The isentropic Euler vortex problem describes the advection of a smooth
vortex by a constant background velocity field. Owing to its known
analytic solution, this problem serves as a standard benchmark for
assessing the accuracy and robustness of high-order numerical methods
for the compressible Euler equations
\cite{vincent2011insights,shu2006essentially}.

The computational domain is the periodic square
$\Omega=[-10,10]^2$. The exact solution is given by
\begin{equation}
\label{eq:isentropic_vortex}
\begin{aligned}
u(x,y,t)
&= u_b
-\frac{\beta}{2\pi}\bigl(y-v_b t\bigr)
\exp\!\left(
\frac12\Bigl[1-(x-u_b t)^2-(y-v_b t)^2\Bigr]
\right),\\
v(x,y,t)
&= v_b
+\frac{\beta}{2\pi}\bigl(x-u_b t\bigr)
\exp\!\left(
\frac12\Bigl[1-(x-u_b t)^2-(y-v_b t)^2\Bigr]
\right),\\
T(x,y,t)
&=
1-\frac{(\gamma-1)\beta^2}{8\gamma\pi^2}
\exp\!\left(
1-(x-u_b t)^2-(y-v_b t)^2
\right),\\
\rho(x,y,t)
&= T(x,y,t)^{\frac{1}{\gamma-1}},
\qquad
p(x,y,t)
= T(x,y,t)^{\frac{\gamma}{\gamma-1}}.
\end{aligned}
\end{equation}
Here $\gamma=1.4$ is the ratio of specific heats, $\beta=5$ controls the
vortex strength, and $(u_b,v_b)=(1,1)$ denotes the constant background
velocity.

To assess the performance of the proposed method on curved geometries,
we employ a curvilinear mesh in this test. The mesh is generated by
first constructing a uniform Cartesian grid with $N_x=N_y=M$ elements,
and then applying the smooth coordinate transformation
\begin{equation}
\label{eq:curve_mesh}
\begin{aligned}
x' &= x + 0.05\,\sin(\alpha\pi x)\cos(\alpha\pi y),\\
y' &= y + 0.10\,\cos(\alpha\pi x)\sin(\alpha\pi y),
\end{aligned}
\end{equation}
with $\alpha=1.5$.

Figure~\ref{fig:vortex_ic_mesh} illustrates the initial density field
together with the resulting curvilinear mesh.
\begin{figure}[H]
  \centering
  \begin{subfigure}{0.42\textwidth}
    \centering
    \includegraphics[width=\linewidth]{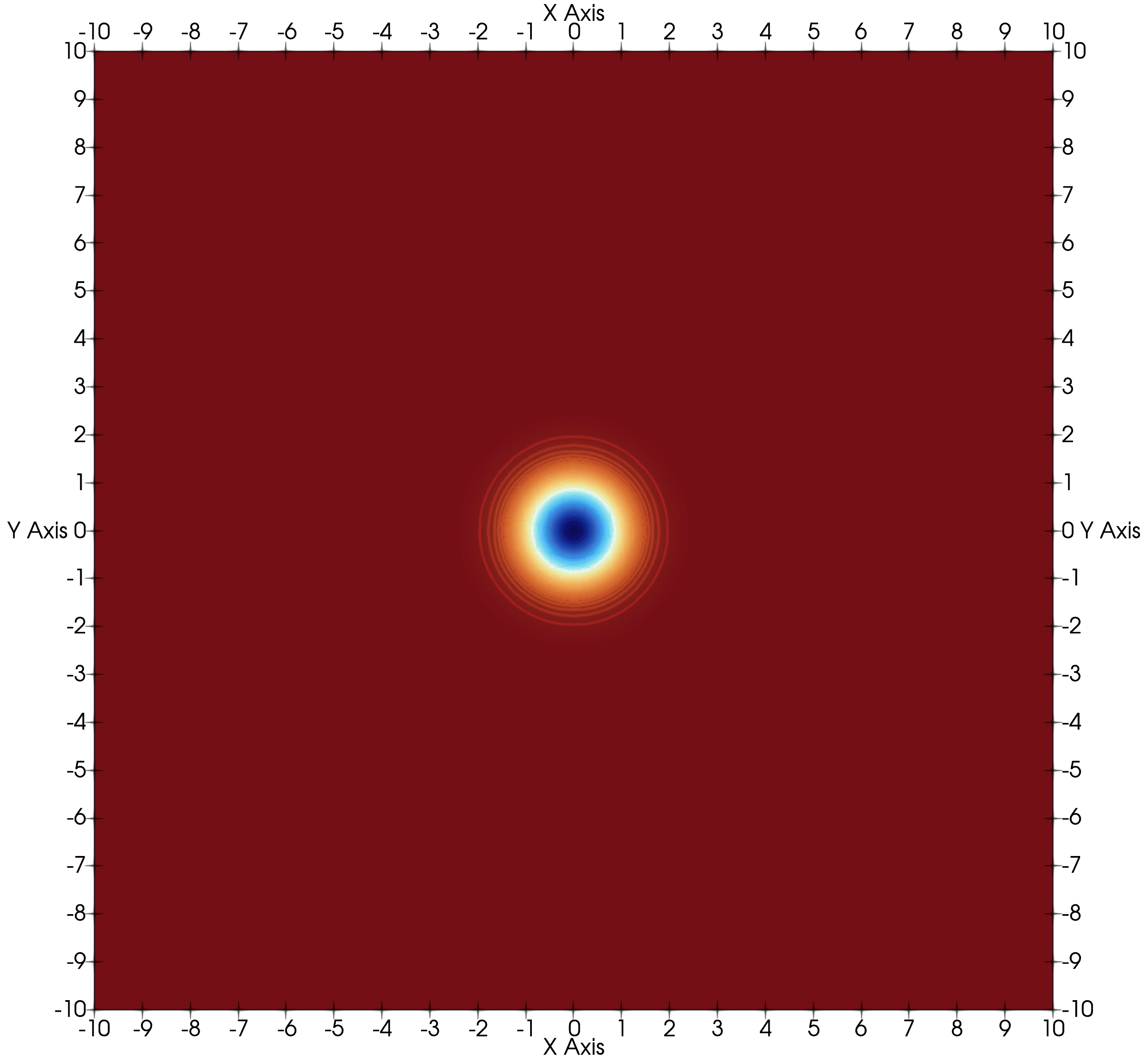}
    \caption{Initial Density Profile }
    %\label{fig:vortex_ic}
  \end{subfigure}\hfill
  \begin{subfigure}{0.42\textwidth}
    \centering
    \includegraphics[width=\linewidth]{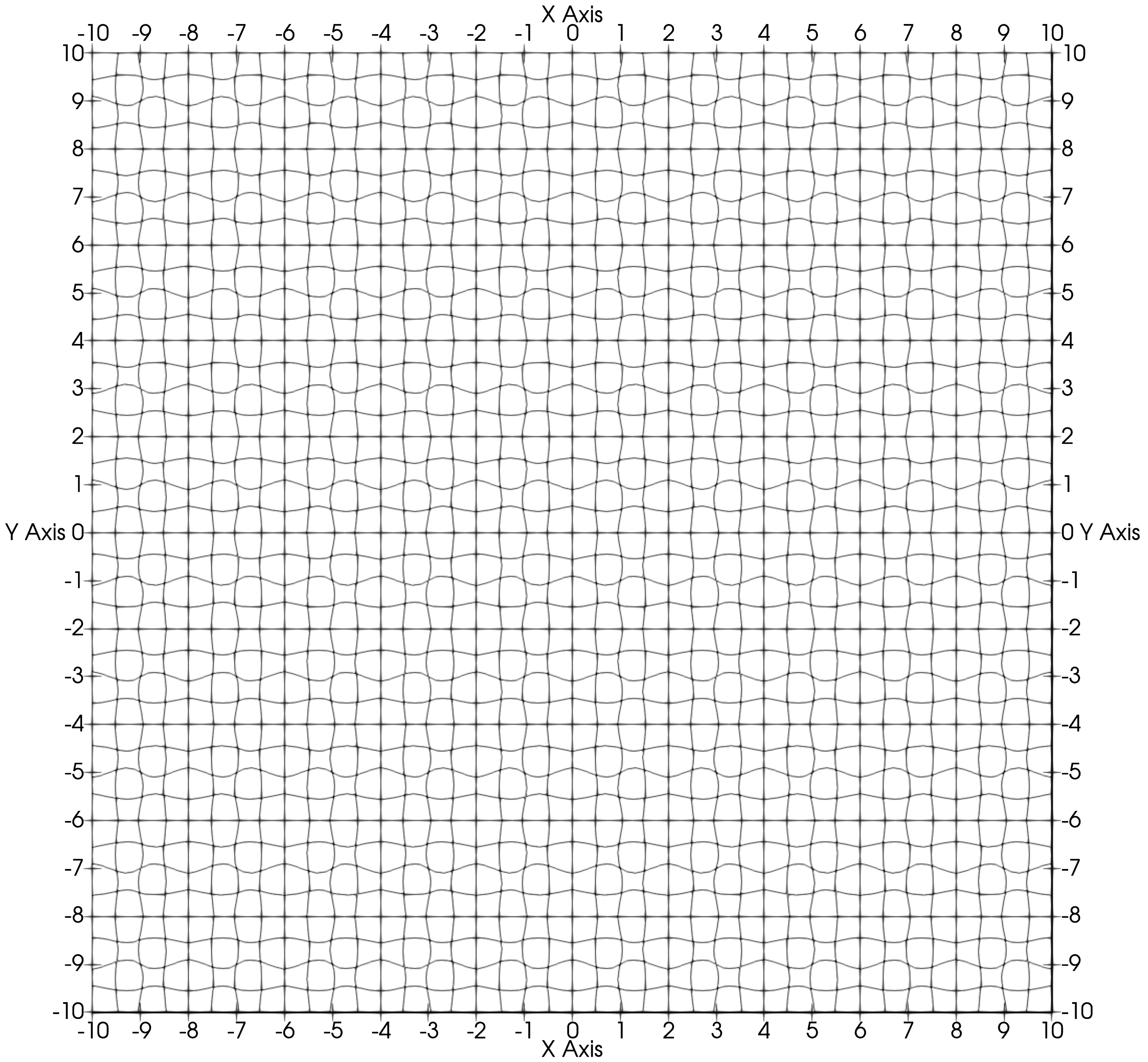}
    \caption{Curvilinear Grids}
    %\label{fig:vortex_mesh}
  \end{subfigure}
  \caption{\textit{Isentropic Euler vortex problem}: initial condition of density and Curvilinear grids for Isentropic Euler Vortex problem. Mesh size: $N_x=N_y=M=40$}
  \label{fig:vortex_ic_mesh}
\end{figure}

To assess the convergence behavior, the simulation is advanced to
$t=20$, corresponding to one full advection period of the vortex across
the domain. The $L^2$ errors of all conservative variables are then
computed for polynomial orders $N=3$ and $N=4$ in
$\mathbb Q_N$. Note that due to smoothness of this test case, the OE procedure is not activated. 
%{\color{blue} We use OE scale factor $s=0.2$ and shock indicator threshold $C=0.02$ in all the tests of this subsection. However, due to the smoothness of this problem, the OE is never activated.}

%{\color{red} RK: what time step size do you use? what is s \& C? Mention that due to problem smoothness, OE is never activated...}

Figure~\ref{fig:vortex_convergence} demonstrates that the proposed
method achieves the expected $(N+1)$-th order convergence rate, i.e.,
$\mathcal O(h^{N+1})$, where the mesh size is defined as
$h=20/M$. This behavior is consistent with the theoretical convergence
results for smooth solutions of the discontinuous Galerkin method
\cite{cockburn1999discontinuous}.
\begin{figure}[H]
  \centering
  \begin{subfigure}{0.48\textwidth}
    \centering
    \includegraphics[width=\linewidth]{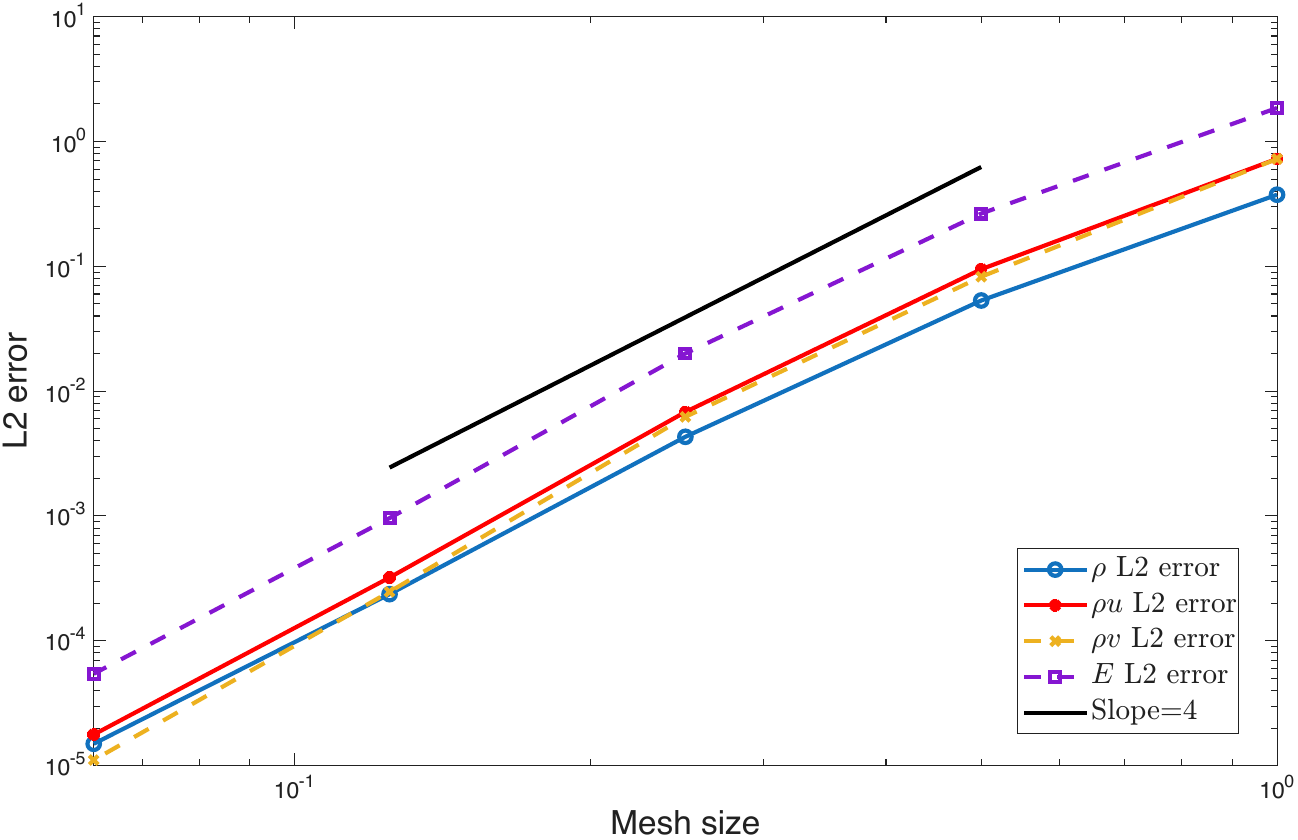}
    \caption{$N=3$ }
    %\label{fig:vortex_ic}
  \end{subfigure}\hfill
  \begin{subfigure}{0.48\textwidth}
    \centering
    \includegraphics[width=\linewidth]{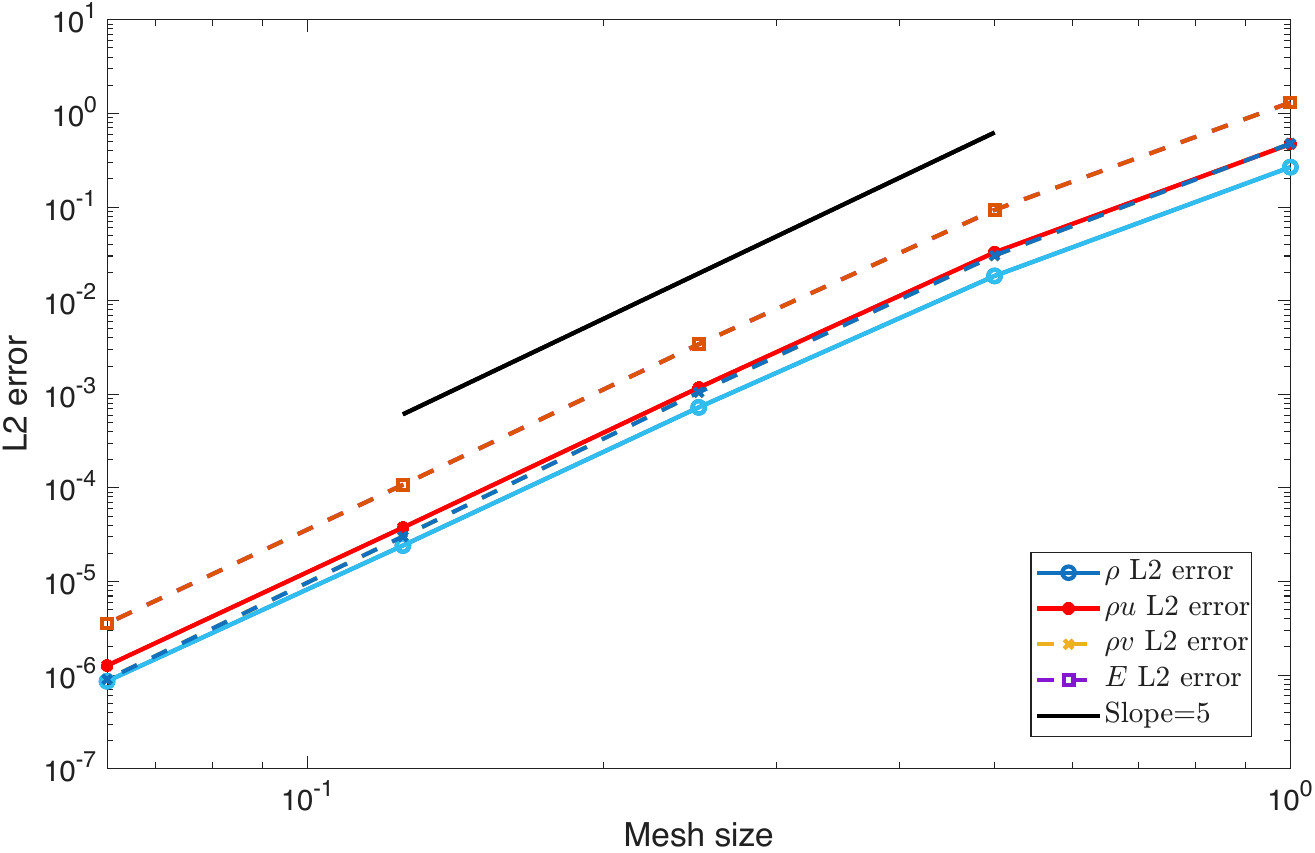}
    \caption{$N=4$}
    %\label{fig:vortex_mesh}
  \end{subfigure}
  \caption{\textit{Isentropic Euler vortex problem}: convergency Rate of Isentropic Euler vortex problem in curvilinear mesh. Left figure has polynomial order of $N=3$, and the polynomial order of right figure is $N=4$ }
  \label{fig:vortex_convergence}
\end{figure}

We further examine the entropy stability of the proposed scheme for
this test case. Specifically, we compute the numerical approximation of
the total entropy
\begin{equation}
\label{eq:vortex_entropy_integral}
I_{\eta}(t)
=
\int_{\Omega}\eta\, d\bm x
=
\sum_{e}\int_{\Omega_e}\eta\, d\bm x,
\end{equation}
and track its evolution in time.

Figure~\ref{fig:vortex_entropy} shows the time history of
$I_{\eta}(t)$. In both cases, the total entropy is observed to be
non-increasing throughout the simulation, indicating that the proposed
DG scheme satisfies a discrete entropy inequality and is therefore
entropy stable.
\begin{figure}[H]
  \centering
  \begin{subfigure}{0.48\textwidth}
    \centering
    \includegraphics[width=\linewidth]{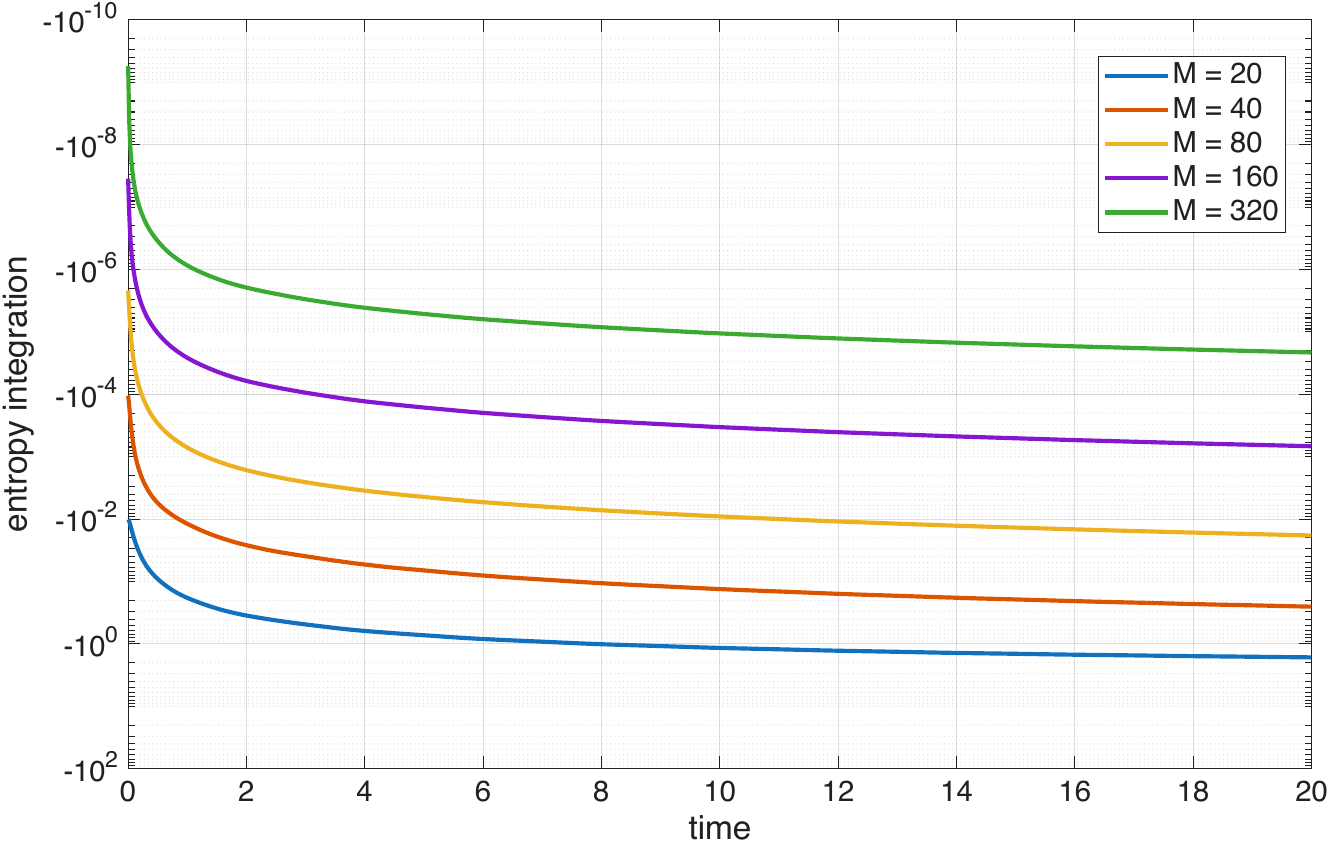}
    \caption{$N=3$ }
    %\label{fig:vortex_ic}
  \end{subfigure}\hfill
  \begin{subfigure}{0.48\textwidth}
    \centering
    \includegraphics[width=\linewidth]{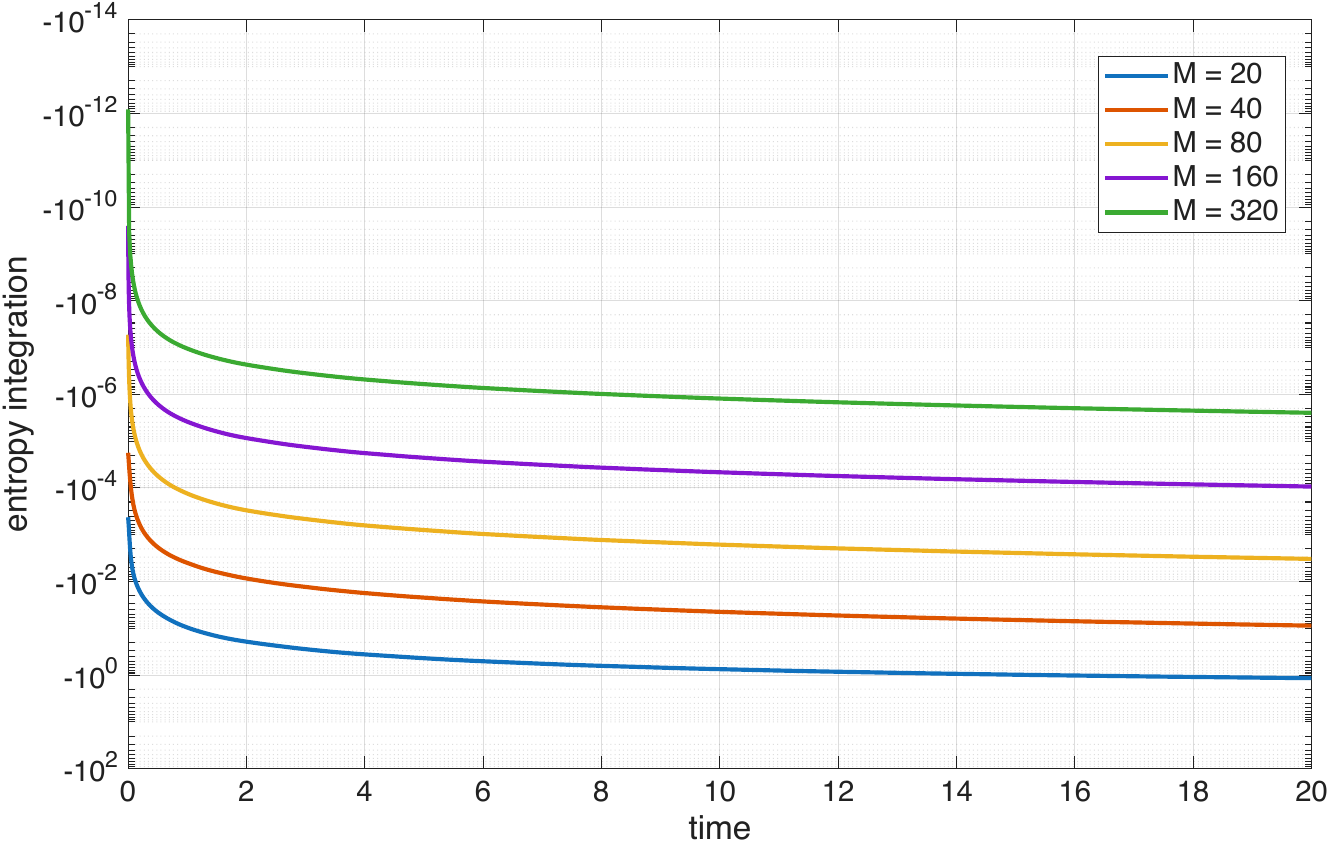}
    \caption{$N=4$}
    %\label{fig:vortex_mesh}
  \end{subfigure}
  \caption{\textit{Isentropic Euler vortex problem}: how entropy integration changes with time. Left figure has polynomial order of $N=3$, and the polynomial order of right figure is $N=4$ }
  \label{fig:vortex_entropy}
\end{figure}

%%%%%%%%%%%%%%%%%%%%%%%%%%%%%%%%%%%%%%%%%%%%%%%%%%%%%
%%%%%%%%%%%%%%%%%%%%%%%%%%%%%%%%%%%%%%%%%%%%%%%%%%%%%
\subsection{Two-dimensional Riemann problems}

We next consider two-dimensional Riemann problems for the compressible
Euler equations. The computational domain $\Omega=[0,1]^2$ is divided
into four quadrants, each of which is assigned a constant initial
state. Depending on the choice of these initial states, a wide variety
of wave interactions and flow structures can be generated
\cite{lax1998solution,schulz1993numerical}. In this work, we focus on
Riemann problems~12 and~13 reported in \cite{lax1998solution}.

In several existing studies, outflow boundary conditions are employed
for these problems \cite{peng2025oedg}. However, in our numerical
experiments, such boundary conditions may induce spurious
oscillations near the domain boundaries. To avoid these effects and to better isolate
the interior wave interactions, we instead impose periodic boundary
conditions as discribed in \cite{guermond2011entropy_nonlinear}.

To accommodate periodic boundaries, the computational domain is
extended to $\Omega=[0,2]^2$, and the initial data are defined through
a periodic extension of the standard Riemann configuration on
$[0,1]^2$. Specifically, the initial condition is given by
\begin{equation}
\label{eq:riemann_ic}
\begin{aligned}
\bm U &= \bm U_1,
& 0<x<0.5,\; 0<y<0.5,\\
\bm U &= \bm U_2,
& 0<x<0.5,\; 0.5<y<1,\\
\bm U &= \bm U_3,
& 0.5<x<1,\; 0<y<0.5,\\
\bm U &= \bm U_4,
& 0.5<x<1,\; 0.5<y<1,\\
\bm U(x,y) &= \bm U(x,2-y),
& 0<x<1,\; 1<y<2,\\
\bm U(x,y) &= \bm U(2-x,y),
& 1<x<2,\; 0<y<1,\\
\bm U(x,y) &= \bm U(2-x,2-y),
& 1<x<2,\; 1<y<2.
\end{aligned}
\end{equation}
Here $\bm U=(\rho,\rho u,\rho v,E)^\top$ denotes the vector of
conservative variables, and $\bm U_1,\bm U_2,\bm U_3,\bm U_4$ are the
prescribed constant states defining the corresponding two-dimensional
Riemann problem on the unit square $[0,1]^2$.

For the following two Riemann problems, computations are performed on
uniform Cartesian grids with $N_x=N_y=320$ elements. In both of the two following two problems, we use the $\mathbb Q_3$ DG approximation on all
element. %The parameters $s,C$ are set as \eqref{eq:sC}.
%{\color{blue} In both of the two following two problems, we use the $\mathbb Q_3$ DG approximation on all elements, the OE scale factor $s=0.2$ and the shock indicator $C=0.02$.}

%{\color{red}Unless otherwise specified, the OE scaling factor in \eqref{eq:oe scale factor} is set to $s=0.1$, and the threshold constant in the shock indicator \eqref{eq:oe C} is chosen as $C=0.05$. WHAT polynomial degree did you use? }
%%%%%%%%%%%%%%%%%%%%%%%%%%%%%%%%%%%%%%%%%%%%%%%%%%%%%
\subsubsection{Two-dimensional Riemann problem: Case 12}

We first consider Riemann problem~12 from \cite{lax1998solution}. The
initial conditions on the unit square $[0,1]^2$ are prescribed as
\begin{equation}
\label{eq:riemann12_ic}
\begin{aligned}
\rho &= 0.8, \quad p=1.0, \quad \bm u=(0,0),
& 0<x<0.5,\; 0<y<0.5,\\
\rho &= 1.0, \quad p=1.0, \quad \bm u=(0.7276,0),
& 0<x<0.5,\; 0.5<y<1,\\
\rho &= 1.0, \quad p=1.0, \quad \bm u=(0,0.7276),
& 0.5<x<1,\; 0<y<0.5,\\
\rho &= 0.5313, \quad p=0.4, \quad \bm u=(0,0),
& 0.5<x<1,\; 0.5<y<1.
\end{aligned}
\end{equation}
These states define the vectors
$\bm U_1,\bm U_2,\bm U_3,$ and $\bm U_4$ in
\eqref{eq:riemann_ic}. The initial condition on the extended domain
$\Omega=[0,2]^2$ is then obtained via the periodic extension described
in \eqref{eq:riemann_ic}.

The simulation is advanced to $t=0.2$. The numerical results for
Riemann problem~12 are shown in Figure~\ref{fig:riemann12}. The results are clipped to the domain $[0,1]^2$. In
Figure~\ref{fig:riemann12}(a), the density profiles demonstrate that the
contact discontinuity in the lower-left quadrant is well resolved,
and the fine-scale flow structures behind the shocks are well captured.
These features are in good agreement with previously reported results
\cite{lax1998solution,schulz1993numerical,kurganov2002solution}.

Figure~\ref{fig:riemann12}(b) displays the distribution of the shock
indicator defined in \eqref{eq:oe_indicator} and
\eqref{eq:oe_indicator_explicit}. The indicator is activated primarily
in regions containing strong shocks, while remaining inactive in
smooth regions. As a result, the OE procedure is applied only where
necessary, leading to a significant reduction in its computational
cost. 
%{\color{red} HOW MUCH COST REDUCED? The computation of damping strength is most costly... my current code will compute all damping coeff.. not just selected troubled cells}
\begin{figure}[htbp]
  \centering
  \begin{subfigure}{0.42\textwidth}
    \centering
    \includegraphics[width=\linewidth]{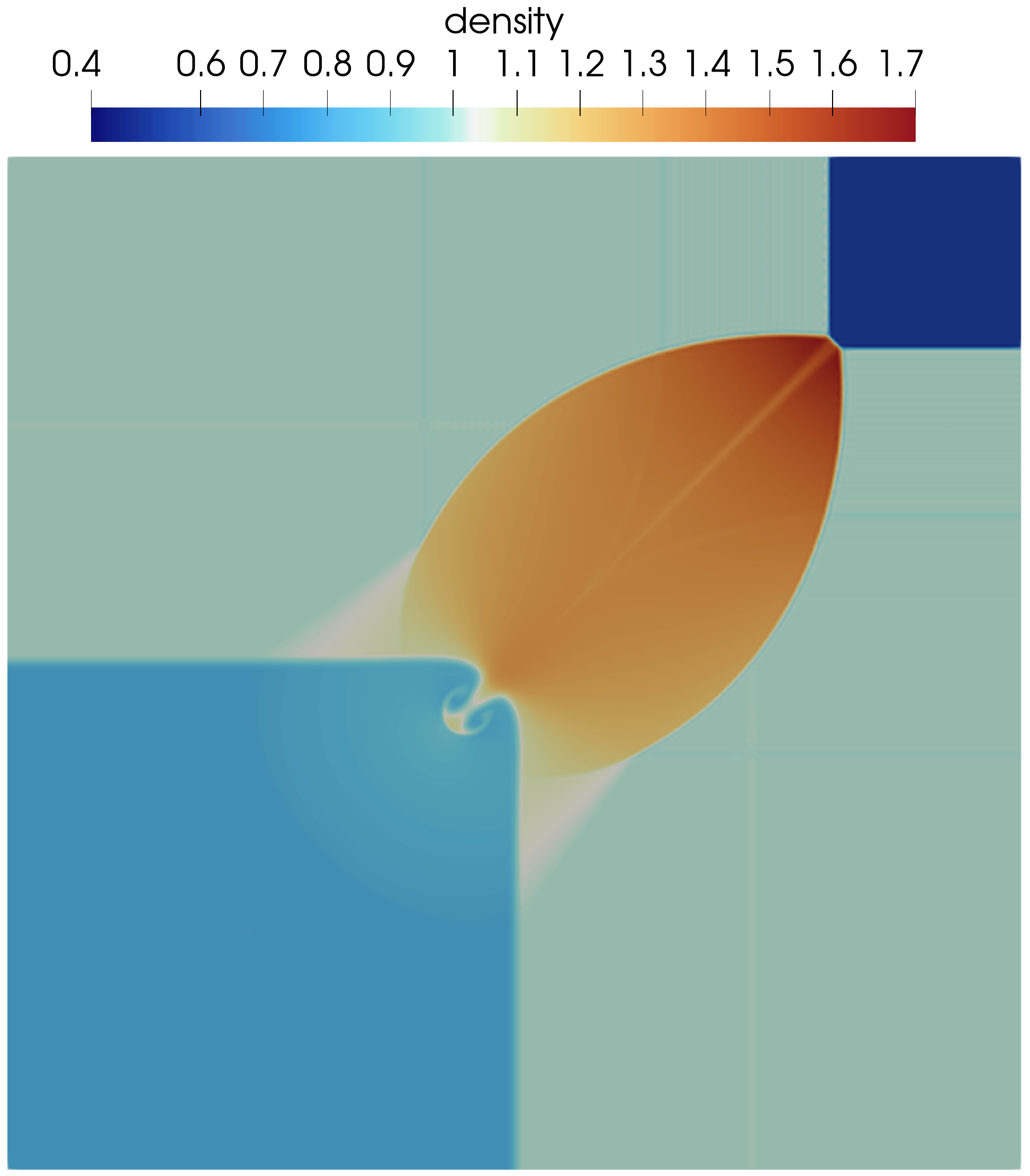}
    \caption{density profile }
    %\label{fig:vortex_ic}
  \end{subfigure}\hfill
  \begin{subfigure}{0.42\textwidth}
    \centering
    \includegraphics[width=\linewidth]{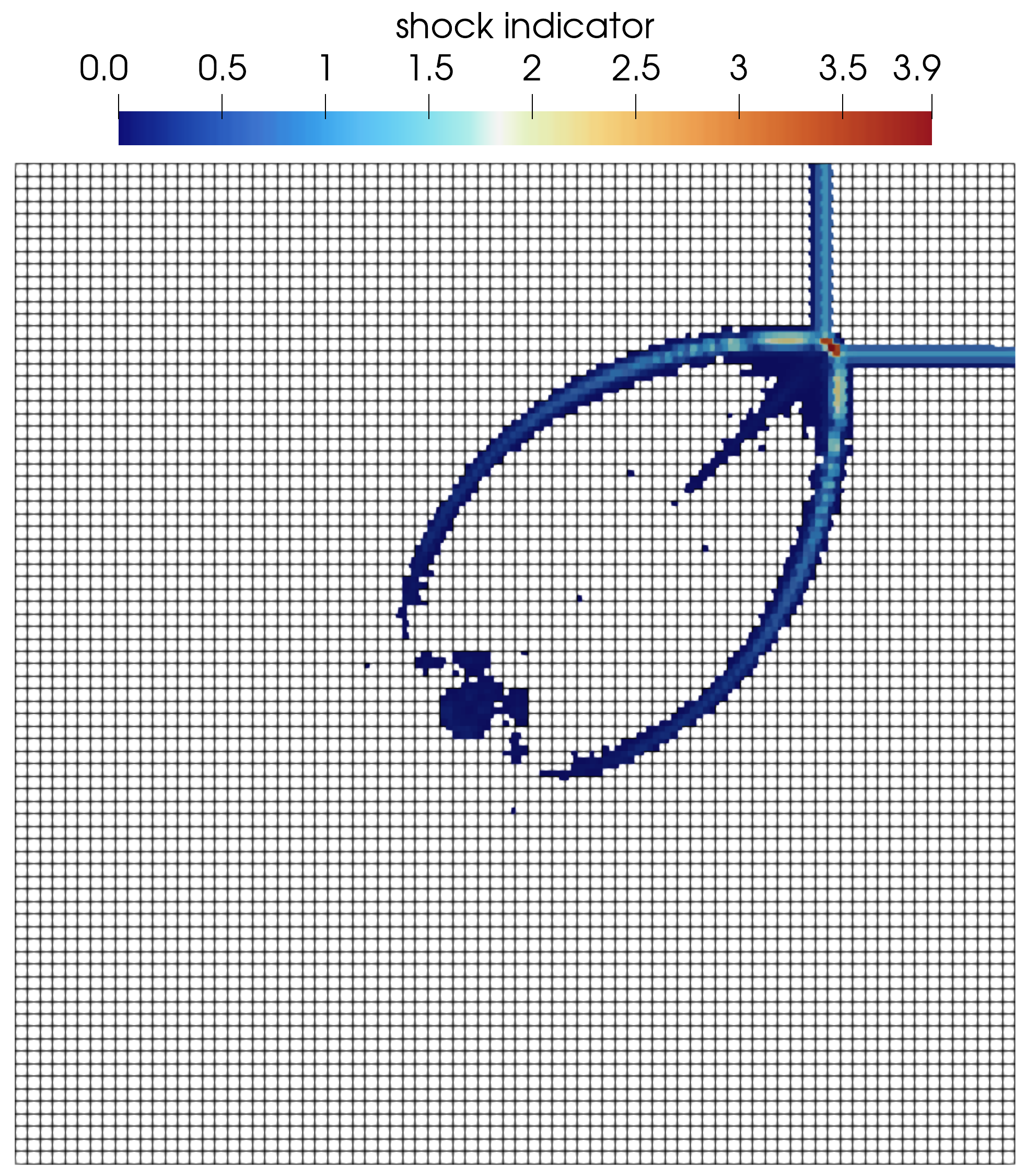}
    \caption{shock indicator}
    %\label{fig:vortex_mesh}
  \end{subfigure}
  \caption{\textit{Two-dimensional Riemann problem 12}: (a) the density profile at $t=0.2$; (b) the shock indicator profile at $t=0.2$.}
  \label{fig:riemann12}
\end{figure}
%%%%%%%%%%%%%%%%%%%%%%%%%%%%%%%%%%%%%%%%%%%%%%%%%%%%%
\subsubsection{Two-dimensional Riemann problem: Case 13}
Then Riemann problem 13 from \cite{lax1998solution} is considered. The initial condition on the unit square $[0,1]^2$ is set as:
\begin{equation}
 \begin{aligned}
     &\rho = 0.8, \quad p= 0.4, \quad \bm u = (0.1,-0.3), \quad  \text{in }\quad 0<x<0.5, \quad 0<y<0.5,\\
     &\rho = 0.5197 \quad p= 0.4, \quad \bm u = (-0.6259,-0.3), \quad  \text{in }\quad 0<x<0.5, \quad 0.5<y<1,\\
     &\rho = 0.5313, \quad p= 0.4, \quad \bm u = (0.1,0.4276), \quad  \text{in }\quad 0.5<x<1, \quad 0<y<0.5,\\
     &\rho = 1, \quad p= 1, \quad \bm u = (0.1,-0.3), \quad  \text{in }\quad 0.5<x<1, \quad 0.5<x<1,\\
 \end{aligned}   
\end{equation}
For this test, we run the simulation to $t=0.3$. Figure~\ref{fig:riemann 13} shows the solution of Riemann problem 13 clipped to the domain $[0,1]^2$. Figure~\ref{fig:riemann 13}(a) shows the density profile, where the shocks and contact discontinuities are captured well. Figure~\ref{fig:riemann 13}(b) shows the profile of shock indicator. 
%{\color{red} TODO: shock indicator only show data >= 0.05, along with the mesh... So you see exactly where limiter is activated.}
\begin{figure}[H]
  \centering
  \begin{subfigure}{0.42\textwidth}
    \centering
    \includegraphics[width=\linewidth]{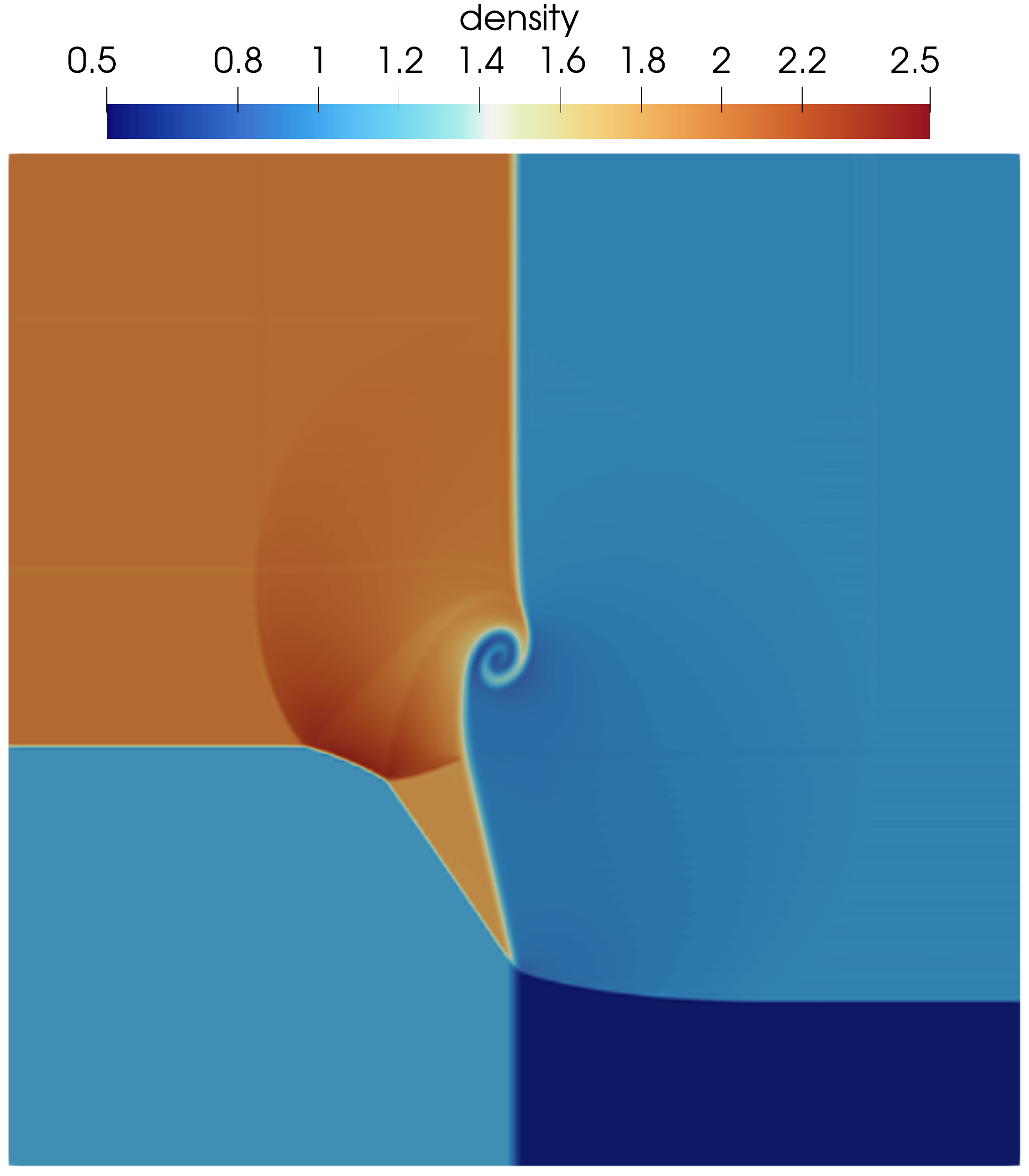}
    \caption{density profile }
    %\label{fig:vortex_ic}
  \end{subfigure}\hfill
  \begin{subfigure}{0.42\textwidth}
    \centering
    \includegraphics[width=\linewidth]{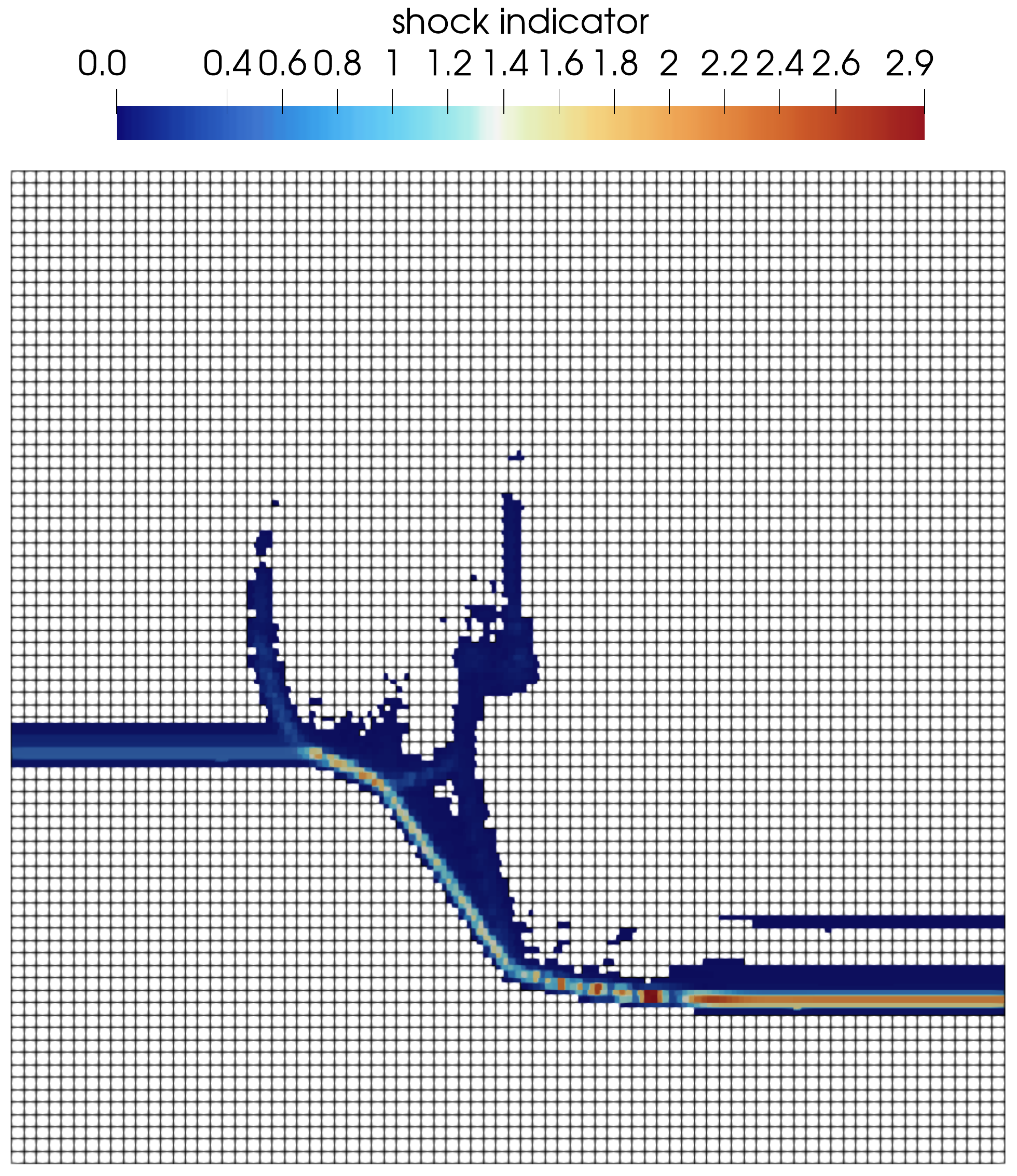}
    \caption{shock indicator}
    %\label{fig:vortex_mesh}
  \end{subfigure}
  \caption{\textit{Two-dimensional Riemann problem 13}: (a) the density profile at $t=0.3$; (b) the shock indicator profile at $t=0.3$.}
  \label{fig:riemann 13}
\end{figure}
%%%%%%%%%%%%%%%%%%%%%%%%%%%%%%%%%%%%%%%%%%%%%%%%%%%%%
%%%%%%%%%%%%%%%%%%%%%%%%%%%%%%%%%%%%%%%%%%%%%%%%%%%%%
\subsection{Double Mach Reflection}

The double Mach reflection problem is a classical benchmark for
evaluating the performance of numerical schemes in the presence of
strong shocks and complex wave interactions. Owing to its rich flow
features, including shock–shock interactions and shear-layer
instabilities, this problem has been extensively studied in the
literature
\cite{woodward1984numerical,shu2003high,zhang2017positivity}.

The computational domain is the rectangular region
$\Omega=[0,4]\times[0,1]$. The initial condition consists of an oblique
shock with Mach number $10$ propagating to the right. The shock
initially intersects the bottom boundary at $(x,y)=(1/6,0)$ and forms
an angle of $60^\circ$ with respect to the horizontal direction. The
corresponding initial states are given by
\begin{equation}
\label{eq:dmr_ic}
(\rho,u,v,p)=
\begin{cases}
(8,\;8.25\cos(\pi/6),\;-8.25\sin(\pi/6),\;116.5),
& x<\dfrac{1}{6}+\dfrac{y}{\sqrt{3}},\\[1ex]
(1.4,\;0,\;0,\;1),
& x>\dfrac{1}{6}+\dfrac{y}{\sqrt{3}}.
\end{cases}
\end{equation}

The boundary conditions are prescribed as follows. An inflow condition
is imposed on the left boundary, while an outflow condition is used on
the right boundary. On the top boundary, a moving shock condition is
applied: the post-shock state is prescribed on the segment
$0<x<\frac{1}{6}+\frac{1}{\sqrt{3}}(1+20t)$, and the pre-shock state is
used elsewhere. On the bottom boundary, the post-shock state is imposed
for $0<x<\frac{1}{6}$, whereas a reflective wall boundary condition is
applied on the remaining portion.

The computation is carried out on a uniform Cartesian grid with
$N_x=480$ and $N_y=120$ elements, corresponding to a mesh size
$h_x=h_y=1/120$. A $\mathbb Q_2$ DG approximation is employed on all
elements. The OE parameters $s$ and $C$ are set as \eqref{eq:sC}.
%{\color{blue}In this test, the OE scale factor $s$ fixed as $s=0.2$ and shock indicator threshold is set as $C=0.02$.}

%{\color{red}The OE scaling factor in \eqref{eq:oe scale factor} is set to $s=0.4$, and the shock indicator threshold in \eqref{eq:oe C} is chosen as $C=0.02$. WE SHALL MAKE THESE CONSTANTS UNIFORM THROUGH ALL TESTS...}

Figure~\ref{fig:dmr} presents the numerical results at time $t=0.2$.
Figure~\ref{fig:dmr}(a) shows the density contours in the region
$[0,3]\times[0,1]$, where nonphysical oscillations near strong
discontinuities are effectively suppressed by the OE procedure.
Figure~\ref{fig:dmr}(b) displays a zoomed-in view of the density field
over $[2.2,2.8]\times[0,0.5]$, revealing fine-scale flow structures,
including small roll-ups associated with Kelvin--Helmholtz
instabilities. In contrast, such details are significantly damped in
\cite{peng2025oedg} due to stronger numerical dissipation.
\begin{figure}[H]
  \centering
  \begin{subfigure}{0.64\textwidth}
    \centering
    \includegraphics[width=\linewidth]{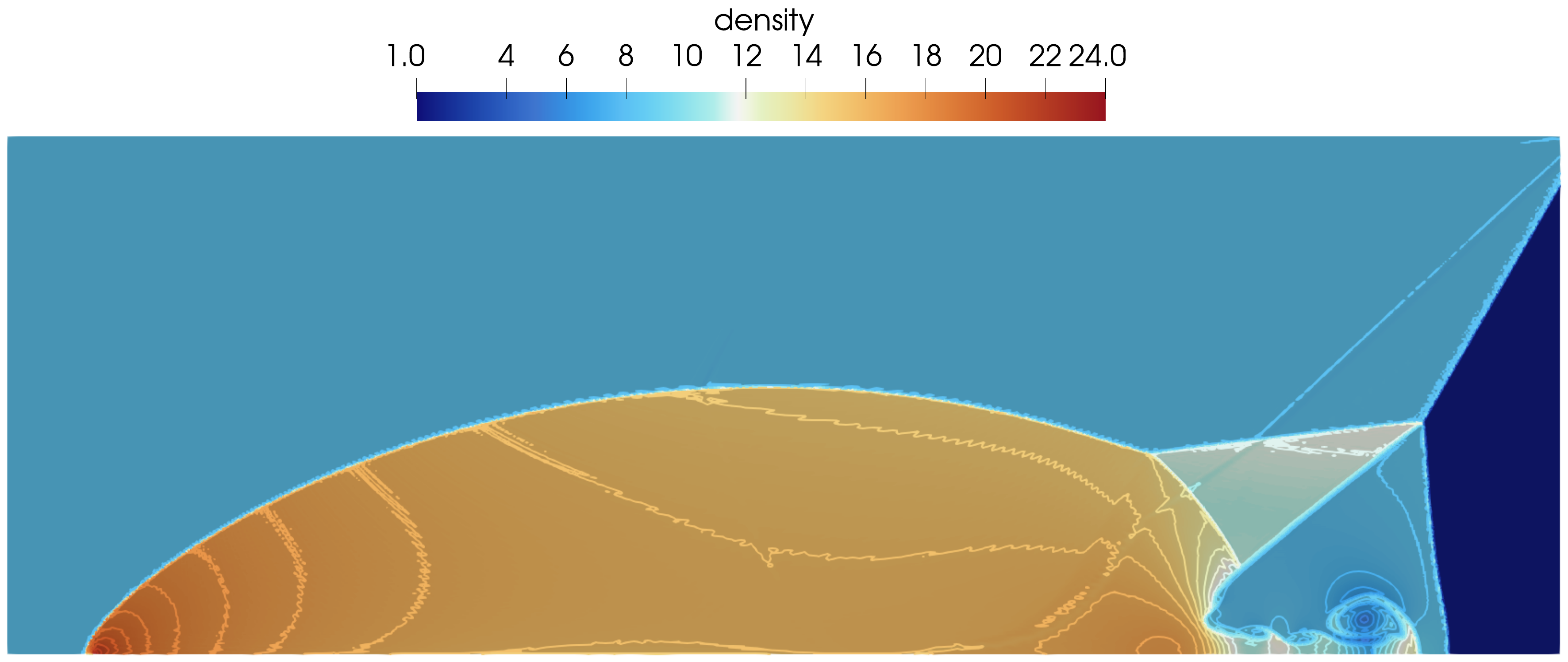}
    \caption{density profile and contour}
    %\label{fig:vortex_ic}
  \end{subfigure}
  \hspace{0.5cm} 
  \begin{subfigure}{0.28\textwidth}
    \centering
    \includegraphics[width=\linewidth]{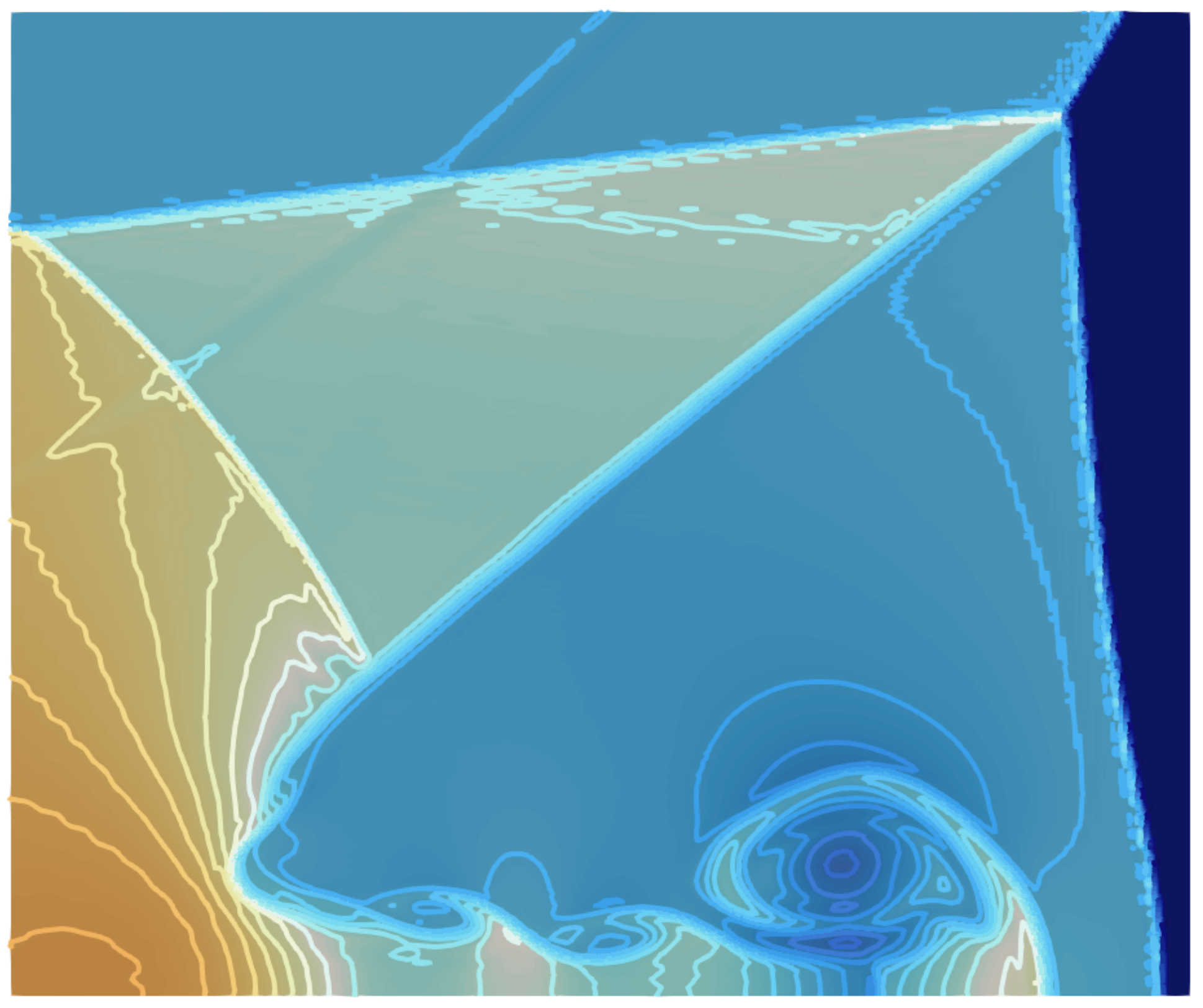}
    \caption{close-up of density profile}
    %\label{fig:vortex_mesh}
  \end{subfigure}

  \vspace{0.5cm}  % 行间距
  
  % 第二行
  \begin{subfigure}{0.64\textwidth}
    \centering
    \includegraphics[width=\linewidth]{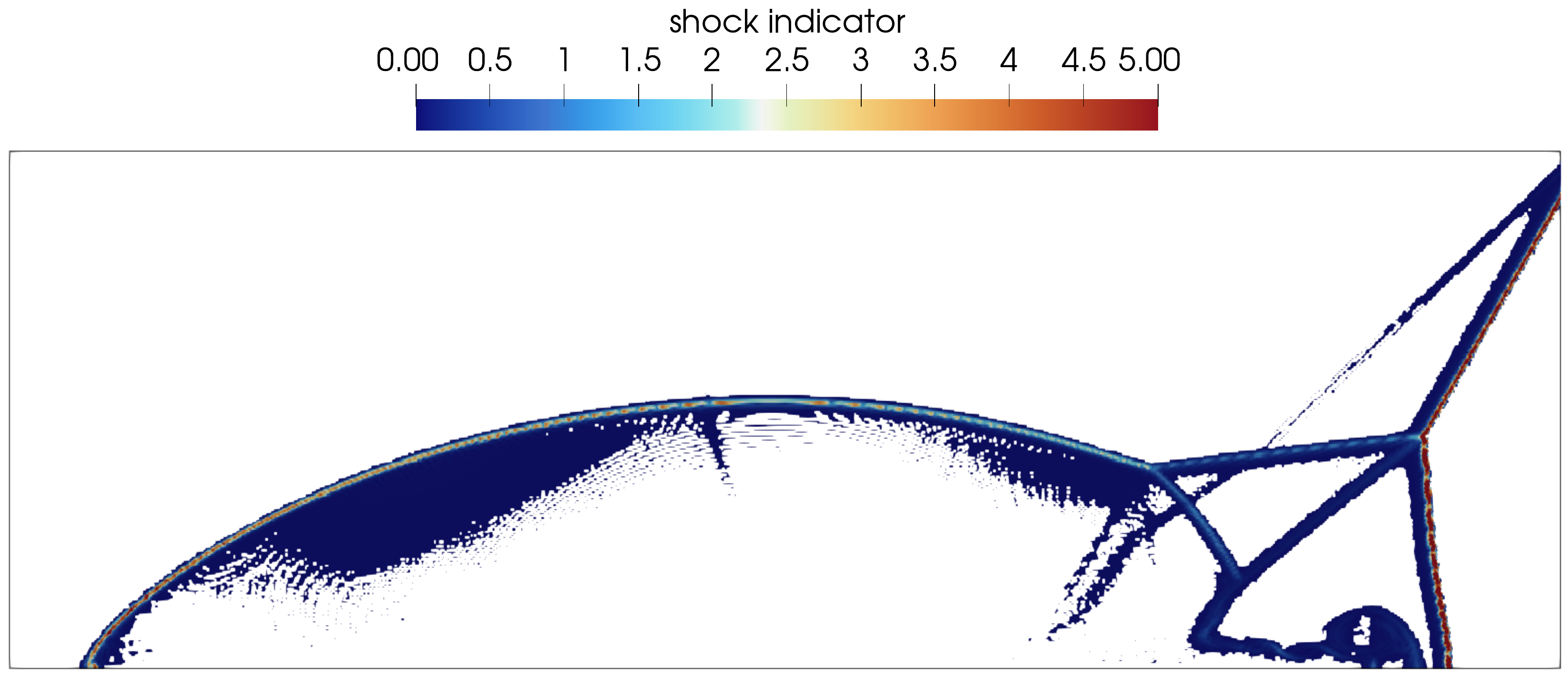}
    \caption{shock indicator profile}
    %\label{fig:sub3}
  \end{subfigure}
  \caption{\textit{Double Mach Reflection}: (a) the density profile and contour in $[0,3]\times[0,1]$ at $t=0.2$. The contour shown in the image is generated from 35 evenly distributed contours; (b) the close-up of density profile in $[2.2,2.8]\times[0,0.5]$ at $t=0.2$; (c) the shock indicator profile.}
  \label{fig:dmr}
\end{figure}
%We also test how the shock indicator threshold  $C$ will impact the computation cost. We use the same parameters as in figure\eqref{fig:dmr}, but change the shock indicator threshold $C=0, 0.02, 0.1$. The following table\eqref{tab:dmr} summaries the computation time for each case. The average CPU time per step is simply $\dfrac{\text{total CPU time}}{\text{total step}}$. Table\eqref{tab:dmr} clearly demonstrates that with increasing value of shock indicator threshold $C$, the computation cost decreases, which shows that the shock indicator can increase the efficiency of OE procedure. The test in the last row of Table\eqref{tab:dmr} was carried out without OE procedure, in which the average CPU time per step is much less than the other tests. However, the total step and total CPU time are larger than other tests, because the numerical-oscillation leads to very small time step which results to more step.
We further investigate the influence of the shock-indicator threshold $C$ on the computational cost.
All parameters are identical to those used in Figure~\eqref{fig:dmr}, except that the threshold is varied as $C=0$, $0.02$, and $0.1$.
Table~\eqref{tab:dmr} reports the corresponding computational statistics.
The average CPU time per time step is defined as $\dfrac{\text{total CPU time}}{\text{total step}}$.

As shown in Table~\eqref{tab:dmr}, increasing the threshold $C$ leads to a reduction in computational cost.
This behavior is expected, since a larger threshold identifies fewer troubled cells and therefore reduces the number of elements where the OE procedure is activated.
These results confirm that the shock-indicator strategy effectively improves the efficiency of the OE procedure.

For comparison, the last row of Table~\eqref{tab:dmr} corresponds to a computation performed without the OE procedure.
In this case, the average CPU time per step is significantly smaller, due to the absence of additional damping operations.
However, the total number of time steps — and consequently the total CPU time — is larger than in the stabilized cases.
This is because uncontrolled numerical oscillations enforce smaller admissible time steps, thereby increasing the overall computational effort.
\begin{table}[ht!]
    \centering
    \begin{tabular}{|c|c|c|c|}
    \hline
         & total step & total CPU time & average CPU time per step\\
         \hline
        $C=0$ & 10261 & 3123.8 $s$ & 0.304 $s$\\
        \hline
        $C=0.02$ & 10261 & 2570.3 $s$ & 0.250 $s$\\
        \hline
        $C=0.1$ & 10302 & 2128.3 $s$ & 0.206 $s$\\
        \hline
        no OE & 40622 & 4029.0 $s$ & 0.099 $s$\\
        \hline
    \end{tabular}
    \caption{\textit{Double Mach Reflection}: comparison of computation time with different shock indicator threshold $C$. In the last row, the OE procedure is not used.}
    \label{tab:dmr}
\end{table}
%The figure\eqref{fig:dmr2} shows the profiles of shock indicator and density with  $C=0, 0.02, 0.1$. As we can see, the density profiles of the cases with  $C=0, 0.02, 0.1$ show very minor differences. However, the density profile of no OE shows severe numerical oscillation.
Figure~\eqref{fig:dmr2} presents the density fields and corresponding shock-indicator distributions for $C=0$, $0.02$, and $0.1$.
The density solutions obtained with different threshold values exhibit only minor differences, indicating that moderate variations of $C$ do not significantly affect the resolved flow structures.
In contrast, the solution computed without the OE procedure exhibits pronounced nonphysical oscillations, highlighting the necessity of oscillation control for robust simulations.
\begin{figure}[h]
  \centering
  \begin{subfigure}{0.49\textwidth}
    \centering
    \includegraphics[width=\linewidth]{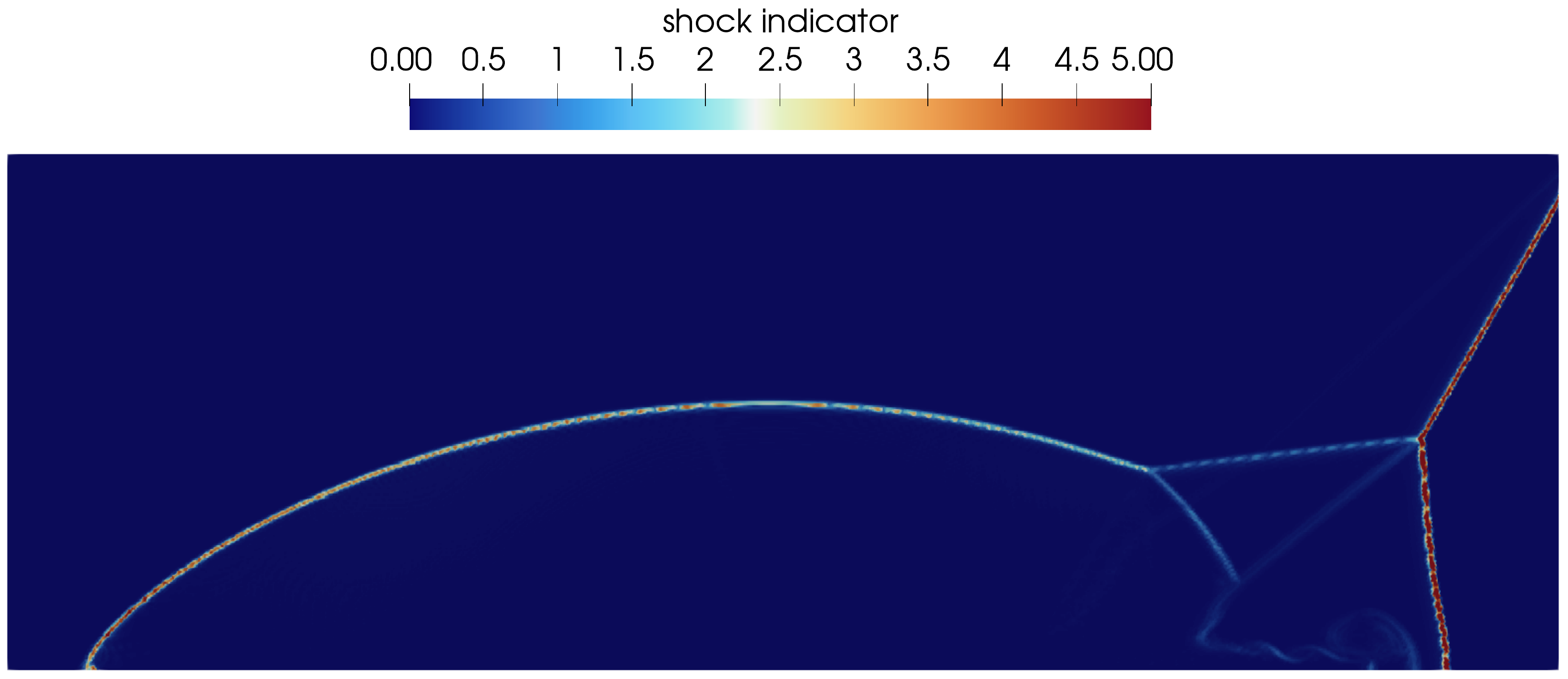}
    \caption{shock indicator $C=0$}
    %\label{fig:vortex_ic}
  \end{subfigure}
  \hspace{0.0cm}  % 行间距
  \begin{subfigure}{0.49\textwidth}
    \centering
    \includegraphics[width=\linewidth]{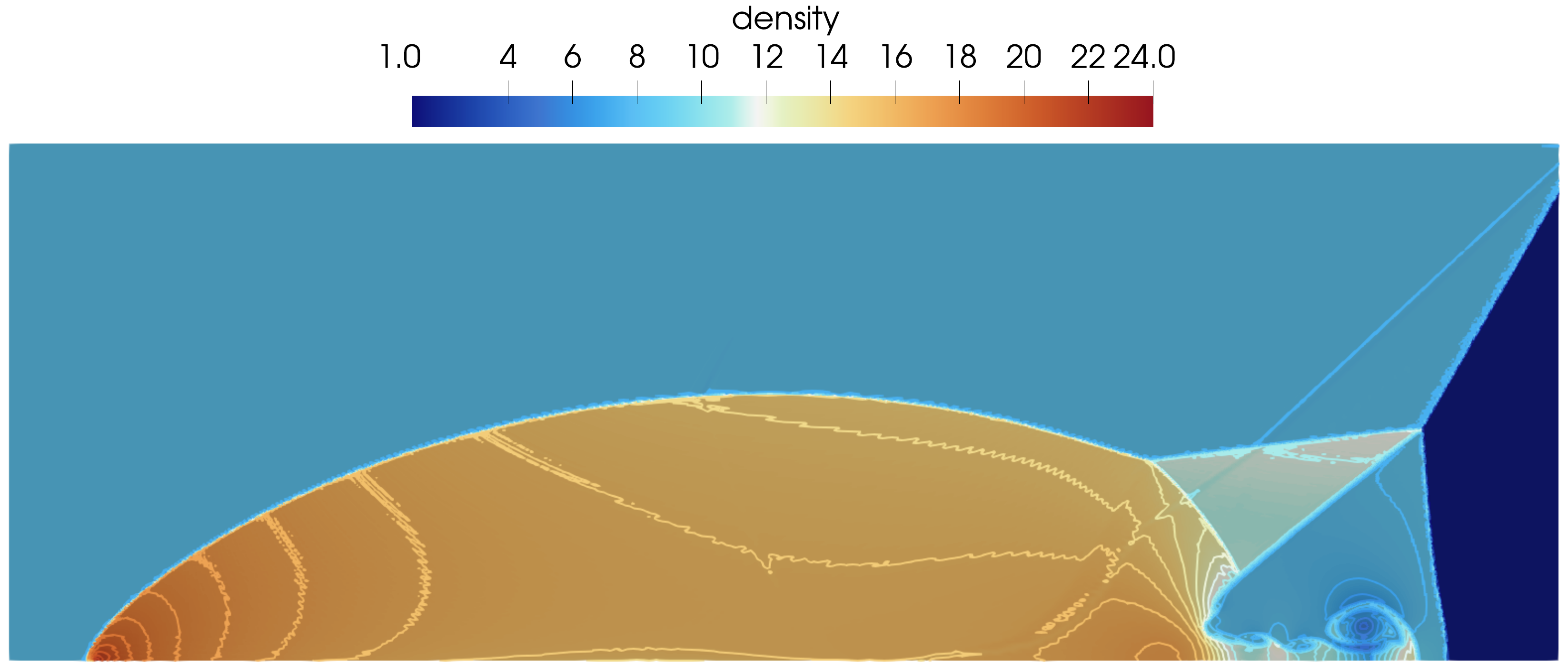}
    \caption{density $C=0$}
    %\label{fig:vortex_mesh}
  \end{subfigure}

  \vspace{0.cm}  % 行间距

 % 第二行
  \begin{subfigure}{0.49\textwidth}
    \centering
    \includegraphics[width=\linewidth]{dmr/indicator_C=0.02.pdf}
    \caption{shock indicator $C=0.02$}
    %\label{fig:vortex_ic}
  \end{subfigure}
  \hspace{0.0cm}  % 行间距
  \begin{subfigure}{0.49\textwidth}
    \centering
    \includegraphics[width=\linewidth]{dmr/density_contour_C=0.02.pdf}
    \caption{density $C=0.02$}
    %\label{fig:vortex_mesh}
  \end{subfigure}

  \vspace{0.cm}  % 行间距

 % 第二行
  \begin{subfigure}{0.49\textwidth}
    \centering
    \includegraphics[width=\linewidth]{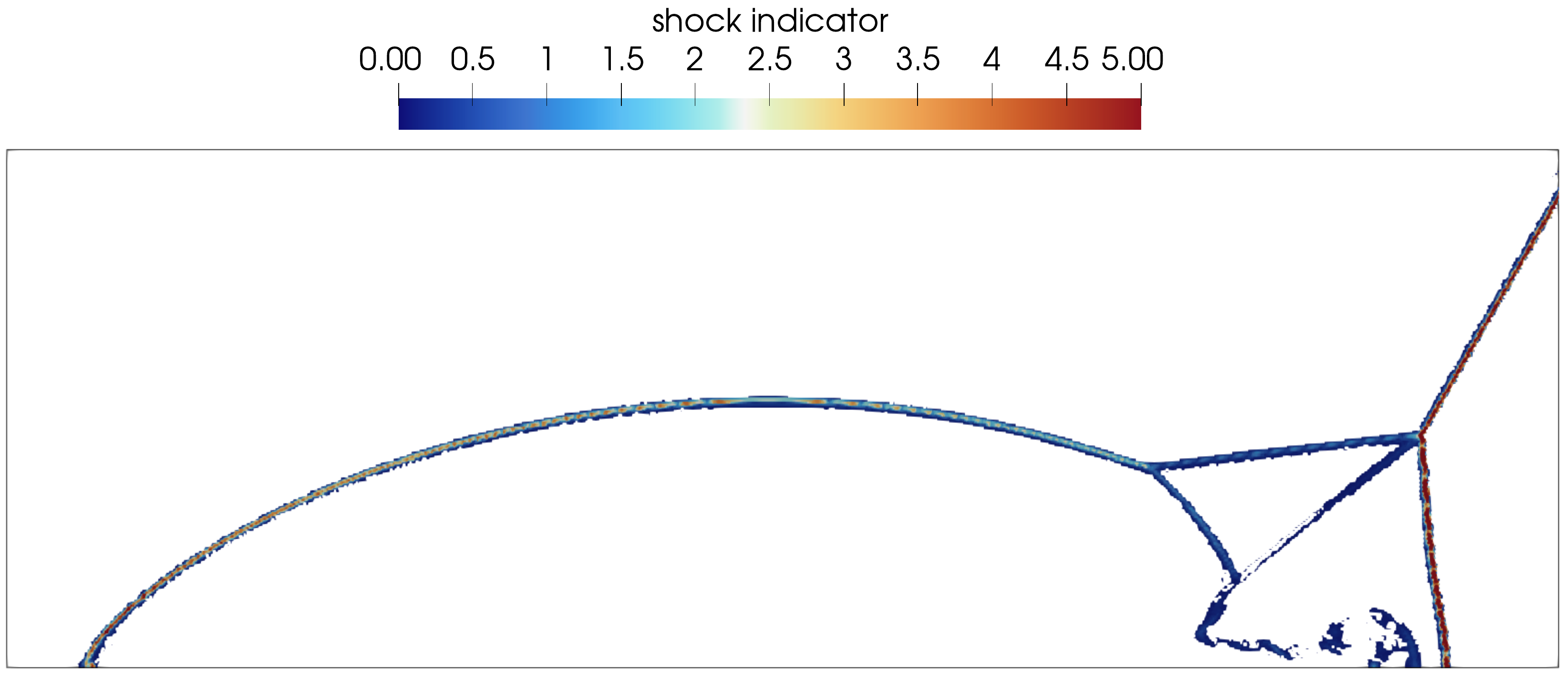}
    \caption{shock indicator $C=0.1$}
    %\label{fig:vortex_ic}
  \end{subfigure}
  \hspace{0.0cm}  % 行间距
  \begin{subfigure}{0.49\textwidth}
    \centering
    \includegraphics[width=\linewidth]{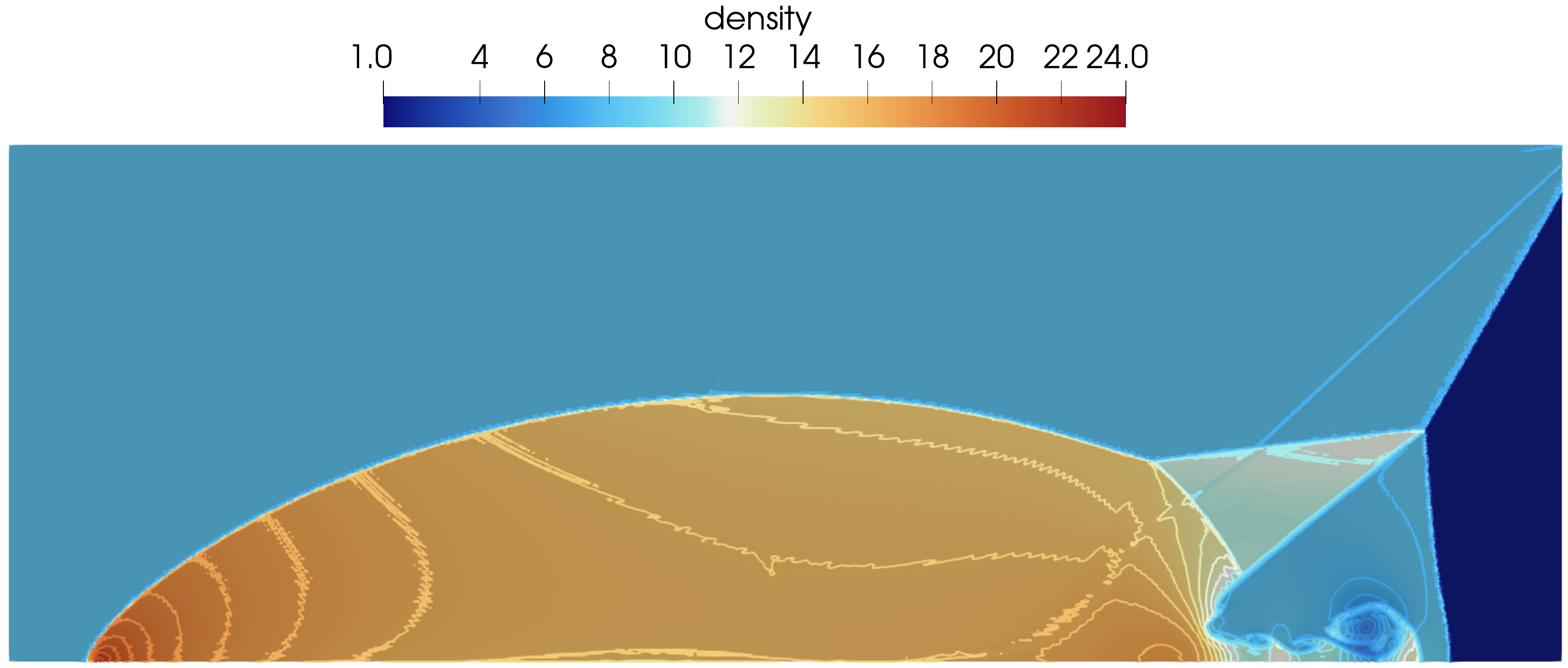}
    \caption{density $C=0.1$}
    %\label{fig:vortex_mesh}
  \end{subfigure}

  \vspace{0.cm}  % 行间距
  
  % 第二行
  \begin{subfigure}{0.49\textwidth}
    \centering
    \includegraphics[width=\linewidth]{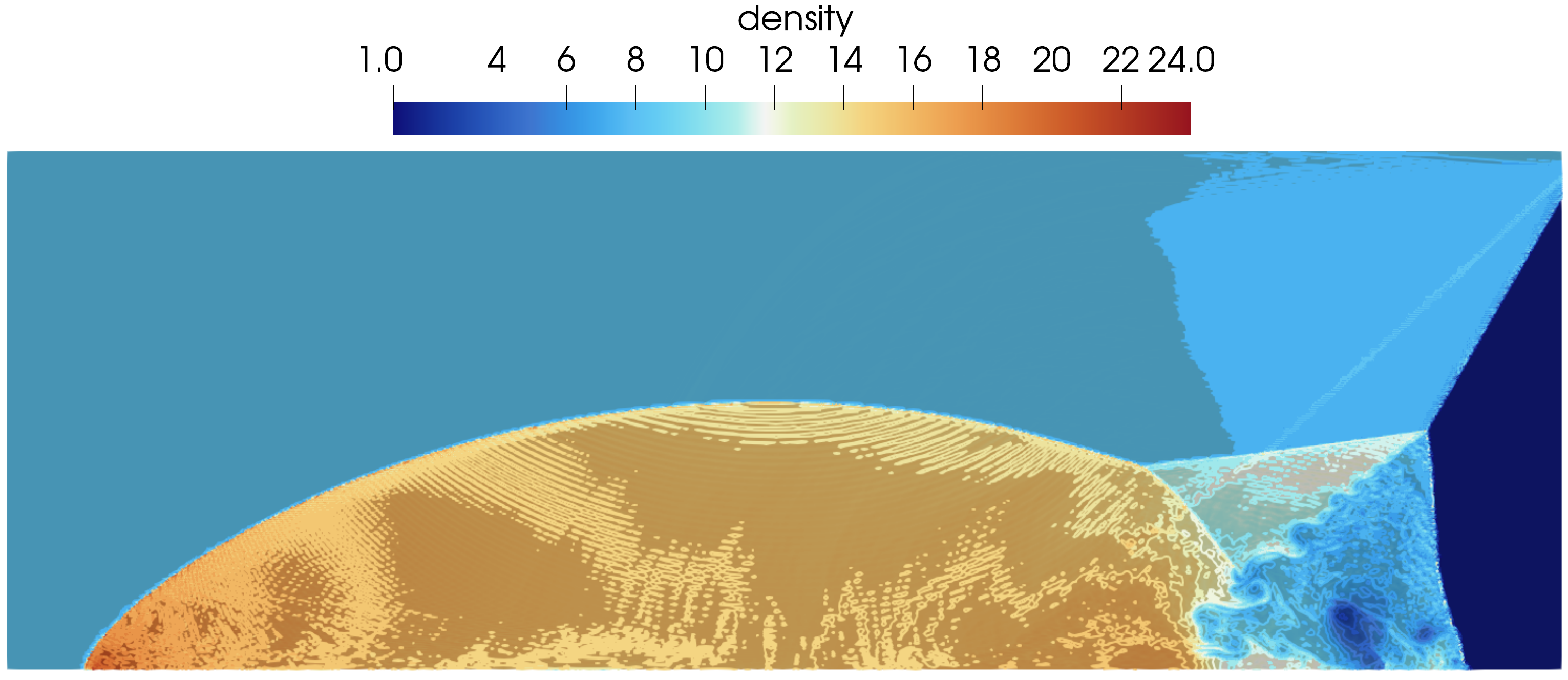}
    \caption{no OE procedure}
    %\label{fig:sub3}
  \end{subfigure}
  \caption{\textit{Double Mach Reflection}: Shock indicator profile and density profile with different threshold value $C$.}
  \label{fig:dmr2}
\end{figure}

%%%%%%%%%%%%%%%%%%%%%%%%%%%%%%%%%%%%%%%%%%%%%%%%%%%%%
%%%%%%%%%%%%%%%%%%%%%%%%%%%%%%%%%%%%%%%%%%%%%%%%%%%%%
\subsection{Supersonic Flow Around a Two-Dimensional Cylinder}

We next consider the problem of supersonic flow past a circular
cylinder, which has been extensively studied in the literature
\cite{guermond2018second,deng2023new,nazarov2017numerical,chen2025high}.
This is a highly challenging benchmark, as it involves multiple
interacting flow features, including a strong bow shock, oblique
shocks, a downstream ``fishtail'' shock, Kelvin--Helmholtz
instabilities, and their nonlinear interactions. Accurately resolving
these phenomena requires a numerical method that is both robust and
minimally dissipative.

The computational domain is the rectangular region
$\Omega=[-0.6,3.4]\times[-1,1]$, containing a circular cylinder of
diameter $D=0.5$ centered at $(0,0)$. The mesh is generated using the
open-source software \texttt{Gmsh}
\cite{geuzaine2009gmsh,remacle2012blossom}. We employ high-order
curvilinear quadrilateral elements whose geometric order matches the
polynomial degree of the DG approximation space
$\mathbb Q_N$. The mesh is generated with a characteristic mesh size of $h=0.0085$, and consists of $126{,}082$ elements. Using
third-order elements ($N=3$), the total number of degrees of freedom is
$2{,}017{,}312$.

The initial condition corresponds to a uniform Mach~3 supersonic inflow
with $\rho=1$, $p=1$, and velocity $\bm u=(3.55,0)$. Inflow and outflow
boundary conditions are imposed on the left and right boundaries,
respectively. Slip wall (reflective) boundary conditions are applied on
the upper and lower boundaries, as well as on the surface of the
cylinder. We use $\mathbb Q_3$ for the DG approximation polynomial.
Here, we vary the OE scale factor $s$ to examine its impact on the performance of the OE procedure described in Algorithm~\ref{alg:oe_curve}. The shock indicator threshold remains fixed at $C=0.02$.

Figure~\ref{fig:cylinder1} presents results obtained with OE scaling
factor $s=0.1$ at times $t=0.5$, $2.5$, $3.0$, and $4.5$, together with
the corresponding shock indicator distributions. To highlight the shock structures and shear layers, the flow field is visualized using a Schlieren-like representation define by $\log(1+|\nabla\rho|)$:

Initially, the Mach~3 inflow impinges on the cylinder, generating a
strong bow shock that propagates toward the upper and lower walls. Two
oblique shocks form near the downstream sides of the cylinder and
travel toward the outflow boundary. In addition, a distinct downstream
shock, commonly referred to as the ``fishtail'' shock
\cite{chen2025high,nazarov2017numerical}, develops in the wake region.
Flow separation occurs near the onset points of the oblique shocks,
leading to a recirculating region behind the cylinder. Farther
downstream, the wake becomes unsteady and a long trailing vortex
structure emerges.

At later times ($t\gtrsim2$), Kelvin--Helmholtz instabilities develop
along the contact lines emanating from the shock-wave triple points
near the upper and lower walls. The successful resolution of these
small-scale roll-up structures indicates that the numerical diffusion
introduced by the OE procedure remains sufficiently low. Overall, the
computed flow structures are in good agreement with the results
reported by Guermond et~al.~\cite{guermond2018second} and Deng
\cite{deng2023new}. The shock indicator distributions, shown in the
right column of Figure~\ref{fig:cylinder1}, are well aligned with the
locations of strong shocks, confirming that the OE procedure is
selectively activated in nonsmooth regions.
\begin{figure}[htbp]
  \centering
  % 第一行
  \begin{subfigure}{0.44\textwidth}
    \centering
    \includegraphics[width=\linewidth]{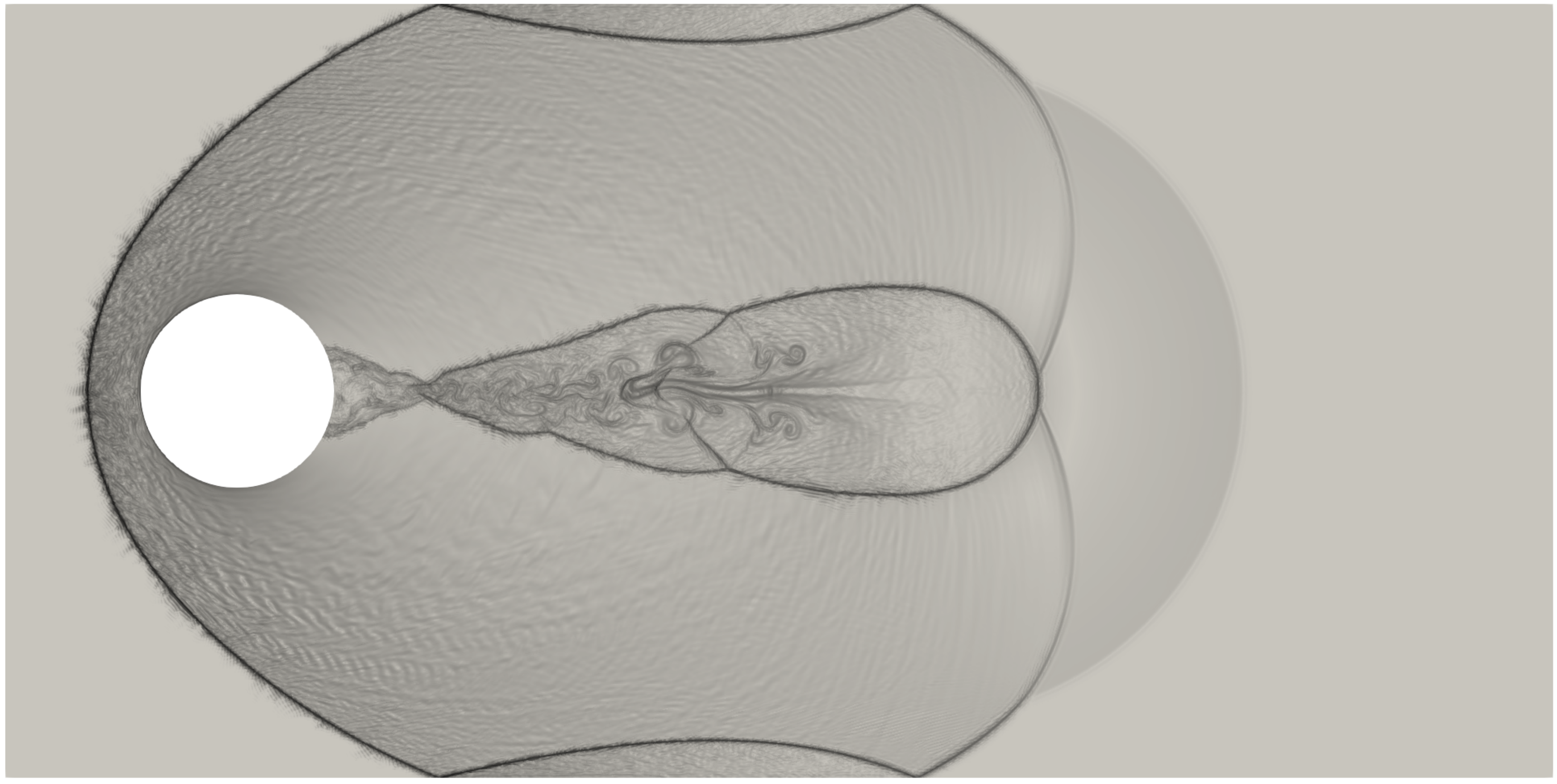}
    \caption{density at $t=0.5$}
    %\label{fig:sub1}
  \end{subfigure}
  \hspace{0.02\textwidth}
  \begin{subfigure}{0.49\textwidth}
    \centering
    \includegraphics[width=\linewidth]{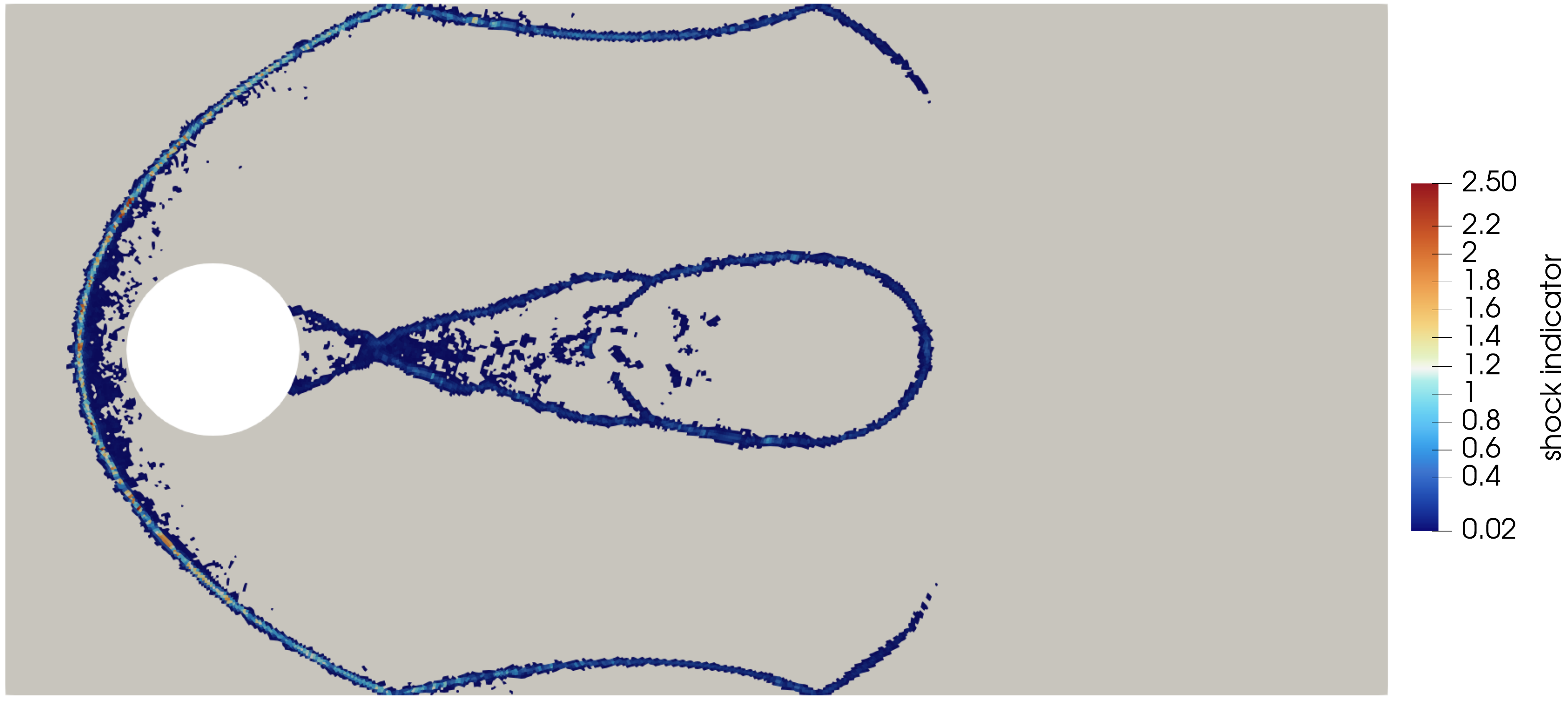}
    \caption{shock indicator at $t=0.5$}
    %\label{fig:sub2}
  \end{subfigure}
  
  \vspace{0.5cm}  % 行间距
  
  % 第二行
  \begin{subfigure}{0.44\textwidth}
    \centering
    \includegraphics[width=\linewidth]{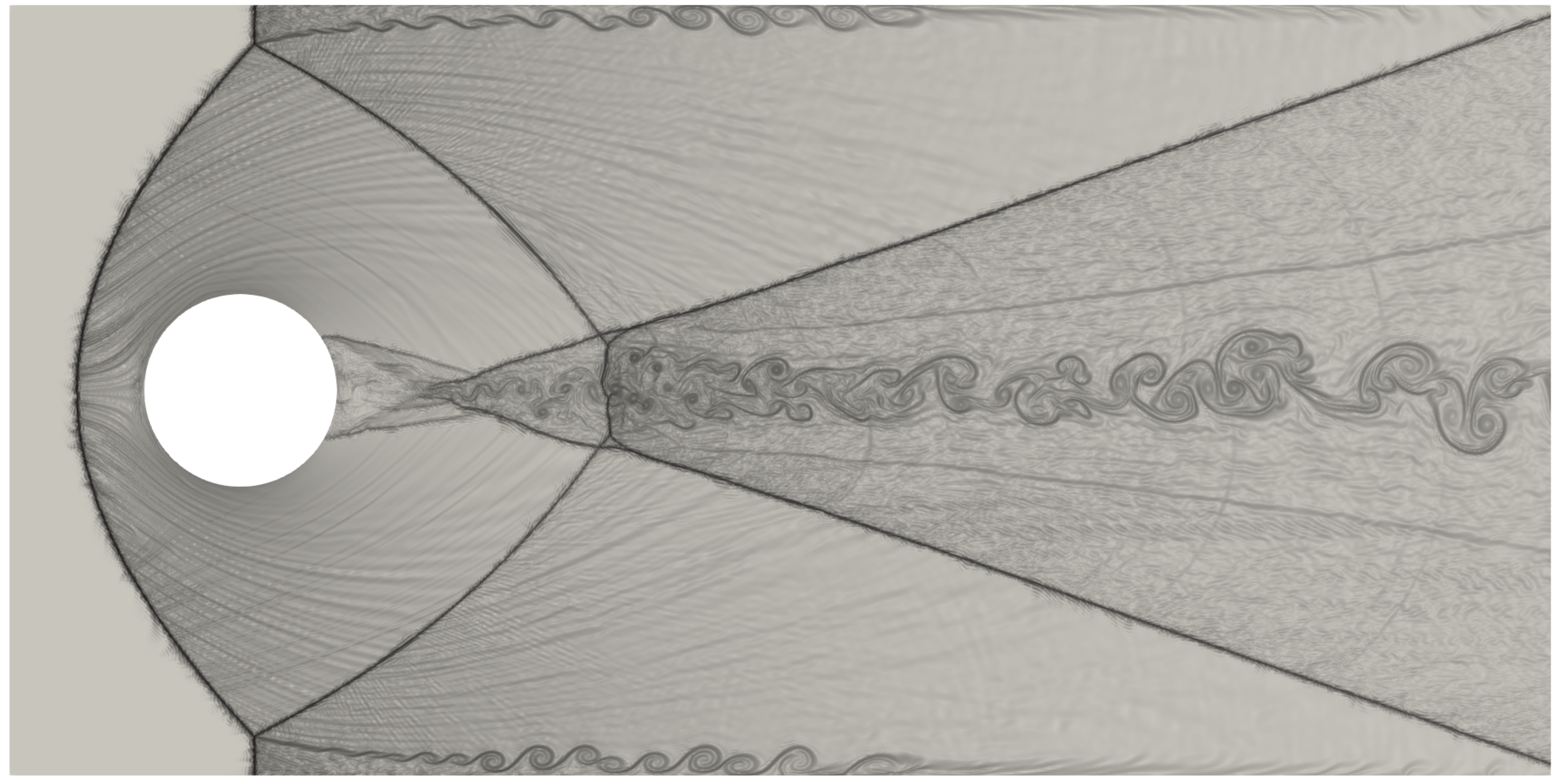}
    \caption{density at $t=2$}
    %\label{fig:sub3}
  \end{subfigure}
  \hspace{0.02\textwidth}
  \begin{subfigure}{0.49\textwidth}
    \centering
    \includegraphics[width=\linewidth]{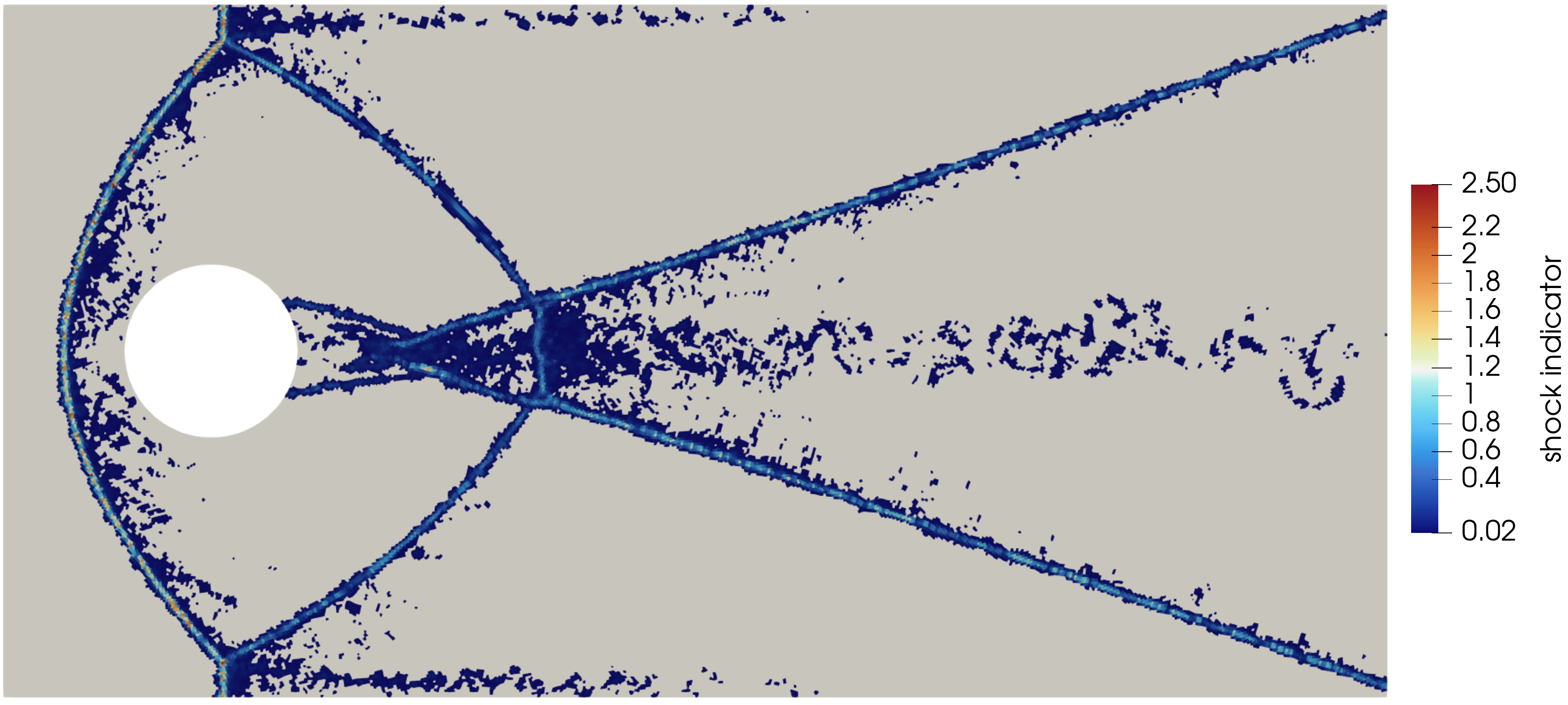}
    \caption{shock indicator at $t=2$}
    %\label{fig:sub4}
  \end{subfigure}
  
  \vspace{0.5cm}
  
  % 第三行
  \begin{subfigure}{0.44\textwidth}
    \centering
    \includegraphics[width=\linewidth]{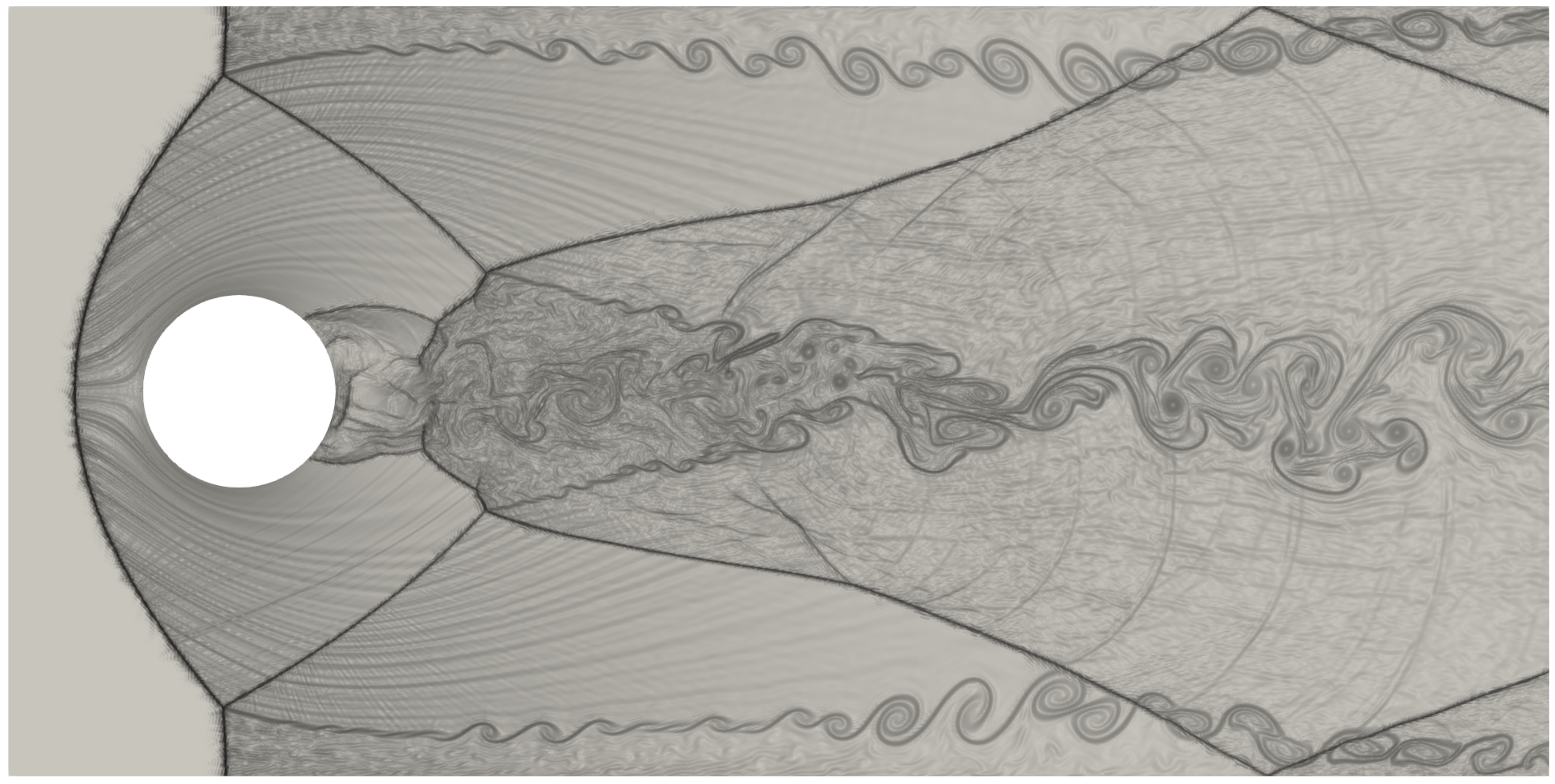}
    \caption{density at $t=3$}
    %\label{fig:sub5}
  \end{subfigure}
  \hspace{0.02\textwidth}
  \begin{subfigure}{0.49\textwidth}
    \centering
    \includegraphics[width=\linewidth]{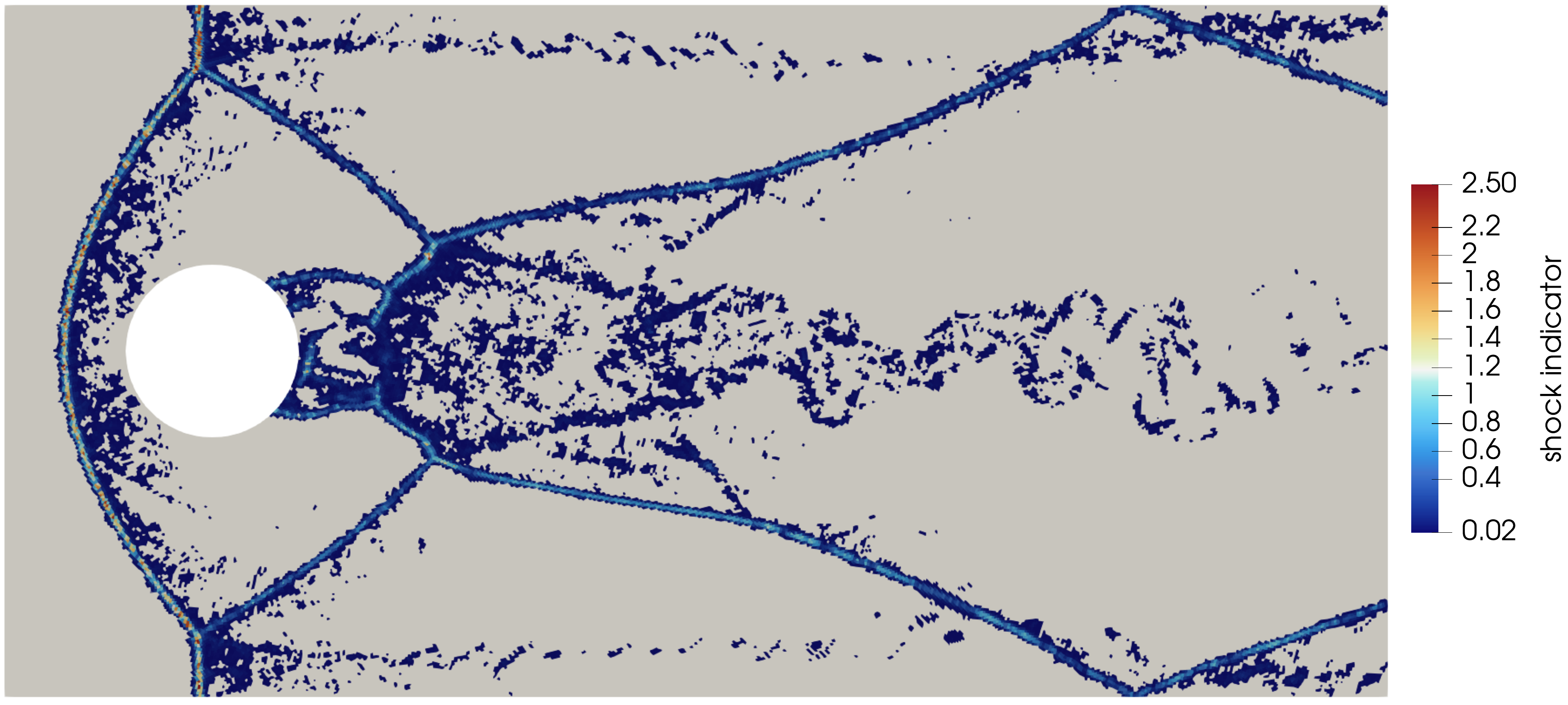}
    \caption{shock indicator at $t=3$}
    %\label{fig:sub6}
  \end{subfigure}
  
  \vspace{0.5cm}
  
  % 第四行
  \begin{subfigure}{0.44\textwidth}
    \centering
    \includegraphics[width=\linewidth]{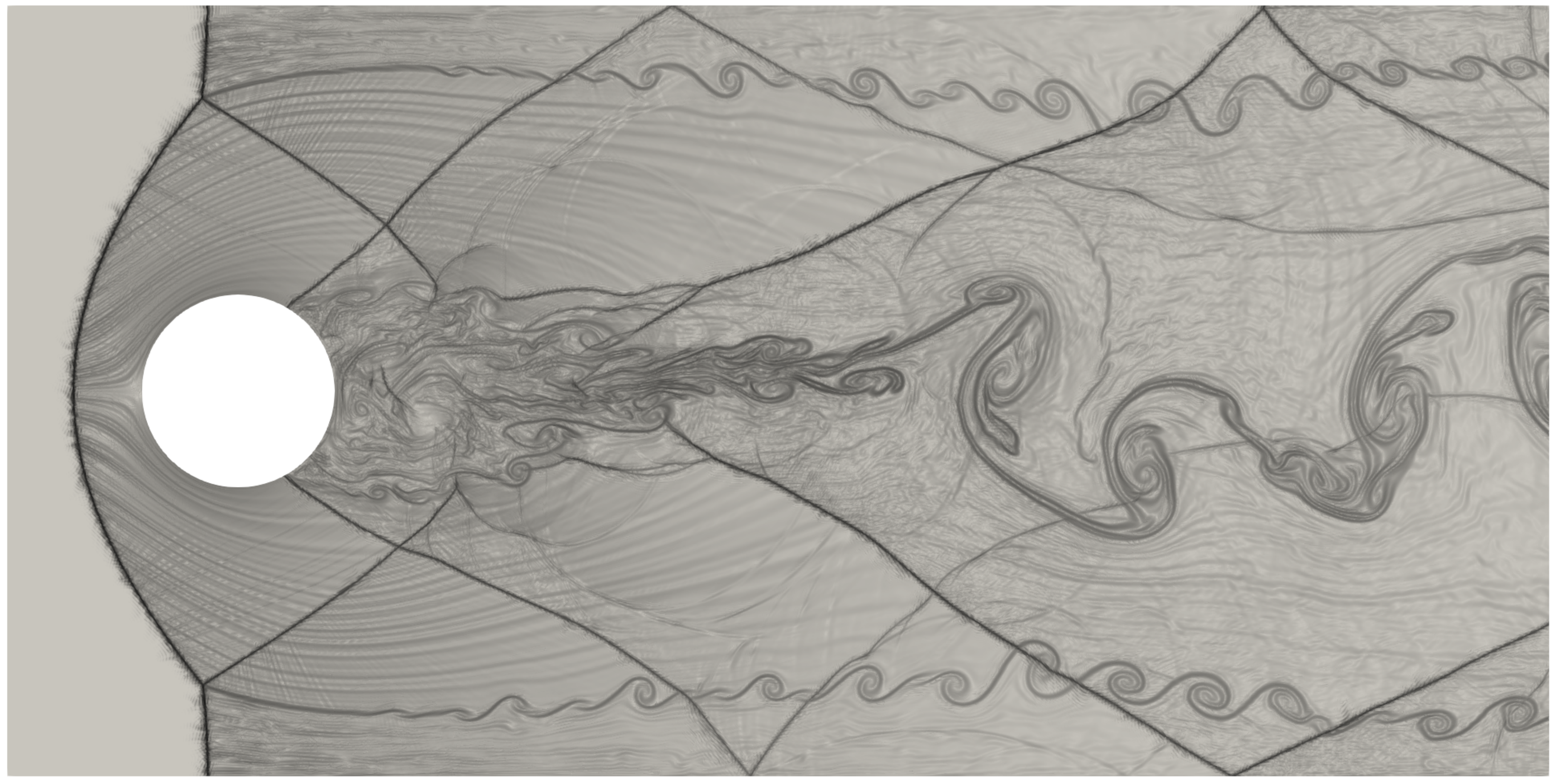}
    \caption{density at $t=4.5$}
    %\label{fig:sub7}
  \end{subfigure}
  \hspace{0.02\textwidth}
  \begin{subfigure}{0.49\textwidth}
    \centering
    \includegraphics[width=\linewidth]{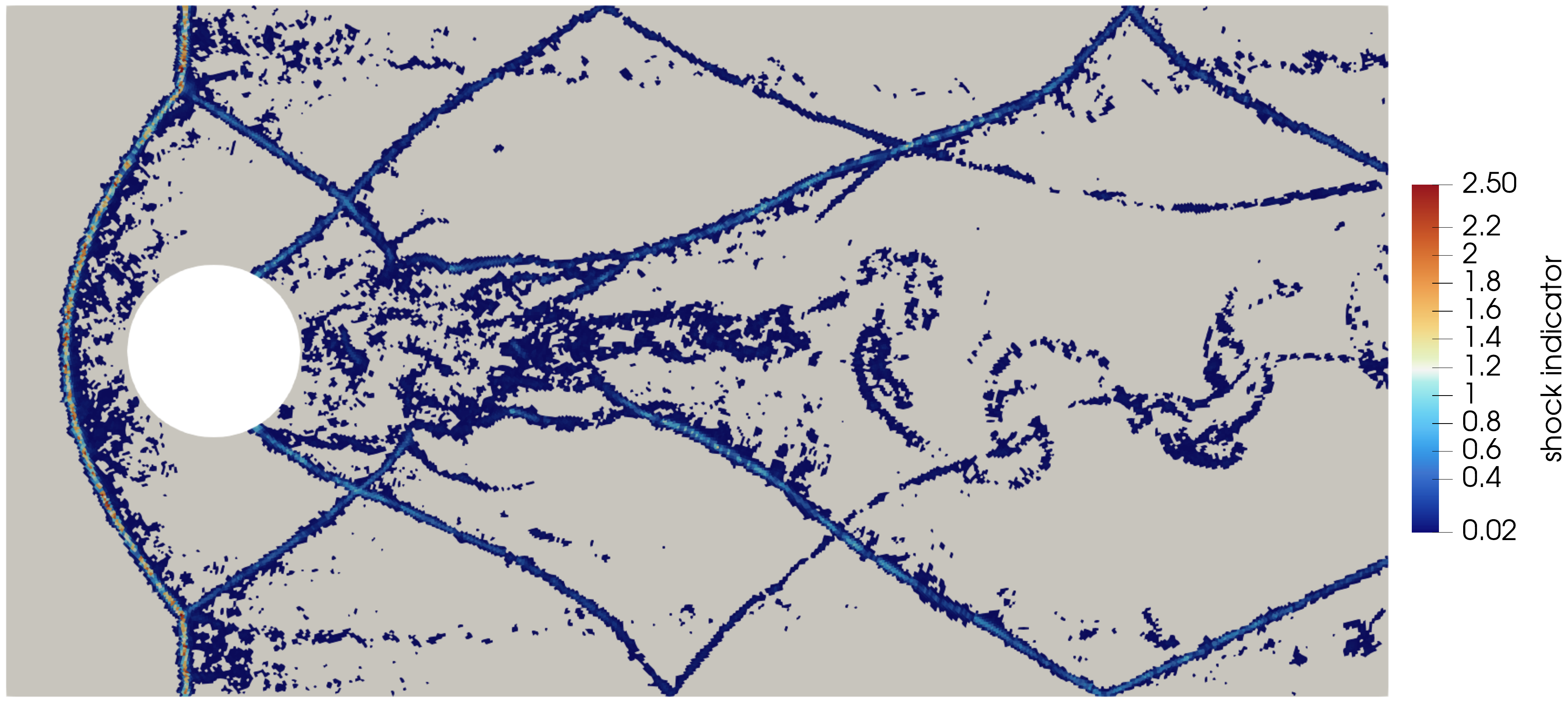}
    \caption{shock indicator at $t=4.5$}
    %\label{fig:sub8}
  \end{subfigure}
  
  \caption{\textit{Supersonic Flow Around a Two-Dimensional Cylinder:} the figures on the left row are Schlieren-like plot of density field; the figures on the right row are the distributions of shock indicator\eqref{eq:oe_indicator}.}
  \label{fig:cylinder1}
\end{figure}

We next compare solutions obtained with two different values of $s$.
Figure~\ref{fig:cylinder2} presents representative results for
$s = 0.1$ and $s = 0.8$.

As shown in Figure~\ref{fig:cylinder2}, increasing the OE scaling factor
leads to a noticeable increase in numerical dissipation. In
particular, for $s=0.8$, the Kelvin--Helmholtz instabilities developing
along the contact lines near the upper and lower walls are significantly
suppressed, indicating excessive numerical diffusion. In contrast, for
the smaller value $s=0.1$, these instabilities are well resolved,
demonstrating that the OE procedure remains sufficiently non-dissipative
to capture fine-scale flow features.

On the other hand, in the vicinity of the bow shock ahead of the
cylinder, the case $s=0.1$ exhibits more pronounced numerical
oscillations, whereas these oscillations are strongly damped when
$s=0.8$. As a result of the different levels of numerical diffusion,
the overall temporal evolution of the flow differs between the two
cases, although the main shock structures remain qualitatively similar.
Similar sensitivity to numerical dissipation has been reported in
\cite{chen2025high}, where variations in grid resolution and numerical
schemes lead to noticeable differences in small-scale flow details,
while the overall flow topology is preserved.
\begin{figure}[H]
  \centering
  % 第一行
  \begin{subfigure}{0.45\textwidth}
    \centering
    \includegraphics[width=\linewidth]{cylinder/density_t=0.5.pdf}
    \caption{OE scale factor $s=0.1$ at $t=0.5$}
    %\label{fig:sub1}
  \end{subfigure}
  \hspace{0.02\textwidth}
  \begin{subfigure}{0.45\textwidth}
    \centering
    \includegraphics[width=\linewidth]{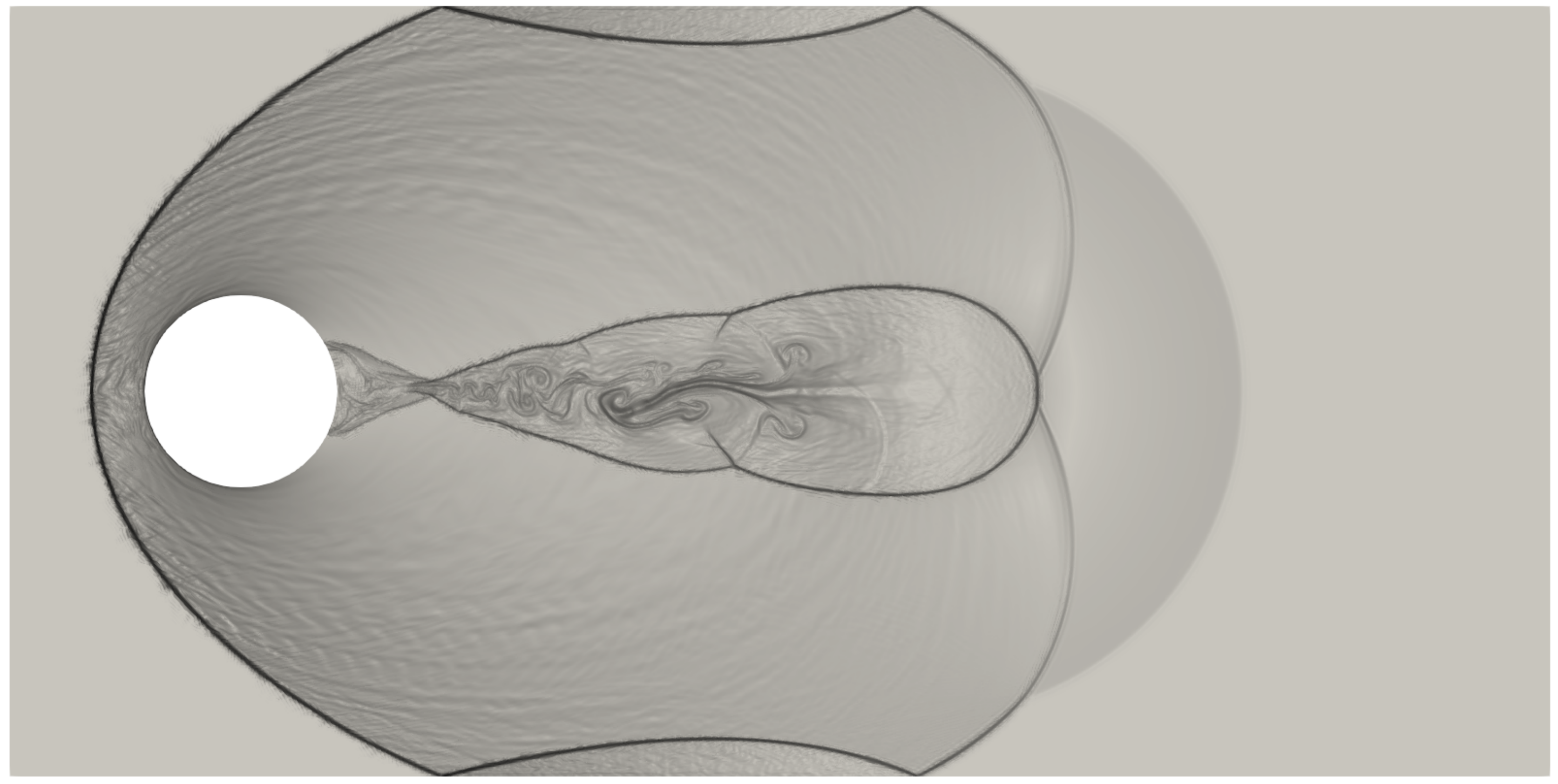}
    \caption{OE scale factor $s=0.8$ at $t=0.5$}
    %\label{fig:sub2}
  \end{subfigure}
  
  \vspace{0.2cm}  % 行间距
  
  % 第二行
  \begin{subfigure}{0.45\textwidth}
    \centering
    \includegraphics[width=\linewidth]{cylinder/density_t=3.pdf}
    \caption{OE scale factor $s=0.1$ at $t=3$}
    %\label{fig:sub3}
  \end{subfigure}
  \hspace{0.02\textwidth}
  \begin{subfigure}{0.45\textwidth}
    \centering
    \includegraphics[width=\linewidth]{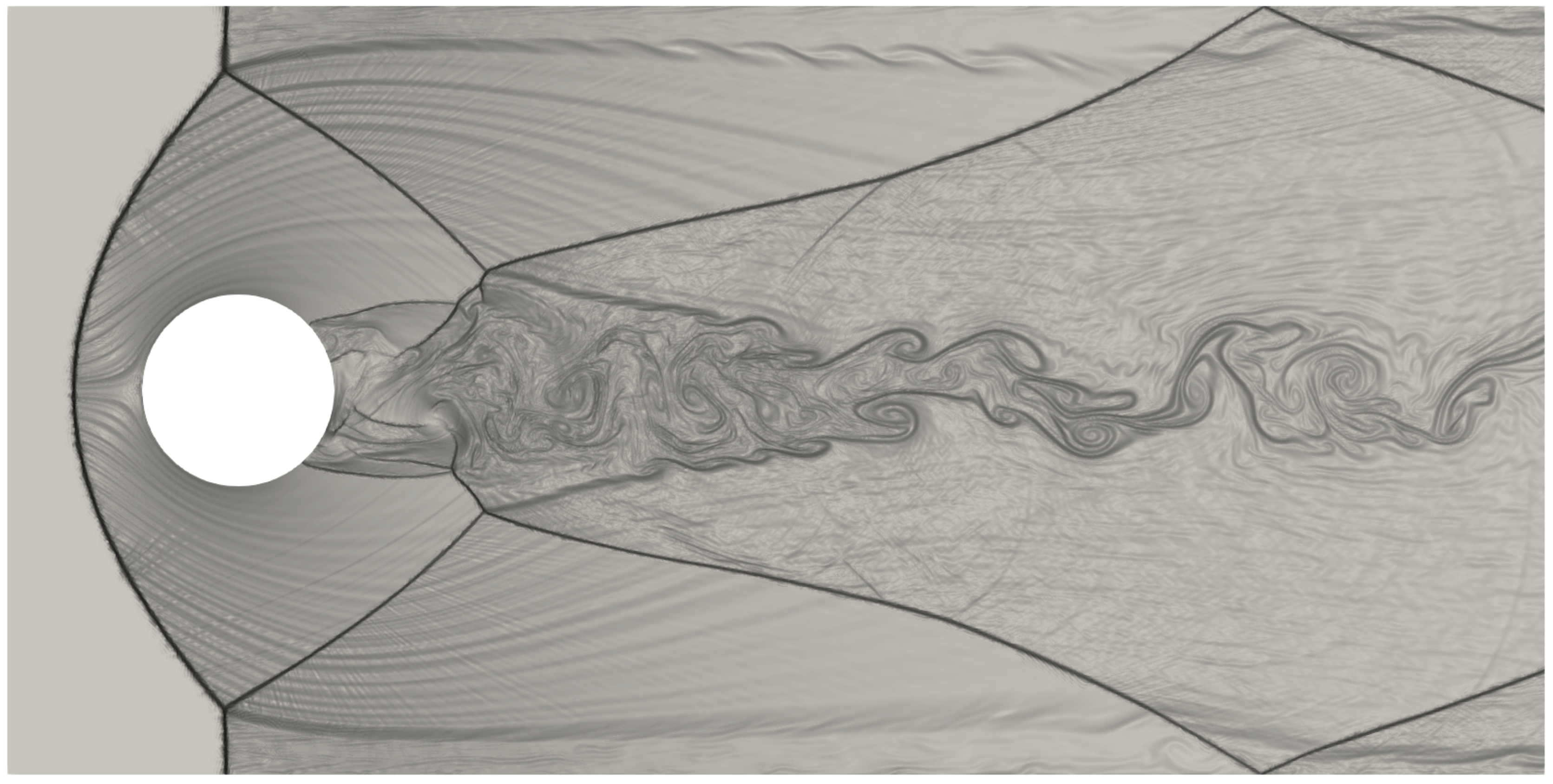}
    \caption{OE scale factor $s=0.8$ at $t=3$}
    %\label{fig:sub4}
  \end{subfigure}

  \vspace{0.2cm}  % 行间距
  % 第二行
  \begin{subfigure}{0.45\textwidth}
    \centering
    \includegraphics[width=\linewidth]{cylinder/density_t=4.5.pdf}
    \caption{OE scale factor $s=0.1$ at $t=4.5$}
    %\label{fig:sub3}
  \end{subfigure}
  \hspace{0.02\textwidth}
  \begin{subfigure}{0.45\textwidth}
    \centering
    \includegraphics[width=\linewidth]{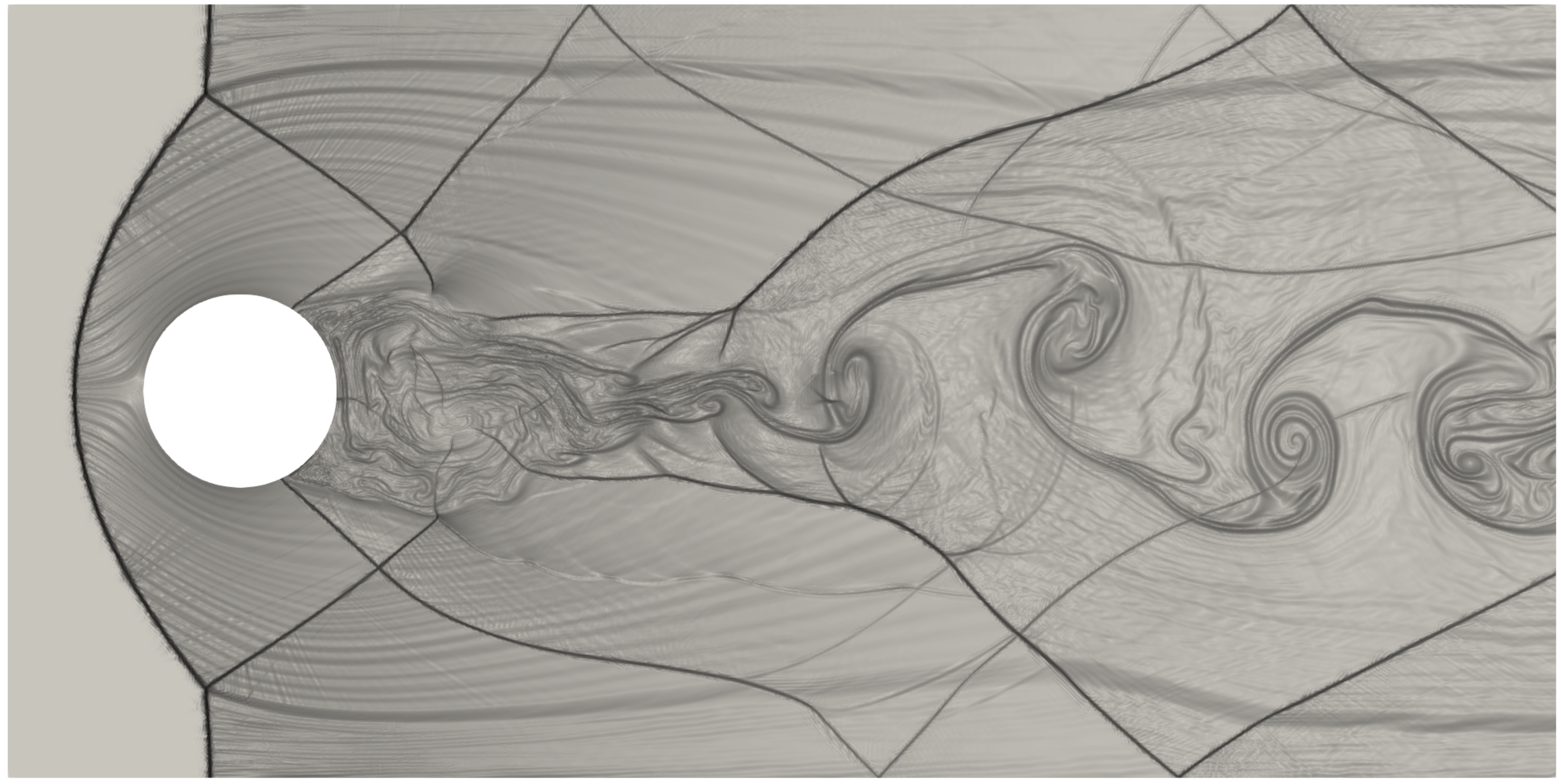}
    \caption{OE scale factor $s=0.8$ at $t=4.5$}
    %\label{fig:sub4}
  \end{subfigure}
  
  \caption{\textit{Supersonic Flow Around a Two-Dimensional Cylinder:} the figures on the left row are computed with OE scale factor $s=0.1$; the figures on the right row are computed with OE scale factor $s=0.8$.}
  \label{fig:cylinder2}
\end{figure}
\subsection{Explosion in a Domain with Multiple  Cylinders}
We next consider an explosion problem in a domain containing multiple
cylindrical obstacles, which serves as a demanding benchmark for
assessing the robustness of numerical schemes on complex geometries.
Similar configurations have been studied in
\cite{nazarov2017numerical,liska2003comparison,toro2013riemann}.

The computational domain is a two-dimensional circular region of
radius $2$ centered at the origin. Eight circular cylinders of radius
$0.3$ are placed inside the domain, with their centers located at a
distance $1.4$ from the origin and uniformly distributed at angular
positions
\[
\theta_k = \frac{k\pi}{4}, \qquad k=0,\dots,7.
\]
The mesh consists of second-order curvilinear quadrilateral elements
generated using \texttt{Gmsh}, with a characteristic mesh size
$h=0.005$. The total number of elements is $215,805$ with order 2.

The initial condition corresponds to a circular explosion. Inside a
disk of radius $\sqrt{0.4}$ centered at the origin, the gas is set to
$\rho=1$, $p=1$, and $\bm u=(0,0)$. Outside this disk, the gas is
initialized with $\rho=0.125$, $p=0.1$, and $\bm u=(0,0)$. Reflective
wall boundary conditions are imposed on all boundaries, including the
outer circular boundary and the surfaces of the eight interior cylinders. We use $\mathbb Q_2$ for the DG polynomial space, and set parameter $s=0.8$ and $C=0.02$ in this test.

Figure~\ref{fig:explosion} presents Schlieren-type visualizations of
the density field at several representative times. Initially, a strong
circular shock is generated at the interface between the two gases and
propagates radially outward. A circular contact discontinuity follows
the shock at a slower speed, while a rarefaction wave travels inward
toward the origin.

As the outward-propagating shock interacts with the cylindrical
obstacles, part of the wave is reflected by the cylinder surfaces,
while the remaining portion passes through the gaps between cylinders
and subsequently impinges on the outer circular boundary. After
reflection from the outer boundary, the shock wave propagates back
toward the center of the domain. Meanwhile, the circular contact
discontinuity weakens as it expands, eventually reverses direction,
and moves inward. This contact interface becomes unstable and breaks
down into small-scale vortical structures.

Overall, the computed flow evolution and wave interactions are in good
agreement with previously reported results
\cite{nazarov2017numerical}.
\begin{figure}[htbp]
  \centering
  % 第一行
  \begin{subfigure}{0.32\textwidth}
    \centering
    \includegraphics[width=\linewidth]{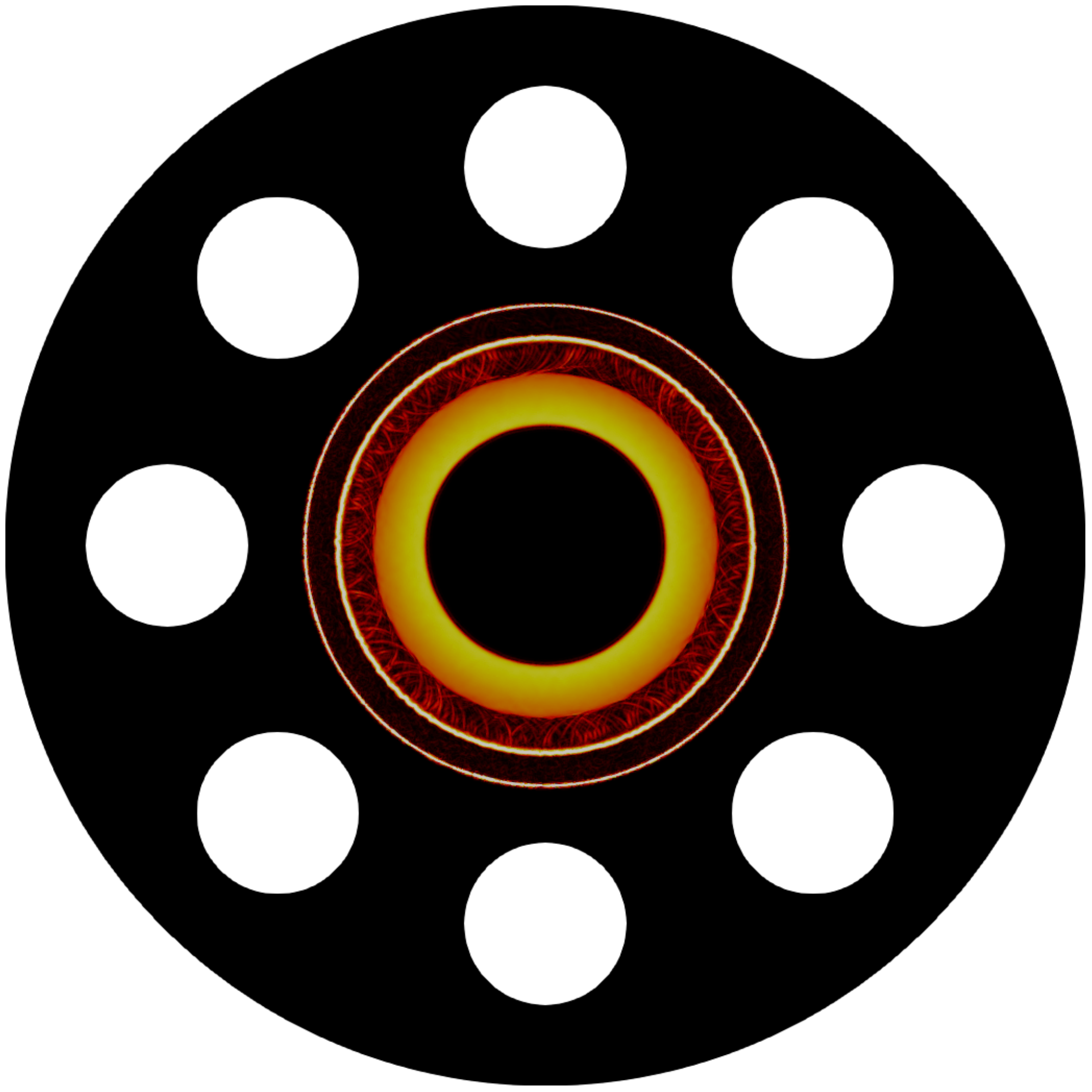}
    \caption{$t=0.15$}
  \end{subfigure}
  %\hspace{0.02\textwidth}
  \begin{subfigure}{0.32\textwidth}
    \centering
    \includegraphics[width=\linewidth]{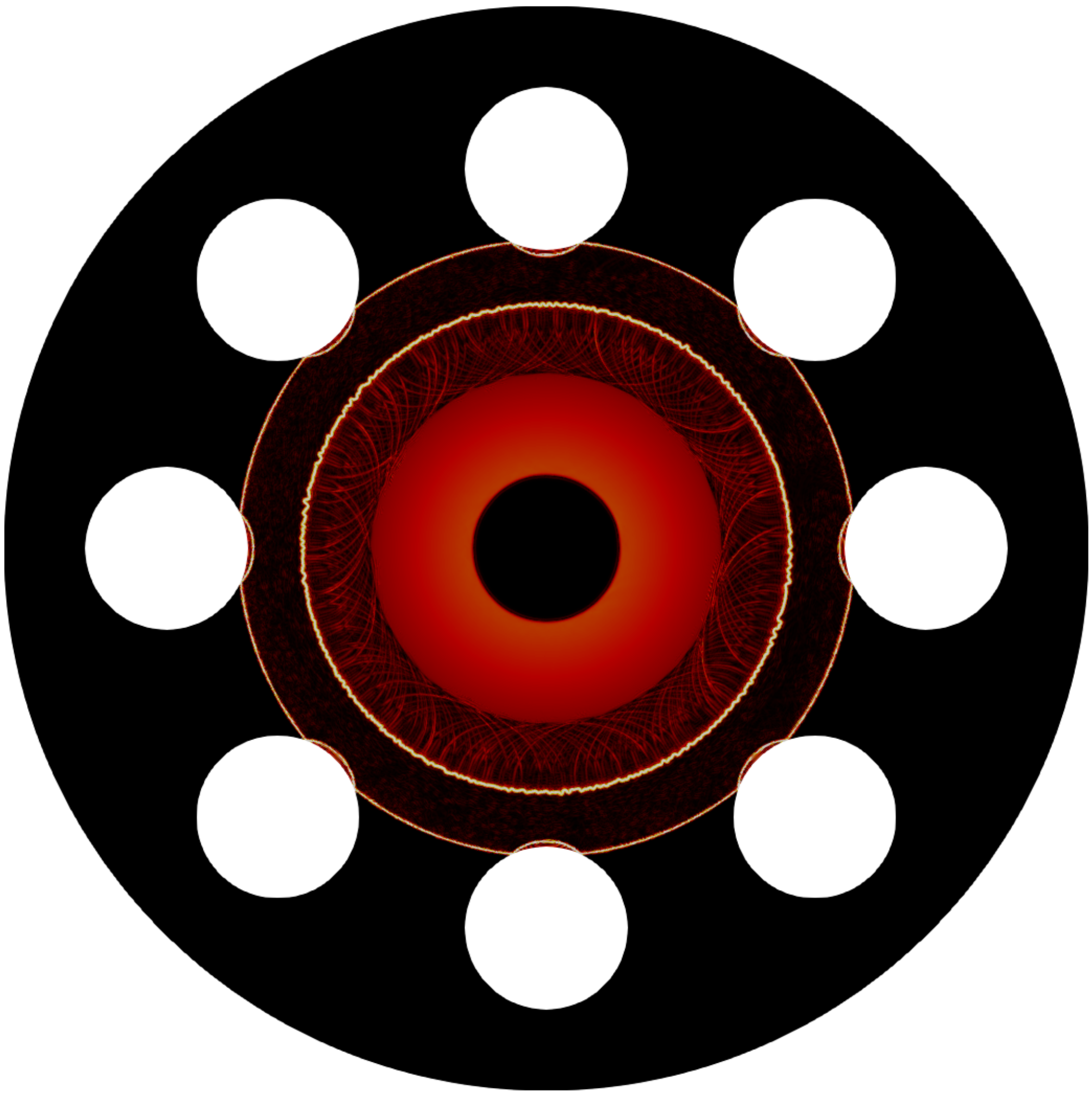}
    \caption{$t=0.3$}
  \end{subfigure}
  %\hspace{0.02\textwidth}
  \begin{subfigure}{0.32\textwidth}
    \centering
    \includegraphics[width=\linewidth]{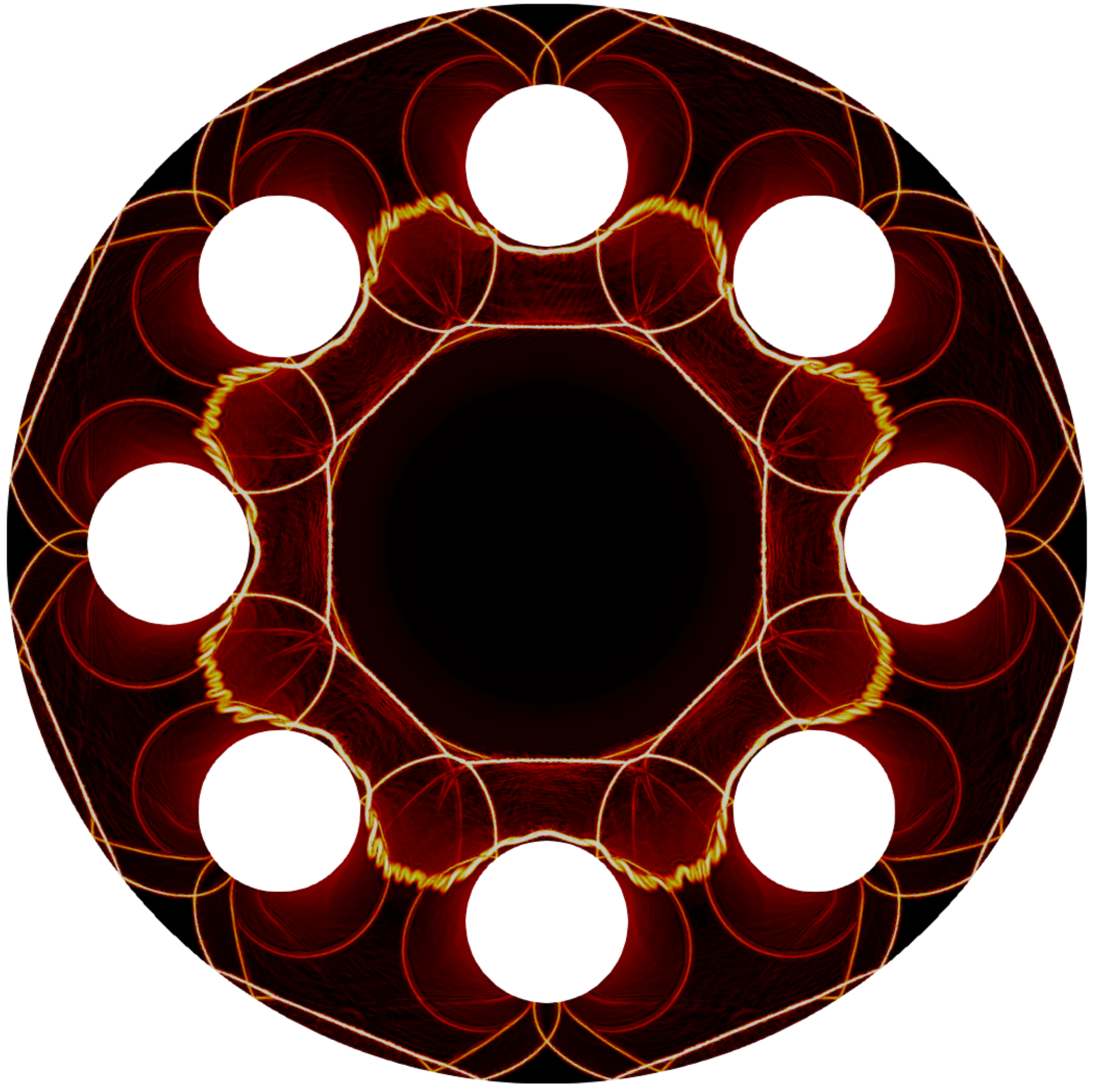}
    \caption{$t=0.9$}
  \end{subfigure}
  
  \vspace{0.5cm}  % 行间距
  
  % 第二行
  \begin{subfigure}{0.32\textwidth}
    \centering
    \includegraphics[width=\linewidth]{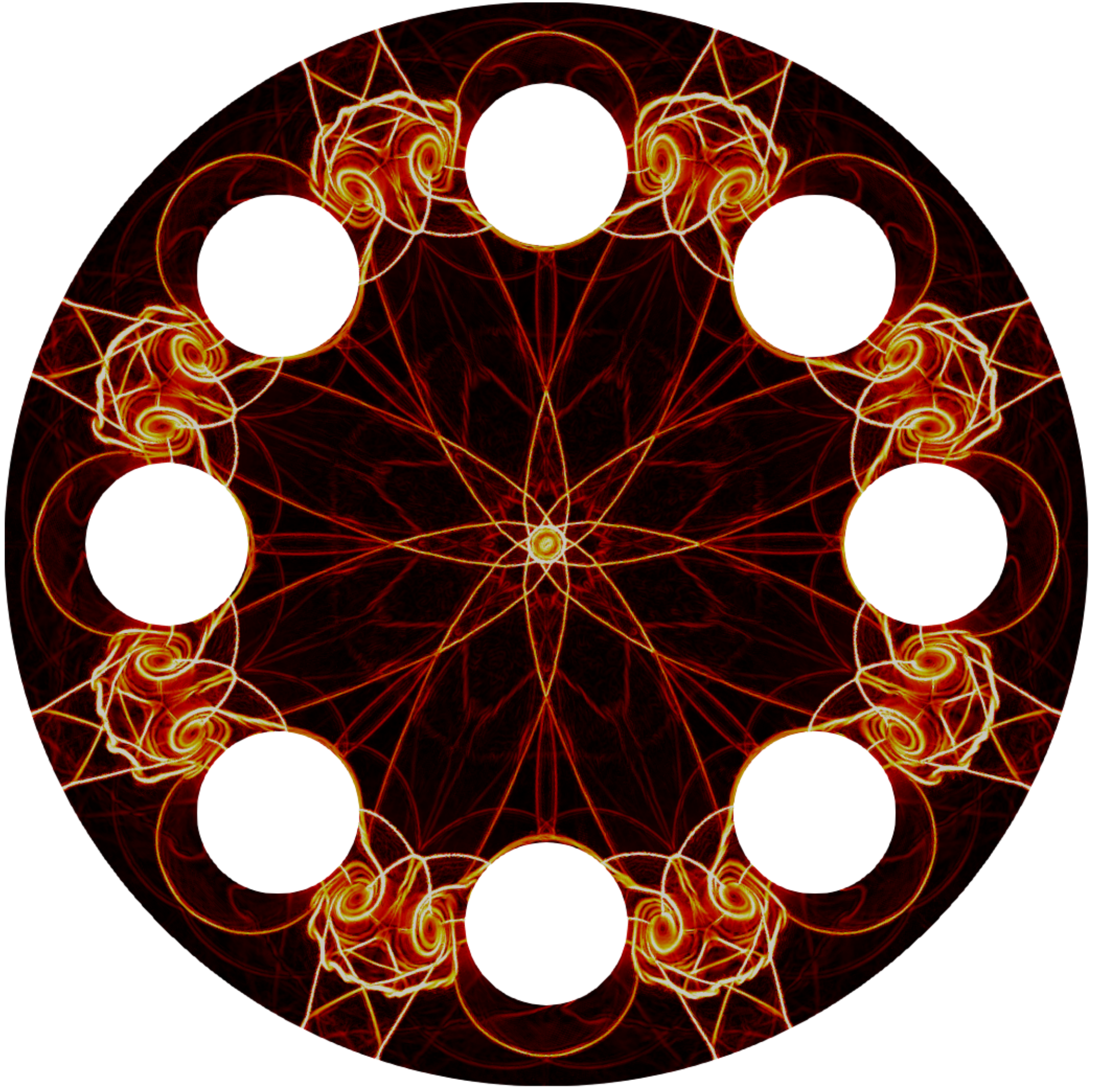}
    \caption{$t=1.8$}
  \end{subfigure}
  %\hspace{0.02\textwidth}
  \begin{subfigure}{0.32\textwidth}
    \centering
    \includegraphics[width=\linewidth]{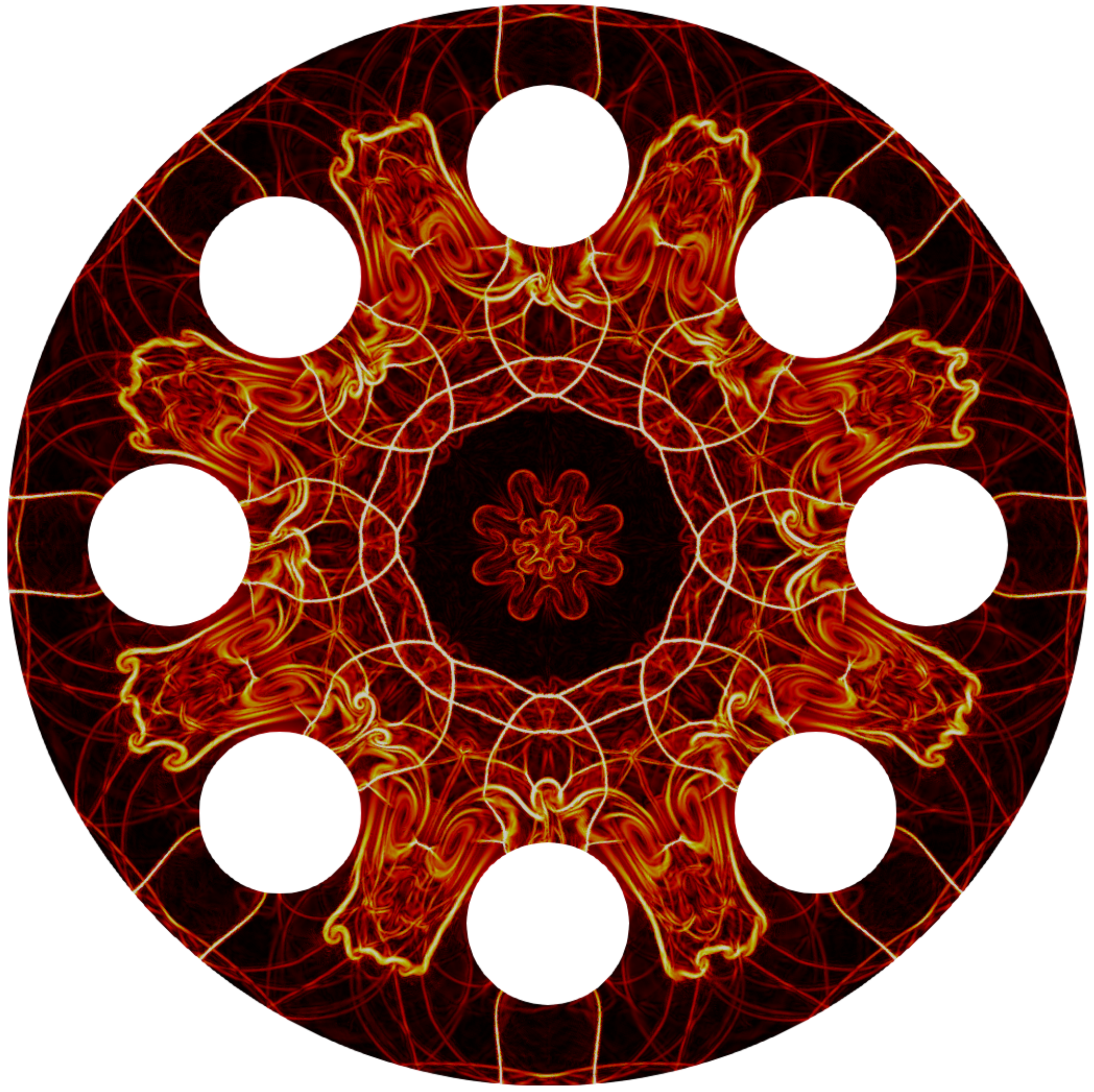}
    \caption{$t=2.4$}
  \end{subfigure}
  %\hspace{0.02\textwidth}
  \begin{subfigure}{0.32\textwidth}
    \centering
    \includegraphics[width=\linewidth]{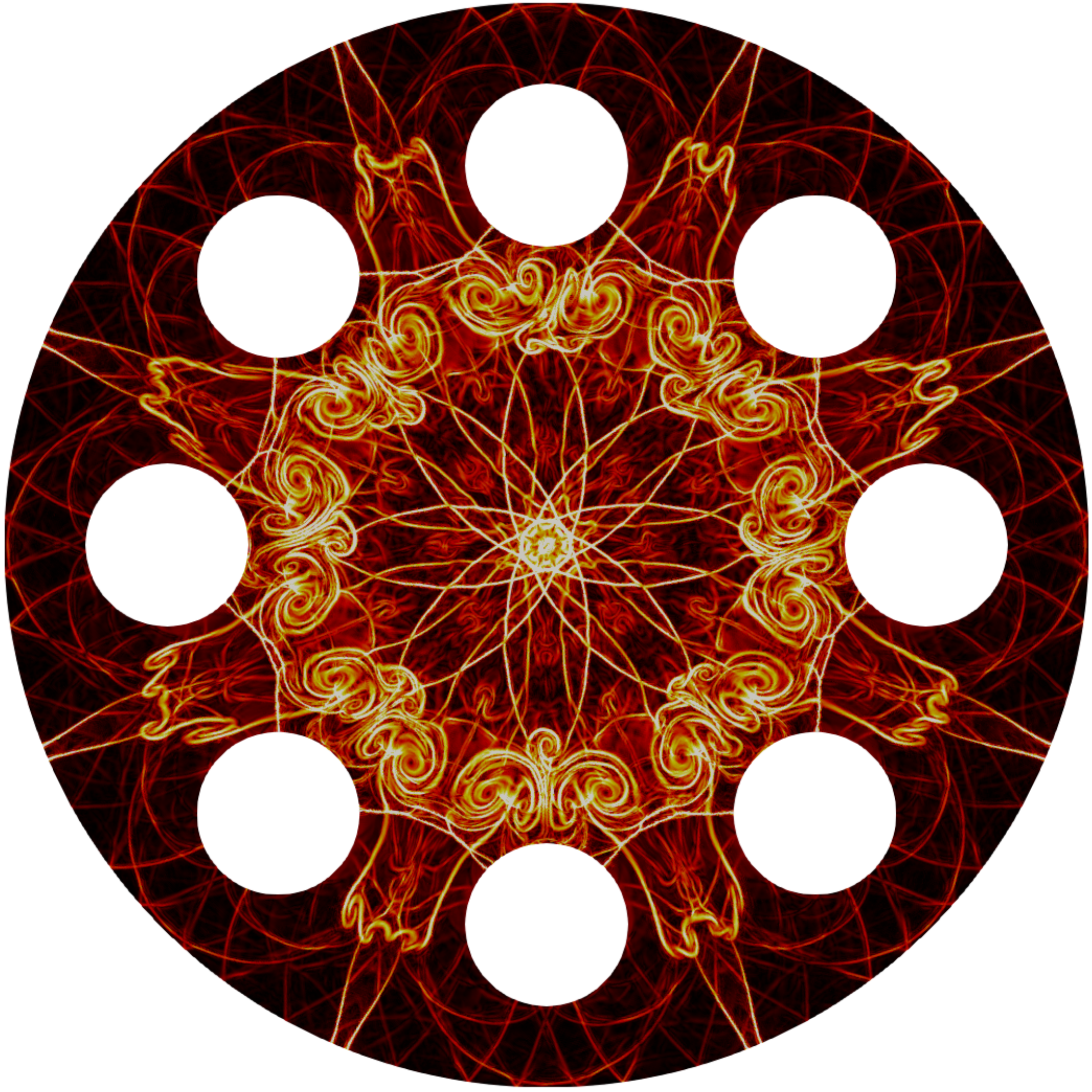}
    \caption{$t=2.85$}
  \end{subfigure}

  \vspace{0.5cm}  % 行间距
  
  % 第3行
  \begin{subfigure}{0.32\textwidth}
    \centering
    \includegraphics[width=\linewidth]{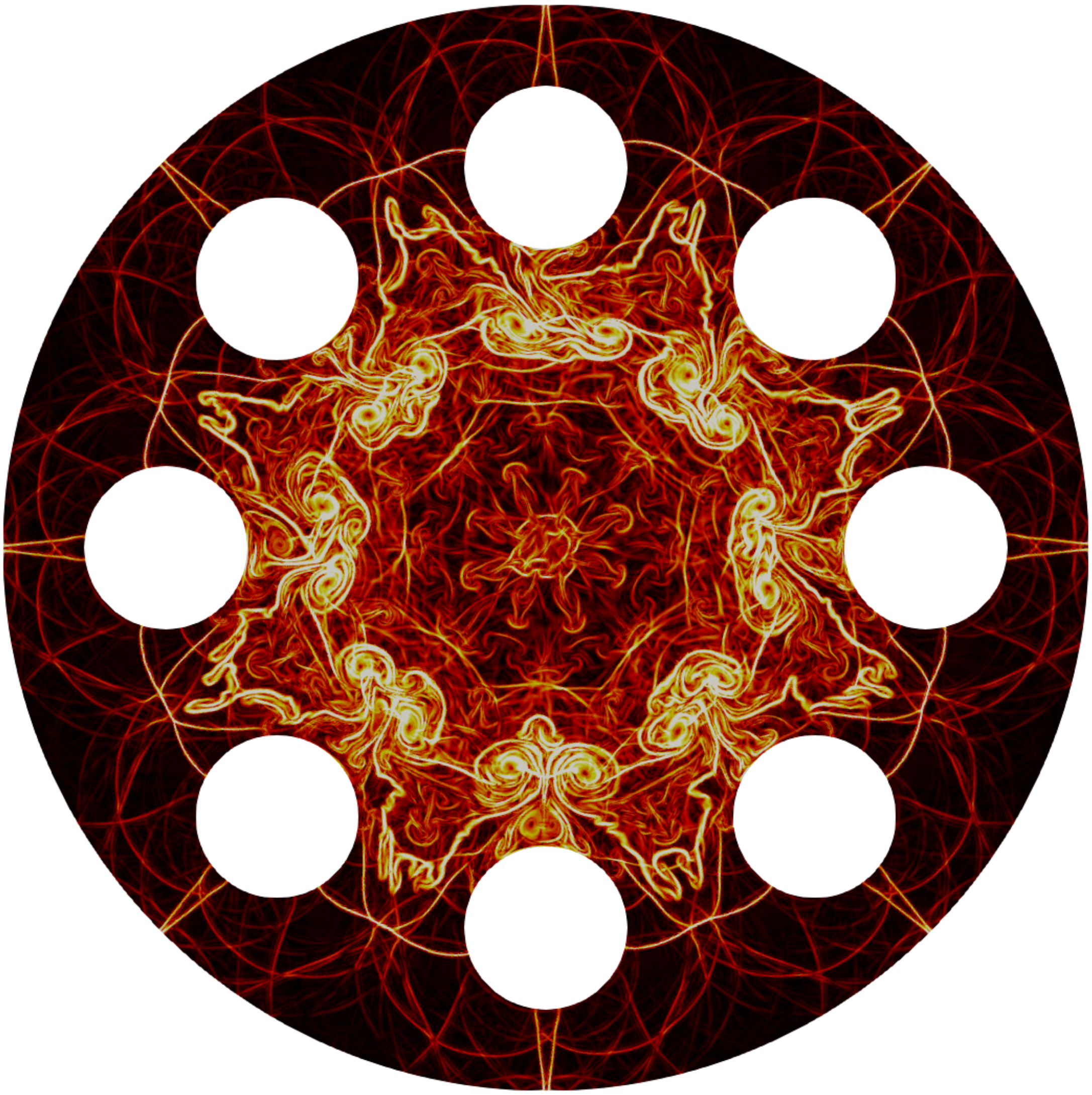}
    \caption{$t=3.45$}
  \end{subfigure}
  %\hspace{0.02\textwidth}
  \begin{subfigure}{0.32\textwidth}
    \centering
    \includegraphics[width=\linewidth]{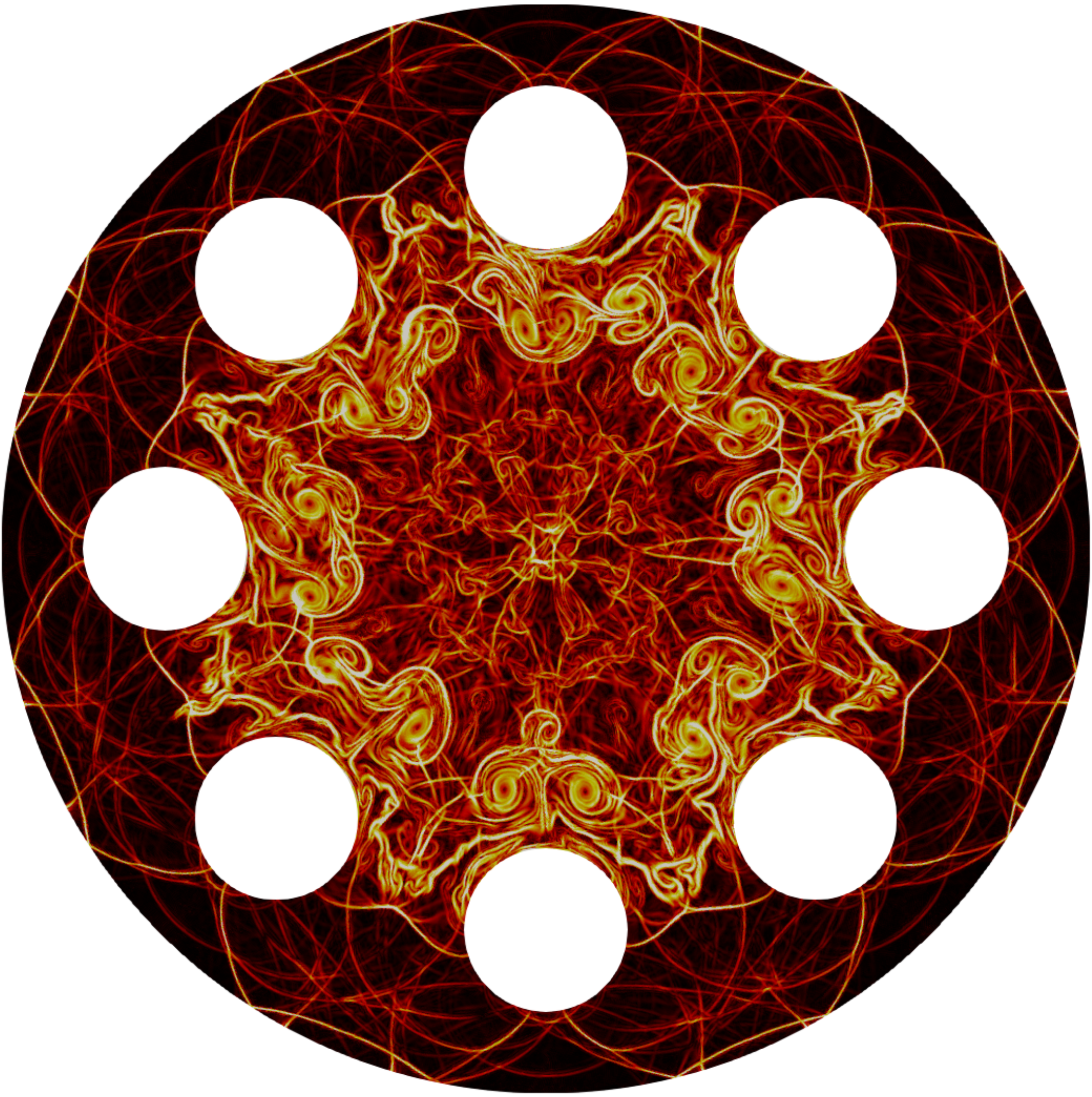}
    \caption{$t=3.9$}
  \end{subfigure}
  %\hspace{0.02\textwidth}
  \begin{subfigure}{0.32\textwidth}
    \centering
    \includegraphics[width=\linewidth]{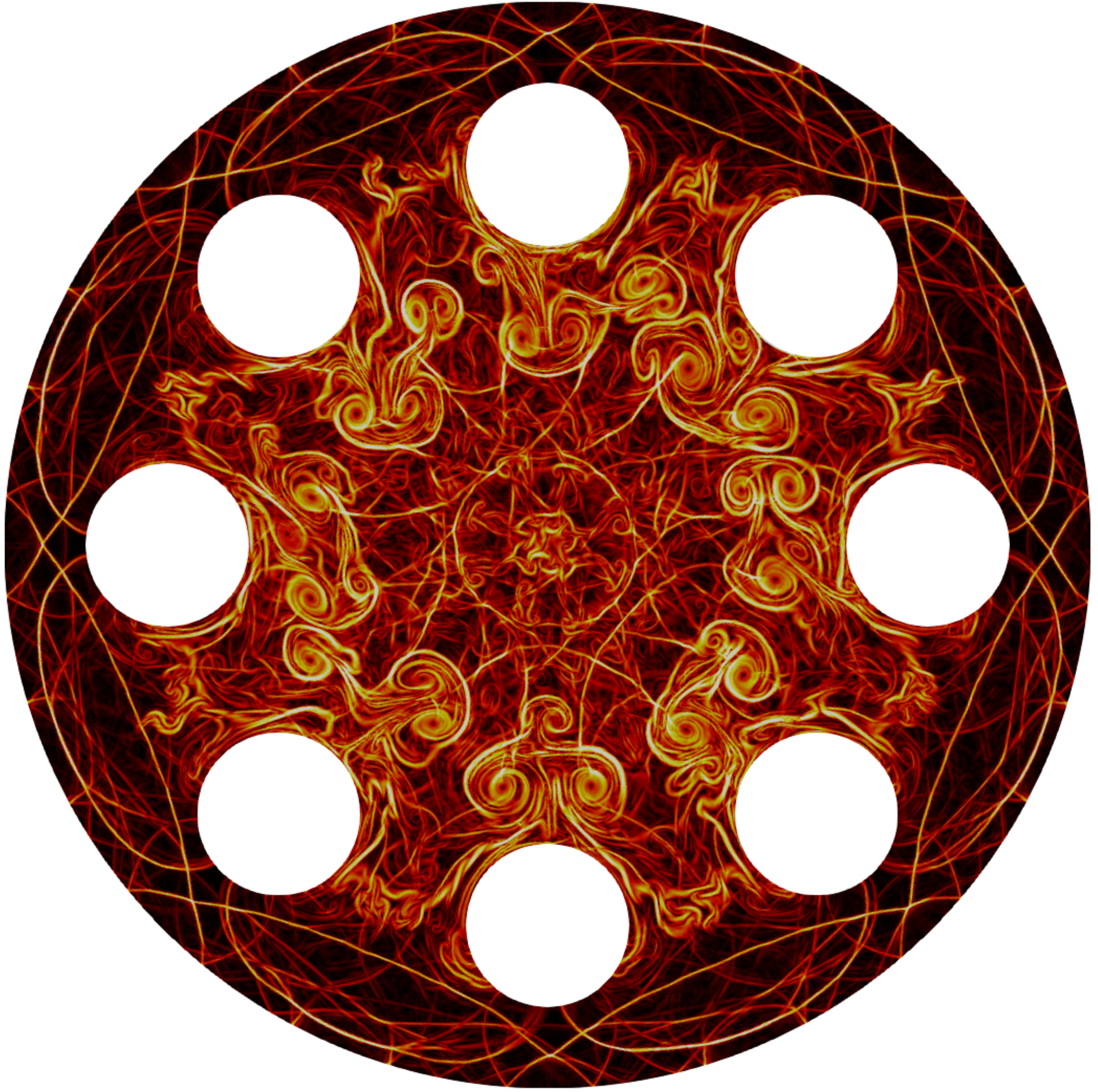}
    \caption{$t=4.2$}
  \end{subfigure}
  
  \caption{\textit{Explosion in a Domain with Multiple  Cylinders:} Schlieren diagram of the density at different time.}
  \label{fig:explosion}
\end{figure}
%%%%%%%%%%%%%%%%%%%%%%%%%%%%%%%%%%%%%%%%%%%%%%%%%%%%
%%%%%%%%%%%%%%%%%%%%%%%%%%%%%%%%%%%%%%%%%%%%%%%%%%%%
%%%%%%%%%%%%%%%%%%%%%%%%%%%%%%%%%%%%%%%%%%%%%%%%%%%%
%%%%%%%%%%%%%%%%%%%%%%%%%%%%%%%%%%%%%%%%%%%%%%%%%%%%
\section{Concluding Remarks}
\label{sec:conclude}
In this work, we have developed an entropy-stable high-order nodal discontinuous Galerkin method for the two-dimensional compressible Euler equations on general curvilinear meshes, equipped with an improved
oscillation-eliminating (OE) procedure.
The formulation is built upon an entropy-stable DG spectral element method
(DGSEM) that satisfies the summation-by-parts (SBP) property and discrete metric identities,  entropy stability at the semidiscrete level on curvilinear grids. 

To effectively control nonphysical oscillations near strong
discontinuities, we incorporated a modified oscillation-eliminating discontinuous Galerkin (OEDG) procedure into the entropy-stable DG framework.
Several improvements to the original OEDG method were introduced.
First, an OE scaling factor was incorporated to regulate the strength of
numerical damping, providing flexibility in balancing oscillation control
and numerical dissipation.
Second, we observed that the OE damping coefficients naturally serve as
effective shock indicators.
By localizing the OE procedure to troubled cells, the overall computational
cost was significantly reduced.
Third, the OE procedure was reformulated using projection operators,
allowing for a systematic extension to curvilinear meshes, where local
orthogonal modal bases are not readily available.

A comprehensive set of numerical experiments was presented to assess the accuracy, robustness, and effectiveness of the proposed method.
The isentropic Euler vortex problem, which admits an exact analytical solution, was computed on curvilinear meshes, demonstrating that the method retains high-order accuracy and preserves entropy stability in the presence
of geometric deformation.
Two-dimensional Riemann problems were used to verify the ability of the method to accurately capture shocks, contact discontinuities, and complex wave interactions, with results consistent with those reported in the
literature.
The double Mach reflection problem further confirmed the robustness of the scheme in resolving strong shock interactions and intricate flow structures
in the downstream region. We also compare the computation cost with different shock indicator threshold in double Mach reflection problem and found that the shock indicator can indeed improve the efficiency.

Particular emphasis was placed on challenging problems involving complex geometries.
In the simulation of supersonic flow around a two-dimensional circular cylinder, the proposed method successfully resolved a wide range of
flow features, including bow shocks, oblique shocks, contact discontinuities, fishtail shocks, and Kelvin--Helmholtz instabilities.
A detailed comparison of different OE scaling factors revealed the expected trade-off between numerical dissipation and resolution of fine-scale flow
features: larger scaling factors suppress small-scale instabilities, while smaller values preserve detailed structures at the expense of increased oscillations near strong shocks.
Finally, the explosion problem in a domain with multiple cylinders further demonstrated the capability of the method to robustly and accurately capture complex shock dynamics in highly nontrivial curvilinear geometries.

Overall, the numerical results indicate that the proposed entropy-stable DGSEM combined with the modified OEDG procedure provides a robust and accurate framework for simulating compressible flows with strong discontinuities on general curvilinear meshes.
The flexibility of the OE mechanism, together with entropy stability and high-order accuracy, makes the method a promising tool for challenging compressible flow applications involving complex geometries and multiscale flow phenomena. Future work will focus on extending the present entropy-stable OEDG framework to more complex systems, including the compressible Navier–Stokes equations and magnetohydrodynamics. Such extensions would further demonstrate the versatility of the proposed approach for challenging flow problems on general curvilinear meshes.

\section*{Appendix: Proof of Lemma \ref{lma:1}}
% \label{sec:appendix 1}
\begin{proof}
    For notational simplicity, the element superscript $(\cdot)^e$ is omitted
throughout this appendix.

The conservation property follows directly by summing the semidiscrete scheme
\eqref{eq_esdgsem} over all nodal indices $p_1$ and $p_2$.
Using the symmetry of the numerical fluxes and cancellation of interior
contributions across element interfaces, one obtains exact conservation of all
conserved variables \eqref{eq_cc}.

The proof of entropy stability \eqref{eq_ss} is more involved, particularly in the presence of
curvilinear coordinates.
We multiply the semidiscrete scheme \eqref{eq_esdgsem} by the entropy variables
$(V_k)_{p_1,p_2}$ and sum over all nodal indices and equation components.
After regrouping terms, this yields
\begin{equation}
\label{ecec}
    \begin{aligned}
0=&\sum_{p_1=0}^N\sum_{p_2=0}^N\sum_{k=1}^4 w_{p_1}w_{p_2}\mathcal{J}_{p_1,p_2}\,(\partial_tU_k)_{p_1,p_2}(V_k)_{p_1,p_2}\\
&+
2\sum_{p_1=0}^N\sum_{p_2=0}^N\sum_{i=0}^N\sum_{k=1}^4w_{p_1}w_{p_2} D_{p_1,i}\,(\tilde{f}_k^{\#})_{(i,p_1),p_2}(V_k)_{p_1,p_2}\\
&+
2\sum_{p_1=0}^N\sum_{p_2=0}^N\sum_{j=0}^N\sum_{k=1}^4w_{p_1}w_{p_2} D_{p_2,j}\,(\tilde{g}_k^{\#})_{p_1,(j,p_2)}(V_k)_{p_1,p_2}\\
&-\sum_{p_1=0}^N\sum_{k=1}^4 w_{p_1}\Big(
-(\Tilde{g}_k-\Tilde{g}_k^*)_{p_1,0}(V_k)_{p_1,0}
+(\Tilde{g}_k-\Tilde{g}_k^*)_{p_1,N}(V_k)_{p_1,N}\Big)\\
&
-\sum_{p_2=0}^N\sum_{k=1}^4
w_{p_2}\Big(
-(\Tilde{f}_k-\Tilde{f}_k^*)_{0,p_2}(V_k)_{0,p_2}
+(\Tilde{f}_k-\Tilde{f}_k^*)_{N,p_2}(V_k)_{N,p_2}\Big).
\end{aligned}
\end{equation}

For convenience, we decompose the above expression into four terms,
\begin{align*}
I = &\; \sum_{p_1=0}^N\sum_{p_2=0}^N\sum_{k=1}^4 w_{p_1}w_{p_2}\mathcal{J}_{p_1,p_2}\,(\partial_tU_k)_{p_1,p_2}(V_k)_{p_1,p_2},\\
II = &\; 
2\sum_{p_1=0}^N\sum_{p_2=0}^N\sum_{i=0}^N\sum_{k=1}^4w_{p_1}w_{p_2} D_{p_1,i}\,(\tilde{f}_k^{\#})_{(i,p_1),p_2}(V_k)_{p_1,p_2}\\
&\;-\sum_{p_2=0}^N\sum_{k=1}^4
w_{p_2}\Big(
-(\Tilde{f}_k)_{0,p_2}(V_k)_{0,p_2}
+(\Tilde{f}_k)_{N,p_2}(V_k)_{N,p_2}\Big),\\
III = &\; 
2\sum_{p_1=0}^N\sum_{p_2=0}^N\sum_{j=0}^N\sum_{k=1}^4w_{p_1}w_{p_2} D_{p_2,j}\,(\tilde{g}_k^{\#})_{p_1,(j,p_2)}(V_k)_{p_1,p_2}\\
&\;
-\sum_{p_1=0}^N\sum_{k=1}^4 w_{p_1}\Big(
-(\Tilde{g}_k)_{p_1,0}(V_k)_{p_1,0}
+(\Tilde{g}_k)_{p_1,N}(V_k)_{p_1,N}\Big),\\
IV = &\;
\sum_{p_1=0}^N\sum_{k=1}^4 w_{p_1}\Big(
-(\Tilde{g}_k^*)_{p_1,0}(V_k)_{p_1,0}
+(\Tilde{g}_k^*)_{p_1,N}(V_k)_{p_1,N}\Big)\\
&\;+\sum_{p_2=0}^N\sum_{k=1}^4
w_{p_2}\Big(
-(\Tilde{f}_k^*)_{0,p_2}(V_k)_{0,p_2}
+(\Tilde{f}_k^*)_{N,p_2}(V_k)_{N,p_2}\Big).
\end{align*}
With this notation, \eqref{ecec} can be compactly written as
\[
I + II + III + IV = 0.
\]

We now simplify each term in turn.
By definition of the entropy function $\eta(\bm U)$, we have
\[
\partial_t \eta
=
\sum_{k=1}^4 \frac{\partial \eta}{\partial U_k}\,\partial_t U_k
=
\sum_{k=1}^4 V_k\,\partial_t U_k.
\]
Consequently, the first term can be written as
\begin{align}
\label{cell-e}
I =
\sum_{p_1=0}^N\sum_{p_2=0}^N
w_{p_1}w_{p_2}\mathcal{J}_{p_1,p_2}\,
(\partial_t\eta)_{p_1,p_2},
\end{align}
which represents the discrete approximation of the time derivative of the
cell entropy $\int_{\Omega_e} \eta \, dx$ using the LGL quadrature rule.

Next, we analyze the second term $II$.
We fix the indices $k$ and $p_2$ and consider the summation over $p_1$ and $i$.
Define
\begin{align*}
II_{p_2,k}:= &\;
2\sum_{p_1=0}^N\sum_{i=0}^N w_{p_1}D_{p_1,i}\,
(\tilde{f}_k^{\#})_{(i,p_1),p_2}(V_k)_{p_1,p_2}
-\Big(
-(\Tilde{f}_k)_{0,p_2}(V_k)_{0,p_2}
+(\Tilde{f}_k)_{N,p_2}(V_k)_{N,p_2}
\Big).
\end{align*}

Using the symmetry of the entropy-conservative flux,
\[
(\tilde{f}_k^{\#})_{(i,p_1),p_2}
=
(\tilde{f}_k^{\#})_{(p_1,i),p_2},
\]
we obtain
\begin{align*}
II_{p_2,k}
=&\;
\sum_{p_1=0}^N\sum_{i=0}^N
w_{p_1}D_{p_1,i}\,
(\tilde{f}_k^{\#})_{(i,p_1),p_2}(V_k)_{p_1,p_2}
+
\sum_{p_1=0}^N\sum_{i=0}^N
w_iD_{i,p_1}\,
(\tilde{f}_k^{\#})_{(i,p_1),p_2}(V_k)_{i,p_2}\\
&\;
-\Big(
-(\Tilde{f}_k)_{0,p_2}(V_k)_{0,p_2}
+(\Tilde{f}_k)_{N,p_2}(V_k)_{N,p_2}
\Big)\\
=&\;
\sum_{p_1=0}^N\sum_{i=0}^N
w_{p_1}D_{p_1,i}\,
(\tilde{f}_k^{\#})_{(i,p_1),p_2}
\Big(
(V_k)_{p_1,p_2}-(V_k)_{i,p_2}
\Big),
\end{align*}
where in the last step we used the SBP property \eqref{eq_sbp_1} together with the
consistency of the numerical flux,
\[
(\tilde{f}_k^{\#})_{(j,j),p_2}
=
(\tilde{f}_k)_{j,p_2},
\qquad j=0,N,
\]
to cancel the boundary contributions.

Summing over all components $k=1,\dots,4$ and invoking the defining property of
the entropy-conservative fluxes \eqref{ecx}--\eqref{ec}, we obtain
\begin{align*}
\sum_{k=1}^4 II_{p_2,k}
=&\;
\sum_{p_1=0}^N\sum_{i=0}^N
w_{p_1}D_{p_1,i}
\Big(
\{\!\{y_{\eta}\}\!\}_{(i,p_1),p_2}
(\psi^f_{p_1,p_2}-\psi^f_{i,p_2})
-
\{\!\{x_{\eta}\}\!\}_{(i,p_1),p_2}
(\psi^g_{p_1,p_2}-\psi^g_{i,p_2})
\Big).
\end{align*}

Using again the symmetry of the averaged metric terms,
\[
\{\!\{y_{\eta}\}\!\}_{(i,p_1),p_2}
=
\{\!\{y_{\eta}\}\!\}_{(p_1,i),p_2},
\qquad
\{\!\{x_{\eta}\}\!\}_{(i,p_1),p_2}
=
\{\!\{x_{\eta}\}\!\}_{(p_1,i),p_2},
\]
we have 
\begin{align*}
\sum_{k=1}^4 II_{p_2,k}
   = &\; \sum_{p_1=0}^N\sum_{i=0}^N(w_{p_1}D_{p_1,i}-w_{i}D_{i,p_1})\Big(\{\!\{y_{\eta}\}\!\}_{(i,p_1),p_2}
\psi^f_{p_1,p_2}
-
\{\!\{x_{\eta}\}\!\}_{(i,p_1),p_2}\psi^g_{p_1,p_2}\Big)\\
   = &\; 2\sum_{p_1=0}^N\sum_{i=0}^Nw_{p_1}D_{p_1,i}\Big(\{\!\{y_{\eta}\}\!\}_{(i,p_1),p_2}
\psi^f_{p_1,p_2}
-
\{\!\{x_{\eta}\}\!\}_{(i,p_1),p_2}\psi^g_{p_1,p_2}\Big)\\
&\; - 
\underbrace{\Big((y_{\eta})_{N,p_2}
\psi^f_{N,p_2}
-(x_{\eta})_{N,p_2}\psi^g_{N,p_2}
\Big)}_{=\Tilde{\psi}^f_{N, p_2}}
+
\underbrace{\Big((y_{\eta})_{0,p_2}
\psi^f_{0,p_2}
-(x_{\eta})_{0,p_2}\psi^g_{0,p_2}
\Big)}_{=\Tilde{\psi}^f_{0, p_2}}\\
   = &\; \sum_{p_1=0}^N\sum_{i=0}^Nw_{p_1}D_{p_1,i}\Big((y_{\eta})_{i,p_2}
\psi^f_{p_1,p_2}
-
(x_{\eta})_{i,p_2}\psi^g_{p_1,p_2}\Big) - \Tilde{\psi}^f_{N, p_2} + \Tilde{\psi}^f_{0, p_2}
\end{align*}
where we used the SBP property \eqref{eq_sbp_1} in the second equality, and 
the summation identity \eqref{sum-1} in the last equality.
This implies that 
\begin{align}
    II = 
\sum_{p_2=0}^Nw_{p_2}  \sum_{k=1}^4 II_{p_2,k}
=&\; 
\sum_{p_2, p_1, i=0}^Nw_{p_2} w_{p_1}D_{p_1,i}\Big((y_{\eta})_{i,p_2}
\psi^f_{p_1,p_2}-
(x_{\eta})_{i,p_2}\psi^g_{p_1,p_2}\Big) \nonumber\\
&\;
+
\sum_{p_2=0}^Nw_{p_2}\Big(-\Tilde{\psi}^f_{N, p_2} + \Tilde{\psi}^f_{0, p_2}\Big).
\end{align}
Similar, we can simplify the third term to be 
\begin{align}
    III = &\;
\sum_{p_1=0}^N\sum_{p_2=0}^N\sum_{j=0}^Nw_{p_1} w_{p_2}D_{p_2,j}\Big(-(y_{\xi})_{p_1,j}
\psi^f_{p_1,p_2}+
(x_{\xi})_{p_1,j}\psi^g_{p_1,p_2}\Big) 
\nonumber\\
&\;+
\sum_{p_1=0}^Nw_{p_1}\Big(-\Tilde{\psi}^g_{p_1, N} + \Tilde{\psi}^g_{p_1, 0}\Big).
\end{align}
% where 
% \[
% \Tilde{\psi}^g_{p_1, j}
% := -(y_{\xi})_{p_1,j}
% \psi^f_{p_1,j}
% +
% (x_{\xi})_{p_1,j}\psi^g_{p_1,j}, \text{ for } j=0, N.
% \]
Next, using the metric identity \eqref{eq_metric_id_dis}, there holds 
\[
\sum_{p_1=0}^N\sum_{p_2=0}^N\sum_{j=0}^Nw_{p_1} w_{p_2}D_{p_2,j}(x_{\xi})_{p_1,j}
\psi^g_{p_1,p_2}
= 
\sum_{p_2=0}^N\sum_{p_1=0}^N\sum_{i=0}^Nw_{p_2} w_{p_1}D_{p_1,i}(x_{\eta})_{i,p_2}
\psi^g_{p_1,p_2}
\]
and 
\[
\sum_{p_1=0}^N\sum_{p_2=0}^N\sum_{j=0}^Nw_{p_1} w_{p_2}D_{p_2,j}(y_{\xi})_{p_1,j}
\psi^f_{p_1,p_2}
= 
\sum_{p_2=0}^N\sum_{p_1=0}^N\sum_{i=0}^Nw_{p_2} w_{p_1}D_{p_1,i}(y_{\eta})_{i,p_2}
\psi^f_{p_1,p_2}.
\]
Hence the volumn integration terms in $II+III$ cancels out, and we are left with boundary contributions:
\begin{align}
    \label{ii}
    II+III = 
    \sum_{p_2=0}^Nw_{p_2}\Big(-\Tilde{\psi}^f_{N, p_2} + \Tilde{\psi}^f_{0, p_2}\Big)+
\sum_{p_1=0}^Nw_{p_1}\Big(-\Tilde{\psi}^g_{p_1, N} + \Tilde{\psi}^g_{p_1, 0}\Big)
\end{align}
Finally, combining the above expression with  $IV$, we get 
\begin{align*}
    II+III+IV = &\;
\sum_{p_1=0}^N\sum_{k=1}^4 w_{p_1}\Big(
-(\Tilde{g}_k^*)_{p_1,0}(V_k)_{p_1,0}
+ \Tilde{\psi}^g_{p_1, 0}\Big) 
+ 
\sum_{p_1=0}^N\sum_{k=1}^4 w_{p_1}\Big(
(\Tilde{g}_k^*)_{p_1,N}(V_k)_{p_1,N}
- \Tilde{\psi}^g_{p_1, N}
\Big)\\
+\sum_{p_2=0}^N&\sum_{k=1}^4
w_{p_2}\Big(
-(\Tilde{f}_k^*)_{0,p_2}(V_k)_{0,p_2}
+\Tilde{\psi}^f_{0, p_2}\Big)\;+\sum_{p_2=0}^N\sum_{k=1}^4
w_{p_2}\Big(
(\Tilde{f}_k^*)_{N,p_2}(V_k)_{N,p_2}
-\Tilde{\psi}^f_{N, p_2}\Big)
% \\
% = &\; 
% \sum_{k=1}^4\left\langle (\Tilde{f}_k^{e,*},\Tilde{g}_k^{e,*})\cdot \hat{\bm n}, V_k \right\rangle_{\domgr,w}
% -
% \left\langle (\Tilde{\psi}^f, \Tilde{\psi}^g)\cdot\hat{\bm n}, 1 \right\rangle_{\domgr,w},
\end{align*}
Combining the above relation with \eqref{cell-e}, we conclude the proof of \eqref{eq_ss}.
\end{proof}

%%%%%%%%%%%%%%%%%%%%%%%%%%%%%%%%%%%%%%%%%%%%%%%%%%%%
%%%%%%%%%%%%%%%%%%%%%%%%%%%%%%%%%%%%%%%%%%%%%%%%%%%%
\section*{Acknowledgments}
Our numerical implementation is based on a DGSEM framework developed by Will Pazner within the MFEM library. We thank Will Pazner for kindly sharing his code with us and for valuable discussions related to the code design. 

\bibliographystyle{plain}
\bibliography{refrence}
\end{document}